\documentclass[11pt]{amsart}

\usepackage{amsmath,amssymb,amsfonts,amsthm,graphicx}
\usepackage[all]{xy}

\usepackage{pinlabel}

\usepackage[usenames,dvipsnames]{color}

\usepackage[usenames,dvipsnames]{color}

\newtheorem{dummy}{dummy}[section]
\newtheorem{lemma}[dummy]{Lemma}
\newtheorem{theorem}[dummy]{Theorem}

\newtheorem{proposition}[dummy]{Proposition}

\theoremstyle{definition}
\newtheorem{definition}[dummy]{Definition}

\newtheorem{example}[dummy]{Example}
\newtheorem{remark}[dummy]{Remark}

\newtheorem{question}[dummy]{Question}
\newtheorem{construction}[dummy]{Construction}

\numberwithin{equation}{section}

\newcommand{\R}{\mathbb {R}}

\newcommand{\C}{\mathbb {C}}
\newcommand{\Z}{\mathbb {Z}}
\newcommand{\F}{\mathbb {F}}

\newcommand{\aug}{\mathit{Aug}}

\newcommand{\alg}{\mathcal{A}}
\newcommand{\algCW}{\mathcal{A}^{\mathit{CW}}}

\newcommand{\rk}{\mathit{rk}}
\newcommand{\coeff}{K}
\def\fig#1{\raisebox{-2.2ex}{\includegraphics[height=5.2ex]{#1}}}

\begin{document}

\title[Skein relations for Legendrian surfaces]{Skein relations for invariants of Legendrian surfaces}

\author{Dan Rutherford}
\address{Ball State Unversity}

\author{Michael Sullivan}
\address{University of Massachusetts Amherst}

\begin{abstract}  
For Legendrian surfaces in $1$-jet spaces, 
we consider augmentation number invariants that together generalize the ruling polynomial of $1$-dimensional Legendrian knots and establish skein relations for these invariants.  The main skein relation involves four ways of resolving a Legendrian with a double point: perturbing the double point to a contractible Reeb chord in two different ways, and removing the double point via the two Lagrangian surgeries.  We apply the skein relations in examples, and use them to provide an alternate proof of the relation between augmentations of Legendrian 2-weaves and face colorings due to Casals, Murphy, and Sackel.  
\end{abstract}

\maketitle

{\small \tableofcontents}

\section{Introduction}

The Legendrian contact homology (abbrv. LCH) dg-algebra (abbrv. DGA) is a much applied invariant of Legendrian submanifolds defined via $J$-holomorphic curve theory.  In many cases of interest, the homology of the LCH dg-algebra is infinite dimensional, even in individual grading components.  One method for extracting numerical invariants from the LCH dg-algebra is to make (normalized) counts of homomorphisms to a simpler dg-algebra such as a finite field.  In this article, we establish skein relations for augmentation number invariants of Legendrian surfaces defined via this procedure, and apply these skein relations to make computations.

To state our main results, we consider a Legendrian submanifold $\Lambda$ in standard contact $\R^{2n+1}$, or more generally in a $1$-jet space, with vanishing Maslov class and equipped with a choice of $\Z$-valued Maslov potential, i.e. {\it $\Z$-graded Legendrian submanifolds}. 
  (See Section \ref{sec:LCHdgalgebra} for definitions.)  We work with the version of the LCH dg-algebra $(\alg(\Lambda), \partial)$ whose underlying algebra $\alg(\Lambda)$ is the free associative algebra over $\Z[H_1(\Lambda)]$ generated by the Reeb chords of $\Lambda$ as originally constructed in \cite{EES05a,EES05c,EES07}.   With $q$ a prime power,
we let $\mathbb{F}_q$ denote the finite field of order $q$.  An $\mathbb{F}_q$-valued {\it augmentation}
is a dg-algebra (here algebra means $\Z$-algebra)
 homomorphism $\epsilon: (\alg(\Lambda), \partial) \rightarrow (\mathbb{F}_q,0)$, i.e. a unital, grading preserving ring homomorphism satisfying $\epsilon \circ \partial =0$.  Denote by $V(\Lambda, \mathbb{F}_q)$ the set of all $\mathbb{F}_q$-valued augmentations, called the {\it augmentation variety}.
Following
\footnote{In \cite{NgSab} attention is restricted to the case $q=2$, so that the $(q-1)^{\mathit{rank}(H_0(\Lambda))-\mathit{rank}(H_1(\Lambda))}$ term is equal to $1$ and is not present.  The choice of normalization used in defining $\mathit{Aug}_\Lambda(q)$ here differs slightly from the ones used in \cite{HR} and \cite{NRSS}.}  Ng and Sabloff \cite{NgSab}, Legendrian isotopy invariant {\it augmentation numbers}, $\mathit{Aug}_\Lambda(q)$, can be defined by 
\begin{equation} \label{eq:1dimskeinrel}
\mathit{Aug}_\Lambda(q) = (q-1)^{\mathit{rank}(H_0(\Lambda))-\mathit{rank}(H_1(\Lambda))}q^{-\chi^*(\alg(\Lambda))/2}\cdot |V(\Lambda, \mathbb{F}_q)|
\end{equation}
where 
$\chi^*(\alg(\Lambda))$ 
is a {\it shifted Euler characteristic} defined via the degree distribution of Reeb chords.  (See Section \ref{sec:DGAaug} for the detailed definition.)

In the case when $\Lambda \subset \R^3$ is a $1$-dimensional Legendrian knot, the augmentation numbers are shown in \cite{NgSab,HR} to arise as specializations of the ($0$-graded) ruling polynomial $R^0_\Lambda(z) \in \Z[z^{\pm1}]$ via
\begin{equation} \label{eq:augruling}
(q-1)^{-\mathit{rank}(H_0(\Lambda))} \mathit{Aug}_\Lambda(q) = R^0_\Lambda(z)|_{z = q^{1/2}-q^{-1/2}}.
\end{equation}  
The ruling polynomials $R^0_\Lambda(z)$ are defined in terms of counts of combinatorial decompositions of the front projection of $\Lambda$ called ($0$-graded) {\it normal rulings}, see \cite{ChP,F,R}, and are moreover known to satisfy a skein relation  
\begin{align}
	\labellist
	\small
	\pinlabel $r$ [l] at 130 96
	\pinlabel $s$ [l] at 130 16
	\endlabellist
	\raisebox{-.6cm}{\includegraphics[scale=.4]{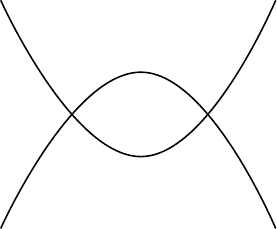}}  \quad - 
	\labellist
	\small
	\pinlabel $r$ [l] at 130 96
	\pinlabel $s$ [l] at 130 16
	\endlabellist
	\quad \raisebox{-.6cm}{\includegraphics[scale=.4]{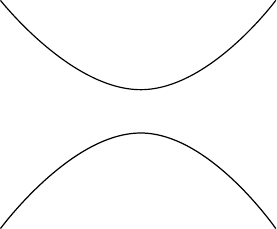}} 
	\quad =  \quad z \left[
	\labellist
	\small
	\pinlabel $s$ [l] at 130 96
	\pinlabel $s$ [l] at 130 16
	\endlabellist 
	\delta_{r,s}\raisebox{-.6cm}{\includegraphics[scale=.4]{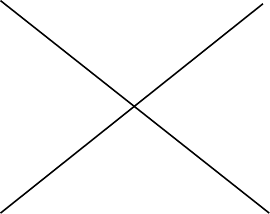}} \quad -  
	\labellist
	\small
	\pinlabel $s+1$ [l] at 130 96
	\pinlabel $s$ [l] at 130 16
	\endlabellist
	\quad \delta_{r,s+1}\raisebox{-.6cm}{\includegraphics[scale=.4]{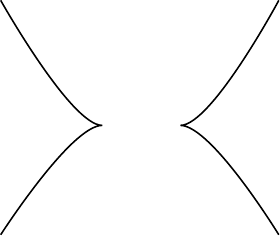}} \quad \quad \quad \right]  \label{eq:1dimSR}
\end{align}
involving four Legendrians whose front projections differ in a disk as pictured with $r,s \in \Z$ used to indicate the value of the Maslov potential, and $\delta_{i,j}$ the Kronecker delta.  
In fact, the skein relation (\ref{eq:1dimSR}) along with the vanishing of $R^0$ on stabilized Legendrian knots and a choice of normalization on the Legendrian unknot uniquely characterizes $R^0$ as a Legendrian isotopy invariant; see \cite{R, Ru2}.  

In this article we establish analogous skein relations for the augmentation numbers of Legendrian surfaces.  
We work with $\Z$-graded Legendrian surfaces in a $1$-jet space $J^1S$ with orientable base surface $S$.  The case of $S = \R^2$ is standard contact $\R^5$.  Recall that $J^1S = T^*S \times \R$ comes with a {\it front projection}, $\pi_{xz}: J^1S \rightarrow J^0S = S \times \R$, a {\it Lagrangian projection}, $\pi_{xy}:J^1S \rightarrow T^*S$, and a {\it base projection}, $\pi_{x}:J^1S \rightarrow S$.  These projections are reviewed in Section \ref{sec:Zgraded}. 

\medskip

\noindent{\bf Color convention.}  When depicting front or base projections of Legendrian surfaces, we use \color{red} red \color{black} to indicate cusp edges (aka front singularities of type $A_2$), and \color{blue} blue \color{black} to indicate crossing arcs (aka front singularities of type $A_1^2$).

\medskip

 The statement of the skein relations involves a {\bf shifted parity function}, $\eta$, defined by
\begin{equation}  \label{eq:parity}
\eta:\Z \rightarrow \{\pm 1\}, \quad \quad \eta(k) = \left\{ \begin{array}{cr} (-1)^k, & k \geq 0 \\\  (-1)^{k+1}, & k<0. \end{array}\right.
\end{equation}

\begin{theorem} \label{thm:main} For any prime power $q$, the $0$-graded augmentation number $\mathit{Aug}_\Lambda(q)$ is a Legendrian isotopy invariant of $\Z$-graded Legendrian surfaces in $1$-jet spaces that satisfies the following skein relations:
\begin{enumerate}
\item[(SR1)] We have 
\begin{align*}
\labellist
\small
\pinlabel $r$ [l] at 130 116
\pinlabel $s$ [l] at 130 36
\endlabellist
q^{\eta(s-r+1)/2} \raisebox{-1cm}{\includegraphics[scale=.4]{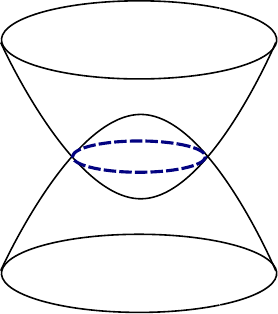}} & \quad - 
\labellist
\small
\pinlabel $r$ [l] at 130 116
\pinlabel $s$ [l] at 130 36
\endlabellist
 \quad q^{\eta(r-s-1)/2} \raisebox{-1cm}{\includegraphics[scale=.4]{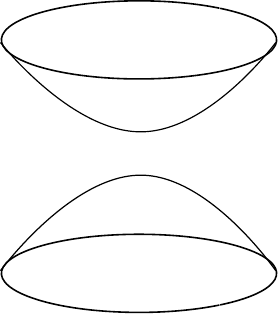}} \\
 & = \delta_{r,s+1} (q-1)^2 \left[
\labellist
\small
\pinlabel $s+1$ [l] at 130 116
\pinlabel $s$ [l] at 130 36
\endlabellist 
\raisebox{-1cm}{\includegraphics[scale=.4]{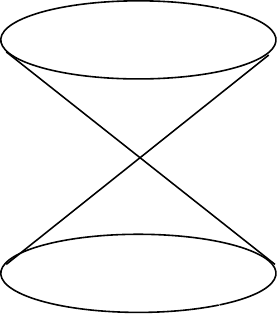}} \quad -  
\labellist
\small
\pinlabel $s+1$ [l] at 130 116
\pinlabel $s$ [l] at 130 36
\endlabellist
\quad \raisebox{-1cm}{\includegraphics[scale=.4]{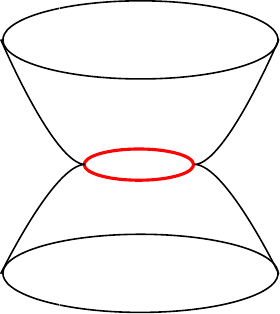}} \quad \quad \quad \right]
\end{align*}
pictured in the front projection where $r, s \in \Z$ denote the values of the Maslov potential.  Note that the first Legendrian pictured on the RHS has a cone point front singularity.

\item[(SR2)]  We have
\[
q^{\eta(r-s)/2} \,\raisebox{-1cm}{\includegraphics[scale=.5]{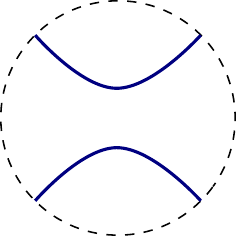}} \quad - 
\quad q^{\eta(s-r)/2} \, \raisebox{-1cm}{\includegraphics[scale=.5]{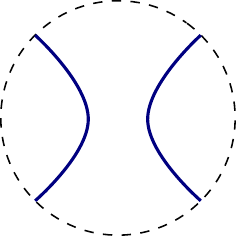}} = \delta_{r,s} (q-1)^2 \left[
\raisebox{-1cm}{\includegraphics[scale=.5]{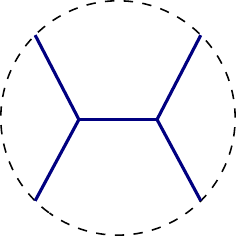}} \quad -  
\quad \raisebox{-1cm}{\includegraphics[scale=.5]{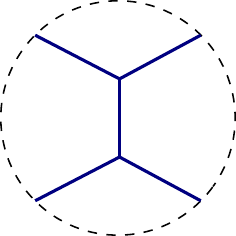}} \right]
\]
where the base projection of the crossing locus between two adjacent sheets is pictured; the trivalent vertices are projections of $D_4^-$ front singularities; and $r$ and $s$ denote the value of the Maslov potential on the two crossing sheets such that in the first term on the LHS the sheet with Maslov potential $r$ appears above (in the $z$-direction) the sheet with Maslov potential $s$ at the center of the pictured disk.

\item[(SR3)]  Suppose that $\Lambda_1, \Lambda_2 \subset J^1S$ are such that $\Lambda_2$ is obtained from $\Lambda_1$ via an ambient Legendrian $0$-surgery, i.e. the front projections are related by a modification within a disk as pictured:  
\[
\Lambda_1 \quad = \quad \raisebox{-1cm}{\includegraphics[scale=.5]{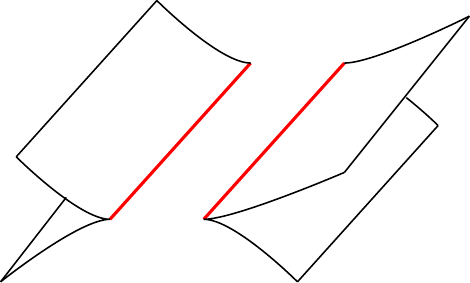}}, \quad \quad \quad \Lambda_2 \quad = \quad \raisebox{-1cm}{\includegraphics[scale=.5]{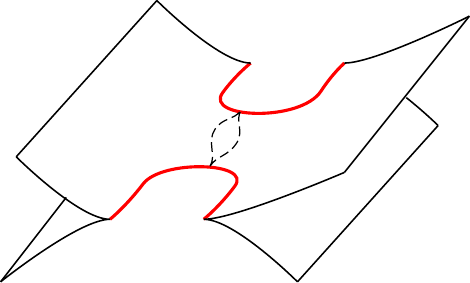}}
\]
Then,
\[
\aug_{\Lambda_2}(q) = \frac{q^{1/2}}{q-1} \aug_{\Lambda_1}(q) = (q^{1/2}-q^{-1/2})^{-1} \aug_{\Lambda_1}(q). 
\]

\item[(SR4)] The invariant vanishes on any Legendrian surface that appears as pictured within some disk in the front projection: 
\[
\raisebox{-1cm}{\includegraphics[scale=.5]{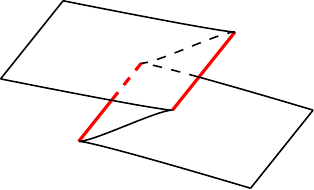}} \quad = \quad 0
\]

\end{enumerate}

Moreover, when $\Lambda_1, \Lambda_2 \subset J^1S$ are Legendrian surfaces with disjoint base projections we have 
\begin{equation} \label{eq:productrule}
\aug_{\Lambda_1 \sqcup \Lambda_2}(q) = \aug_{\Lambda_1}(q) \cdot \aug_{\Lambda_2}(q),
\end{equation}
 and the Legendrian unknot $U$ satisfies  
\begin{equation}  \label{eq:Uvalue}
\aug_U(q) = q^{1/2}-q^{-1/2}.
\end{equation}
\end{theorem}
Let us expand a bit on each of the relations (SR1)-(SR4) and discuss the contact geometric significance of the terms involved.  The skein relation (SR1) is a clear analogue of the $1$-dimensional relation (\ref{eq:1dimSR}), as  the local $2$-dimensional front projections in (SR1) are obtained from spinning the $1$-dimensional front projections from (\ref{eq:1dimSR}) around a vertical line of symmetry.  For an alternate perspective (that applies equally well in the $1$-dimensional case), the local front projection with a tangential double point
\begin{equation}  \label{eq:doublept}
\raisebox{-1cm}{\includegraphics[scale=.4]{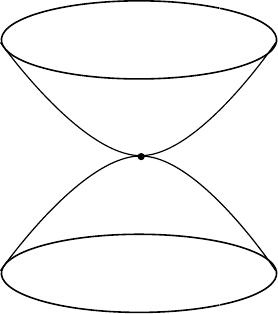}}
\end{equation}
lifts to a portion of an immersed Legendrian with a single double point.  The four terms in (SR1) correspond to four ways to resolve the double point as follows:  The two terms on the left hand side arise from perturbing the immersed Legendrian (without changing its topology) to become embedded with the double point replaced by a Reeb chord of arbitrarily small action (aka length).  Note that these two fronts are connected via a Legendrian regular homotopy whose only non-embedded moment occurs when passing through the singular Legendrian (\ref{eq:doublept}).  In particular, the Reeb chords visible in the left hand side of (SR1) are contractible in the sense of \cite{EHK}.  For both of these fronts, while the double point is removed from the Legendrian lift to $J^1S$, a double point persists in the Lagrangian projection.  In contrast, the Legendrians on the right side of (SR1) have the Lagrangian double point (and also the corresponding Legendrian Reeb chord) removed altogether, but change the topology of the Legendrian.  These two fronts are known as the two Lagrangian surgeries (due to Polterovich \cite{Pol}) on the singular Lagrangian/Legendrian (\ref{eq:doublept}).  See \cite[Theorem 6.3]{CMP19}.

Viewed from the front projection the (contractible) Reeb chord found in the terms on the left side of (SR1) appears as an index 2 or index 0 critical point of the {\it local difference function} associated to the pair of pictured sheets; this is the function $f_u(x)-f_l(x)$ defined along the base projection of the pictured disks where the equations $z=f_u(x)$ and $z=f_l(x)$ define the sheets of the front projection that contain the upper and lower endpoints of the Reeb chord.  The case of a contractible Reeb chord appearing in the front projection as an index $1$ critical point of $f_u-f_l$ appears in (SR2).  The Reeb chord appears at the center of the disks on the LHS, and, following \cite[Theorem 4.21]{CZ}, the Legendrians with $D_4^-$ front singularities on the RHS of (SR2) arise from the Lagrangian surgeries on the singular Legendrian arising from shrinking the Reeb chords on the left to length $0$.
In Section \ref{sec:SR2}, we provide a self contained deduction of (SR2) from (SR1) by applying Legendrian surface Reidemeister moves from \cite{Arnold, Goryunov} to illustrate Legendrian isotopies transforming the four fronts of (SR2) to four front projections related as in (SR1).  We include this argument, in part, to add more examples of non-trivial Legendrian surface isotopies constructed via Riedemeister moves into the literature.  See also \cite{CZ} for a host of useful Legendrian isotopies constructed via surface Reidemeister moves.

The $0$-surgery that appears in (SR3) is a particular case of the ambient Legendrian surgeries introduced by Dimitroglou Rizell in \cite{Rizell}; see also \cite{BST}.  The construction of \cite{Rizell} includes an exact Lagrangian cobordism from $\Lambda_1$ to $\Lambda_2$ consisting of a single index $1$ handle attachment.  

Turning to (SR4), we note that the portion of the front pictured is the product of an interval with  a zig-zag stabilization in the front projection of a $1$-dimensional Legendrian knot, cf. \cite{FT}.  In the case that a surface $\Lambda \subset J^1S$ has only one connected component, the presence of such a zig-zag after a Legendrian isotopy is equivalent to $\Lambda$ being a loose Legendrian surface in the sense of Murphy \cite{Murphy}.

\subsection{Computations with the skein relations}  
In Sections \ref{sec:firstex} and \ref{sec:2weaves}, we illustrate the use of the skein relations (SR1)-(SR4) to compute $\aug_\Lambda(q)$ for some classes of Legendrian surfaces including (i) the infinite class of surfaces introduced by Dmitroglou Rizell in \cite{Rizell1},  
and (ii) all surfaces in $\R^5$ arising as satellites of Legendrian $2$-weaves in $J^1S^2$ with the Legendrian unknot.  We state here the result in the case of (ii), referring to Section \ref{sec:2weaves} for more details.  

Recall that following Treumann-Zaslow \cite{TZ} an 
embedded graph with trivalent vertices,  $G \subset S^2$, specifies a two-sheeted Legendrian $\widetilde{\Lambda}_G \subset J^1S^1$ whose front projection has crossing arcs above $G$ with $D_4^-$-singularities above the trivalent vertices.  
Such surfaces are called {\it Legendrian $2$-weaves} in the terminology of \cite{CZ}.  Taking the Legendrian satellite with the standard Legendrian unknot then produces a Legendrian surface $\Lambda_G \subset \R^5$.    
In \cite{TZ}, it is shown that the chromatic polynomial $P(G^*,x)$ of the dual graph to $G$ (the values of this polynomial at positive integers count face colorings of $G$) can be recovered from the Legendrian isotopy type of $\widetilde{\Lambda}_G$.  This is achieved by relating face colorings to micro-local rank $1$ constructible sheaves with singular support determined by $\widetilde{\Lambda}_G$.    The work \cite{CMS} then illustrates the augmentation and sheaf correspondence \cite{NRSSZ} for $\Lambda_G$ by (i) combinatorially constructing from $G$ a dg-algebra $\mathcal{A}_G$ equivalent to the LCH dg-algebra of $\Lambda_G$ \cite{CJN, ScSh} and (ii) relating augmentations of $\mathcal{A}_G$ with face colorings.   See also \cite{Sackel} for a study of higher rank representations of a fully non-commutative version of $\mathcal{A}_G$. 

In Section \ref{sec:2weaves} we achieve the following augmentation number calculation, consistent with that of \cite{CMS}, using only the skein relations of Theorem  \ref{thm:main}.  

\begin{theorem}[cf. \cite{CMS}]  \label{thm:2weaves} For any such $G \subset S^2$ we have
	\begin{equation} \label{eq:2weaves}
		\aug_{\Lambda_G}(q) = \frac{(q-1)^{6}}{q^{3}(q+1)} \left[ \frac{q^{1/2}}{(q-1)^2} \right]^{Z(G)}  P(G^*, q+1)
	\end{equation}
	where $Z(G) = \chi(S^2 \setminus G)$ is 
	the Euler characteristic of the complement of $G$ and and $P(G^*,x)$ is the chromatic polynomial of the dual graph $G^*$.
\end{theorem}

\begin{remark}
	While \cite{CMS} restricts to the case where $G$ is connected, for our skein relation based proof to go through it is crucial to enlarge considerations to allow disconnected $G$.  We note also that the results of Sackel in \cite{CMS, Sackel} include a computation of the isomorphism type of the augmentation variety of $\Lambda_G$ and its higher rank analog over any field $\mathbb{F}$. 
\end{remark}

\subsection{Further questions}  Notice that from (\ref{eq:augruling}) it is clear that the augmentation numbers, viewed collectively as a function of $q$, generalize the ruling polynomial invariant of $1$-dimensional Legendrian knots to Legendrians of higher dimensions.  In the case of $1$-dimensional Legendrian knots in $\R^3 = J^1\R$ the skein relations analogous to (S1)-(S4) together with the value on the unknot uniquely characterize the ruling polynomial.  Indeed, \cite{R} provides an algorithm for reducing any Legendrian front projection to a combination of stabilized Legendrians and multiples of the Legendrian unknot via application of the skein relations.  

\begin{question}  Do the skein relations from Theorem \ref{thm:main} uniquely determine the augmentation numbers, $\mathit{Aug}_\Lambda(q)$, for Legendrian surfaces in $\R^5$? 
\end{question}

Recall that for a $1$-dimensional Legendrian knot $K \subset J^1\R$  the ruling polynomial is a Laurent polynomial in $z = q^{1/2}-q^{-1/2}$ with non-negative integer coefficients that arises as a sum over $0$-graded normal rulings of $K$ with each normal ruling contributing a weight of the form $z^{j}$, cf. \cite{R}.

\begin{question}  Do the augmentation numbers arise from a Laurent polynomial or rational function (perhaps belonging to a restricted subring of $\mathbb{Q}(q^{1/2})$) in $q^{1/2}$ evaluated at prime powers?  
\end{question}

\begin{question}  Can the $Aug_\Lambda(q)$ be computed from a weighted sum over normal rulings (appropriately defined) of the front projection of $\Lambda$?
\end{question}

In \cite{NRSS}, the augmentation numbers defined by normalized, naive counts of augmentations are given a more naturally Legendrian isotopy invariant interpretation as the homotopy cardinality of an $A_\infty$-category of augmentations, i.e. as a weighted count of moduli of augmentations.  See also \cite{GH} for the even graded case and \cite{Murray} for the case of higher rank augmentations.  Alternatively, one can take the homotopy cardinality of an equivalent (via \cite{NRSSZ}) dg-category of sheaves \cite{STZ}.  

\begin{question}
For Legendrian surfaces, how are the augmentation numbers $\mathit{Aug}_\Lambda(q)$ related to (homotopy) cardinalities of Legendrian invariant augmentation \cite{EL} and sheaf \cite{STZ} categories?  
\end{question}

\subsection{Methods and organization}  Establishing the skein relations of Theorem \ref{thm:main} relies on being able to compare augmentations of the LCH dg-algebras of Legendrian surfaces whose front projections differ locally in a disk.  
To do so, we employ the cellular LCH dg-algebra from \cite{RuSu1,RuSu2,RuSu25} suitably extended to $\Z[H_1(\Lambda)]$ coefficients \cite{RuSu5}.  The cellular DGA provides an explicit model for the LCH DGA based on a cellular decomposition of the base projection of $\Lambda$.  Portions of $\Lambda$ lying above closed cells correspond to sub-DGAs, and this effectively localizes the computation.  
For added flexibility in computing augmentation numbers, we work with a ``multi-curve'' variant of the LCH DGA for surfaces (see Section \ref{sec:multiDGA}) that is analogous to the ``multi-base pointed'' DGA for $1$-dimensional Legendrian knots \cite{NgR}.  

\medskip

The methods of this article rely on the following results whose proofs will be detailed in the article in preparation \cite{RuSu5}:
\begin{itemize}
	\item[(i)] the Legendrian isotopy invariance of the multi-curve version of the LCH dg-algebra for surfaces and its relation with the usual LCH DGA over $\Z[H_1(\Lambda)]$, see Definition \ref{def:multicurveDGA} and Proposition \ref{prop:multiDGA}, and 
	\item[(ii)] the isomorphism (Theorem \ref{thm:CWmulti}) between the multi-curve version of the cellular DGA, defined in Section \ref{sec:CWDGA}, and the multi-curve LCH dg-algebra.
\end{itemize}
The presentation of these topics in the current article is self contained with only the proofs of Proposition \ref{prop:multiDGA} and Theorem \ref{thm:CWmulti} referenced to \cite{RuSu5}.  The detailed proof of Theorem \ref{thm:CWmulti} involves integrating the multi-curves into the setup of \cite{RuSu1, RuSu2, RuSu25}, which is relatively straightforward, and a more delicate argument to extend the isomorphism from coefficients in $\mathbb{F}_2$ to coefficients in $\Z$.   In the case of augmentations to $\mathbb{F}_2$, which is sufficient for the results of this article in the case $q=2$, no results from \cite{RuSu5} are needed as only \cite{RuSu1, RuSu2, RuSu25} are required.  For fields of characteristic $2$, i.e., the cases $q= 2^\ell$ for $\ell \geq 1$, it is sufficient to work with $\mathbb{F}_2[H_1(\Lambda)]$ coefficients, and this simplifies the proof of Theorem \ref{thm:CWmulti} significantly by avoiding considerations involving signs.

\medskip

The remainder of the article is organized as follows.  Section \ref{sec:DGAaug} presents needed algebraic preliminaries on dg-algebras and their augmentations.  Section \ref{sec:LCHdgalgebra} discusses the multi-curve LCH dg-algebra 
and defines the augmentation numbers $\aug_\Lambda(q)$ in this setting.  A cellular computation of the multi-curve DGA then appears in Section \ref{sec:CWDGA}.   With these preliminaries out of the way, the next three sections provide the proof of Theorem \ref{thm:main}.  Section \ref{sec:5} establishes (SR1) with the details of the computation of the DGA of the term with the cone point deferred until the appendix (Section \ref{sec:Cone}).  Proposition \ref{prop:specialiso} may be of some independent interest as it provides isomorphisms between various specializations of the DGAs of the four Legendrians from (SR1).  Section \ref{sec:SR2} recalls from \cite{Arnold, Goryunov} local modifications of Legendrian surfaces (``Legendrian surface moves'') that can be realized by Legendrian isotopy, then applies them to establish (SR2) from (SR1).  Finally, (SR3), (SR4), and the remaining identities (\ref{eq:productrule}) and (\ref{eq:Uvalue}) are established in Section \ref{sec:7}.
The remaining two Sections \ref{sec:firstex} and \ref{sec:2weaves} present example computations of augmentation numbers using the skein relations; this includes the proof of Theorem \ref{thm:2weaves}.  These sections may be read independently of the rest of the article.

\subsection{Acknowledgements}  The first author received support from grant 429536 from the Simons Foundation.  The second author received support from grant SFI-MPS-TSM-00013380 from the Simons Foundation.

\section{DGAs and augmentations}  \label{sec:DGAaug}

The dg-algebras arising in the setting of Legendrian contact homology are equipped with explicit generating sets; have differentials that preserve a filtration by action; and have their stable tame isomorphism types as Legendrian isotopy invariants.  We review here this algebraic setting and define augmentation numbers for such DGAs.

\subsection{DGAs and augmentations}  \label{sec:DGAsubsec}

Let $R$ be a coefficient ring which we assume to have one of the following forms:
\begin{itemize}
\item a group ring $\Z[H_1(\Lambda)]$ where $\Lambda$ is a surface,
 \item a Laurent polynomial ring $\Z[t_1^{\pm1}, \ldots, t_\ell^{\pm 1}]$, or 
\item a Laurent polynomial ring of the form $\mathbb{F}[t_1^{\pm1}, \ldots, t_\ell^{\pm 1}]$ where $\mathbb{F}$ is a field (finite or infinite).
\end{itemize}
When it is necessary to distinguish, we say $R$ is {\bf defined over $\Z$} when $R$ has the form $\Z[H_1(\Lambda)]$ or $\Z[t_1^{\pm1}, \ldots, t_\ell^{\pm 1}]$, and say $R$ is {\bf defined over $\F$} when $R$ has the form $\mathbb{F}[t_1^{\pm1}, \ldots, t_\ell^{\pm 1}]$.  We use $\coeff = \Z$ or $\mathbb{F}$ to denote this subring of $R$.  
Given a set $\{x_1, \ldots, x_n\}$, we denote by $R\langle x_1, \ldots, x_n\rangle$ the free unital  associative $R$-algebra
generated by $x_1, \ldots, x_n$. 
This is the algebra of non-commutative polynomials in indeterminants $x_1, \ldots, x_n$ with coefficients in $R$; as an $R$-module it is free with basis consisting of all finite length words in the $x_i$ (including the empty word $1$ which is the identity element).  
With the coefficient ring $R \cong R \cdot 1 \subset R\langle x_1, \ldots, x_n\rangle$ placed\footnote{In the setting of LCH dg-algebras, the condition of the coefficient ring $R = \Z[H_1(\Lambda)]$ having graded degree $0$ arises when working with Legendrians with vanishing Maslov class.} in graded degree $0$, making a choice of degree for each generator, notated as $|x_i| \in \Z, 1 \leq i \leq n$,  and extending to products as  $|x_{i_1} \cdots x_{i_s}| = |x_{i_1}| + \cdots + |x_{i_s}|$ 
produces a $\Z$-grading,
\[
R\langle x_1, \ldots, x_n\rangle = \oplus_{k \in \Z} A_k, \quad A_{k_1} \cdot A_{k_2} \subset A_{k_1+k_2}.
\]
Here, $A_k$ is the $R$-submodule of $R\langle x_1, \ldots, x_n\rangle$ spanned by words in the $x_i$ of degree $k$.

\begin{definition} \label{def:DGAs}
	\begin{itemize}
		\item
	  A {\bf dg-algebra} (abbrv. {\bf DGA}) {\bf over $R$} is a pair $(\alg, \partial)$ consisting of a $\Z$-graded $R$-algebra, $\alg = \oplus_{k \in \Z} \alg_k$, and a differential, $\partial : \alg \rightarrow \alg$, that is a degree $-1$ graded derivation.  Explicitly, $\partial$ is an $R$-module homomorphism satisfying $\partial^2=0$, $\partial(A_k) \subset A_{k-1}$, and  $\partial(xy) = \partial(x) y + (-1)^{|x|}x \partial(y)$ for all homogeneous elements $x,y \in \alg$.  

\item A {\bf based dg-algebra over $R$} is a dg-algebra $(\alg, \partial)$ equipped with an explicit choice of free generating set $\mathcal{B} = \{x_1, \ldots, x_n\} \subset \alg$ consisting of elements of homogeneous degree such that $R \langle x_1, \ldots, x_n\rangle \cong \alg$.

\item A based dg-algebra is called {\bf triangular} if there is an ordering of the generating set $\mathcal{B} = \{x_1, \ldots, x_n\}$ such that for all $1\leq i \leq n$ we have $\partial x_i \in R\langle x_1, \ldots, x_{i-1}\rangle$.  

\item The {\bf homomorphisms} and {\bf isomorphisms} of dg-algebras over $R$ we consider are unital $R$-algebra homomorphisms that preserve grading and commute with differentials.  Sometimes we consider dg-algebra homomorphisms over the subring $\coeff\subset R$ that are only required to preserve the scalar multiplication by elements of $\coeff$.

\item  An {\bf elementary automorphism} of a based $R$-algebra $(\mathcal{A}, \mathcal{B} = \{x_1,\ldots, x_n\})$ is a unital $R$-algebra automorphism $\Phi:\mathcal{A} \rightarrow \mathcal{A}$ that for some $1\leq j \leq n$ has the form
\[
\begin{array}{lcl} \Phi(x_i) = x_i,  & \quad & \mbox{for $i \neq j$, and} \\
\Phi(x_j) = \alpha x_j + w, & \quad &\mbox{with $\alpha \in R^*$ and $w \in R\langle x_1,\ldots, \widehat{x_j},\ldots, x_n\rangle$.}
\end{array}
\]

\item Let $\coeff$ be $\Z$ or (a field) $\mathbb{F}$ such that the coefficient ring $R$ is defined over $\coeff$.  A {\bf coefficient automorphism} of the based algebra $(\mathcal{A}, \mathcal{B})$ is a $\coeff$-algebra automorphism $\Phi:\mathcal{A} \rightarrow \mathcal{A}$ that restricts to a $\coeff$-algebra automorphism of $R$ and has $\Phi(x_i) = x_i$ for all $x_i \in \mathcal{B}$.

\item A dg-algebra isomorphism, $\Phi:(\alg_1, \partial_1) \rightarrow (\alg_2, \partial_2)$, between based dg-algebras over $R$ is called {\bf tame} (resp. {\bf semi-tame}) if $\Phi$ can be factored into a composition of elementary automorphisms (resp. elementary automorphisms and coefficient automorphisms) of $\alg_1$  followed by an isomorphism between $\alg_1$ and $\alg_2$ induced by a bijection of the generating sets.  
Note that a tame isomorphism is an isomorphism of $R$-algebras, while a semi-tame isomorphism may only be an isomorphism of $\coeff$-algebras.

\item A {\bf stabilization} of a based dg-algebra $((\alg, \partial), \mathcal{B})$ is a based dg-algebra $((S\alg, \partial'), \mathcal{B}')$ with generating set of the form $\mathcal{B}' = \mathcal{B} \sqcup\{y_1, \ldots, y_r, z_1, \ldots, z_r\}$ where $|y_j| = |z_j|+1$ for $1 \leq j \leq r$ and with differential satisfying $\partial'x_i = \partial x_i$ for all $x_i \in \mathcal{B}$  and $\partial'y_j = z_j$ for $1\leq j \leq r$.

\item Two based dg-algebras are {\bf stable tame}  (resp. {\bf semi-tame}) {\bf isomorphic} over $R$ if they have stabilizations that are tamely (resp. semi-tamely) isomorphic as unital dg-algebras over $R$.

\end{itemize}
We note that stable tame isomorphism provides an equivalence relation on the class of based dg-algebras over $R$ as does stable semi-tame ismorphism.
\end{definition}

The following method is useful for producing stable tame isomorphic dg-algebras in the triangular case.

\begin{proposition} \label{prop:cancel} Let $(\alg, \partial)$ be a triangular dg-algebra with ordered generating set $(\mathcal{B}, \leq)$.  Suppose that $\{y_1, \ldots, y_r, z_1, \ldots, z_r\} \subset \mathcal{B}$ has the property that for all $1 \leq j \leq r$, $\partial y_j$ has the form
	\begin{equation} \label{eq:yzvw} 
	\partial y_j = \alpha z_j + v+w,  
	\end{equation}
	where $\alpha \in R^*$ is an invertible element; $v \in \mathcal{I}(y_1, \ldots, y_{j-1}, \partial y_1, \ldots, \partial y_{j-1})$ belongs to the two sided ideal generated by $y_1, \ldots, y_{j-1}, \partial y_1, \ldots, \partial y_{j-1}$; and $w \in R\langle x \in \mathcal{B} \,|\, x < z_j\rangle$.  
Then, the quotient dg-algebra  $\mathcal{A}/\mathcal{I}(y_1, \ldots, y_{r}, \partial y_1, \ldots, \partial y_{r})$	is based and triangular with respect to the generating set $\mathcal{B} \setminus \{y_1, \ldots, y_r, z_1, \ldots, z_r\}$ and is stable tame isomorphic to $(\alg, \partial)$ over $R$.
\end{proposition}

\begin{proof}  See \cite[Theorem 2.1]{RuSu1} for a proof with $R=\Z/2\Z$ that modifies in a straightforward manner to the current setting. 
\end{proof}

\subsection{Specializations of DGAs}  \label{sec:special}  Let $(\alg, \partial_\alg)$ be a based dg-algebra with generating set $\mathcal{B}$ defined over a coefficient ring of the form $R = \Z[t_1^{\pm1}, \ldots, t_\ell^{\pm1}]$.  We will make use of the following two specialization constructions.

\medskip

\begin{construction} \label{cons:spec1} (\textit{Specializing a generator})  Let $b \in \mathcal{B}$ be a generator satisfying $\partial b=0$.  Given a field  $\mathbb{F}$  choose a value, $\overline{b} \in \mathbb{F}$, with the restriction that $\overline{b}=0$ if $|b| \neq 0$.  We then have a {\bf specialization} $(\mathcal{A}|_{b= \overline{b}}, \partial')$ that is the based dg-algebra defined as follows:
\begin{itemize}
\item The coefficient ring is $R' = \mathbb{F}[t_1^{\pm1}, \ldots, t_\ell^{\pm1}]$. 
\item The generating set is $\mathcal{B}' = \mathcal{B} \setminus \{b\}$.
\item For any $x \in \mathcal{B}'$, the differential $\partial'x$ is obtained from $\partial_\alg x$ by replacing all occurrences of $b$ with $\overline{b}$, i.e.  $\partial' x = (\partial_\alg x)|_{b= \overline{b}}$.
\end{itemize}
\end{construction}

Alternatively, $(\mathcal{A}|_{b= \overline{b}}, \partial')$ may be constructed from $(\mathcal{A}, \partial_\alg)$ by first tensoring with $\mathbb{F}$ and then quotienting by the $2$-sided ideal, $\mathcal{I}$, generated by $b -\overline{b}$.  Notice that, with the above requirements on $b$ and $\overline{b}$, $\mathcal{I}$ is a graded, differential ideal so that the quotient is indeed a dg-algebra. 

\medskip

\begin{construction} \label{cons:spec2}  (\textit{Specializing the coefficient ring})  Fix some $t_i \in R = \Z[t_1^{\pm1}, \ldots, t_\ell^{\pm1}]$.  Given a field $\mathbb{F}$ and a choice of non-zero $\overline{\lambda} \in \mathbb{F}^*$, we form the {\bf specialization} $(\mathcal{A}|_{t_i= \overline{\lambda}}, \partial')$ 
		as a based dg-algebra over $R' = \mathbb{F}[t_1^{\pm1}, \ldots, \widehat{t_i}, \ldots, t_\ell^{\pm1}]$ with generating set $\mathcal{B}$ and with the differentials $\partial'x$, $x \in \mathcal{B}$, obtained from $\partial_\alg x$ by replacing all occurrences of $t_i$  with $\overline{\lambda}$.
\end{construction}

\subsection{Augmentation numbers of based dg-algebras}

Let $\mathbb{F}$ denote a field.  Let $(\alg, \partial)$ be a based dg-algebra with coefficient ring $R$ (as above) where $R$ is defined over $\coeff$ with $\coeff= \Z$ or $\coeff=\mathbb{F}$.

\begin{definition} \label{def:aug}
	\begin{itemize}
		\item  An $\mathbb{F}$-valued {\bf augmentation} of $(\alg, \partial)$
is a $\coeff$-algebra homomorphism $\epsilon: \mathcal{A} \rightarrow \mathbb{F}$ satisfying $\epsilon(1) =1$, $\epsilon \circ \partial = 0$, and $\epsilon(x) =0$ whenever $|x| \neq 0$.  Equivalently, $\epsilon:(\mathcal{A}, \partial) \rightarrow (\mathbb{F},0)$ is a unital homomorphism of dg-algebras over $\coeff$ where $\mathbb{F}$ is concentrated in grading $0$.
		\item The set of all $\mathbb{F}$-valued augmentations of $(\mathcal{A}, \partial)$ is denoted
		\[
		V(\mathcal{A}, \mathbb{F}) = \{ \epsilon:(\mathcal{A}, \partial) \rightarrow (\mathbb{F},0) \}
		\]
		and is referred to as the ($\mathbb{F}$-valued) {\bf augmentation variety} of $(\mathcal{A}, \partial)$. 
	\end{itemize}
\end{definition}

\begin{remark} \label{rem:Faugs} When $(\alg, \partial)$ has free generating set $\mathcal{B}= \{x_1, \ldots, x_n\}$ and coefficient ring $R= \Z[t_1^{\pm1}, \ldots, t_\ell^{\pm 1}]$ or $\mathbb{F}[t_1^{\pm1}, \ldots, t_\ell^{\pm 1}]$, an $\mathbb{F}$-valued augmentation is  uniquely determined by a choice of values $\epsilon(t_j) \in \mathbb{F}^*$ and $\epsilon(x_i) \in \mathbb{F}$ satisfying the Laurent polynomial equations $\epsilon(x_i)=0$ for all $i$ with $|x_i| \neq 0$ and $\epsilon(\partial x_i)=0$ for all $i$ with $|x_i|=1$.  
In this way, the augmentation variety is identified with an affine algebraic set in $(\mathbb{F}^*)^\ell \times \mathbb{F}^n$.
In this article, we will not need to use any structure on $V(\alg, \partial)$ beyond that of a set. 
\end{remark}

\begin{remark} \label{rem:mgraded}
	More generally, given a non-negative integer $m\geq 0$, one can consider {\bf $m$-graded augmentations} by relaxing the grading condition to only require that $\epsilon$ preserves grading mod $m$, i.e., that $\epsilon(x) \neq 0$ implies that $x \in \oplus_{k \equiv 0 \mbox{ mod $m$}} \mathcal{A}_k$.
		Definition \ref{def:aug} then corresponds to {\it $0$-graded augmentations} (where $m=0$), also called {\it graded augmentations} or {\it $\Z$-graded augmentations} in the literature.  In this article, we restrict attention to the $m=0$ case so that, except where otherwise stated, all augmentations are assumed to be $0$-graded.  Legendrian isotopy invariant augmentation numbers can also be defined in the $m$-graded case; see \cite{NgSab, HR, NRSS}.
\end{remark}

\begin{definition}[\cite{NgSab}]  \label{def:DGalgAugNum} Let $q \in \Z_{>0}$ a prime power and $\mathbb{F}_q$ denote the finite field of order $q$.  Let $((\mathcal{A},\partial), \mathcal{B})$ be a based dg-algebra with coefficient ring $R$ defined over $\coeff = \Z$ or $\coeff = \mathbb{F}_q$.
	\begin{itemize}
		\item The {\bf shifted Euler characteristic} of $((\mathcal{A},\partial), \mathcal{B})$ is defined by 
		\[
		\chi^*(\alg) = \sum_{k \geq 0} (-1)^k b_k + \sum_{k <0} (-1)^{k+1} b_k
		\] 
		where $b_k$ denotes the number of elements in $\mathcal{B}$ of degree $k$.
		\item  
The $\mathbb{F}_q$-valued {\bf augmentation number} of the dg-algebra  $(\mathcal{A},\partial)$ is 
		\[
		\aug( \alg, q) = q^{-\chi^*(\alg)/2}\cdot |V(\alg, \mathbb{F}_q)|.
		\]
\end{itemize}
\end{definition}
\begin{remark} 
We caution that the augmentation number of a Legendrian surface $\Lambda$, appearing above in (\ref{eq:1dimskeinrel}), differs from the augmentation number of its LCH dg-algebra, $\alg(\Lambda)$, by an additional factor, cf. Definition \ref{def:augLambda} below.  This also differs from the actual `number of augmentations' $|V(\alg, \mathbb{F}_q)|.$  
\end{remark}

\begin{proposition}[\cite{NgSab}]  \label{prop:stableinv}
	If two based dg-algebras $\alg_1$ and $\alg_2$ are stable tame  isomorphic or stable semi-tame isomorphic, then $\aug( \alg_1, q) = \aug(\alg_2,q)$.
\end{proposition}
\begin{proof}  Using that the coefficient ring $R$ is assumed to lie in grading $0$, we see that semi-tamely isomorphic based $R$-algebras have the same shifted Euler characteristic. As the count of augmentations of $(\mathcal{A}, \partial)$ to $\mathbb{F}_q$ depends only on the isomorphism type of $(\mathcal{A}, \partial)$ over $\coeff$, it follows that semi-tamely isomorphic dg-algebras have equal augmentation numbers.  To complete the proof, it suffices to observe that $\aug(S\alg, q) = \aug( \alg, q)$ when $S\alg$ is a stabilization of $\alg$.  To see this, note that with the generating set of $S\alg$ notated as in Definition \ref{def:DGAs}, we have
	\[
	\chi^*(S\alg) = \chi^*(\alg) + 2 Y_0 \quad \mbox{and} \quad |V(S\alg, \mathbb{F}_q)| = q^{Y_0} \cdot |V(\alg, \mathbb{F}_q)|
	\]
	where $Y_0$ is the number of $y_i$ generators having degree $0$.
\end{proof}

\begin{remark}  In the case that $R$ is defined over $K = \Z$, $\aug( \alg, q)$ is defined for all values of $q$ and we can speak collectively of the {\it augmentation numbers}, and view them together as a function $q \mapsto \aug( \alg, q)$ defined on the set of prime powers.
\end{remark}

We record also a property relating augmentations with the specializations from Construction \ref{cons:spec1} and \ref{cons:spec2}.
\begin{proposition} \label{prop:specbiject}
Let $((\mathcal{A}, \partial), \mathcal{B})$ be a based DGA with coefficient ring $\Z[t_1, \ldots, t_\ell]$, and let $\mathbb{F}$ be a field.  
\begin{itemize}
\item Let $b \in \mathcal{B}$ with $\partial b =0$ and $\overline{b} \in \mathbb{F}$ with the restriction that $\overline{b}=0$ if $|b| \neq 0$.  Then, we have a bijection
\[
V(\mathcal{A}|_{b= \overline{b}}, \mathbb{F})  \leftrightarrow \left\{ \epsilon \in V(\mathcal{A}, \mathbb{F}) \,\mid\, \epsilon(b) = \overline{b} \right\}.
\]
\item For any $1\leq i \leq \ell$ and any $\overline{\lambda} \in \mathbb{F}^*$, we have a bijection,
\[
V(\mathcal{A}|_{t_i= \overline{\lambda}}, \mathbb{F})  \leftrightarrow \left\{ \epsilon \in V(\mathcal{A}, \mathbb{F}) \,\mid\, \epsilon(t_i) = \overline{\lambda} \right\}.
\]
\end{itemize}
\end{proposition}
\begin{proof}  This is clear from the viewpoint of Remark \ref{rem:Faugs}.
\end{proof}

\section{The LCH dg-algebra for surfaces with collections of curves}  \label{sec:LCHdgalgebra}

The Legendrian isotopy invariant augmentation numbers, $\aug_{\Lambda}(q)$, considered in this paper may be defined in terms of the standard version of the LCH dg-algebra with coefficients in $\Z[H_1(\Lambda)]$ \cite{EES05a,EES05c,EES07} as was done in the introduction via (\ref{eq:1dimskeinrel}).  In Section \ref{sec:classicLCH} we briefly review this version of the LCH dg-algebra after collecting some basics about $\Z$-graded Legendrian surfaces in Section \ref{sec:Zgraded}.  Following \cite{EENS}, the $H_1(\Lambda)$-coefficients appearing in the differential of generators may be calculated by recording intersections of boundaries of rigid holomorphic disks with a collection of curves, $\Gamma= \{\Gamma_i\}$, forming a basis for $H_1(\Lambda)$.  For computing intersection numbers it is slightly more direct to work instead with co-oriented curves (that we interpret as cocycles in Section \ref{sec:multiDGA}) forming a basis for $H^1(\Lambda)$, and we take this approach in Section \ref{sec:multiDGA}. 
For our later derivation of the skein relations for $\aug_{\Lambda}(q)$, it is convenient to also consider a variant of the LCH dg-algebra, $\mathcal{A}(\Lambda, \Gamma)$ which we call the {\it multi-curve DGA}, that allows for a larger collection of curves $\Gamma$ that spans $H^1(\Lambda)$ but is not necessarily linearly independent.  This is analogous to the multi-base pointed DGA from the case of $1$-dimensional Legendrians from \cite{NgR}, and also fits into the more general scheme of LCH dg-algebras associated to Legendrians equipped with ``homotopy structures'' used in \cite{RuSu5}. 
In Section \ref{sec:defauginvariance}, we discuss this variant of the LCH dg-algebra and use it to provide a more flexible definition for $\aug_{\Lambda}(q)$ generalizing (\ref{eq:1dimskeinrel}).

\subsection{$\Z$-graded Legendrian surfaces} \label{sec:Zgraded} We will work with closed Legendrian surfaces in $1$-jet spaces, $J^1S$, where $S$ is an oriented surface (possibly open).  We use local coordinates of the form $(x_1,x_2,y_1,y_2,z)$ on $J^1S = T^*S \times \R$ where $(x_1,x_2)$ is a local coordinate on $S$, $(y_1,y_2)$ is the resulting coordinate on the co-tangent fibers, and $z$ is the $\R$ coordinate.  The contact structure on $J^1S$ is $dz-\sum_{i=1}^2y_i\,dx_i$ and the Reeb vector field is $\partial_z$.  The case of $S=\R^2$ is standard contact $\R^5$.  

We will often depict (subsets of) a Legendrian surface $\Lambda \subset J^1S$ via its {\bf front projection}, $\pi_{xz}(\Lambda) \subset S \times \R$, or by depicting the image of front singularities in the {\bf base projection} $\pi_x(\Lambda) \subset S$.  We assume familiarity with the generic singularities of the front projections of surfaces, cf. \cite[Section 2.2]{RuSu1}. 
In particular, a generic Legendrian surface $\Lambda \subset J^1S$ has the form $\Lambda = \sqcup_{i= 0}^2 \Lambda_i$ where the front projection $\pi_{xz}|_{\Lambda}$ is an immersion along $\Lambda_0$, has cusp edge singularities along $\Lambda_1$, and has swallowtail singularities at the points of $\Lambda_2$.  Each $\Lambda_i$ is a submanifold of codimension $i$ in $\Lambda$.  Swallowtail points are either {\bf upward} or {\bf downward} depending on whether the sheet that connects the two cusp edges appears above or below the crossing arc near the swallowtail points.  See Figure \ref{fig:UpwardSwallow}.

\begin{figure}
	
\labellist
\small
\pinlabel $x_1$ [t] at -2 -2
\pinlabel $x_2$ [l]  at 72 22
\pinlabel $z$ [b] at 24 70
\endlabellist
\centerline{\includegraphics[scale=.5]{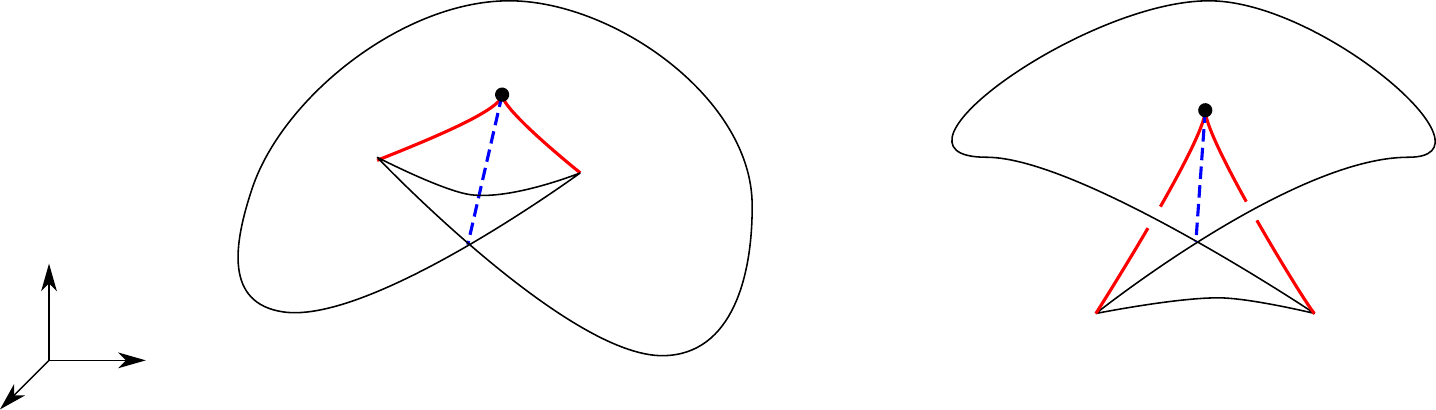} }

\caption{Neighborhoods of an upward (left) and downward (right) swallowtail point pictured in the front projection.  Cusp edges are depicted in solid red, and crossing arcs in dotted blue.
}
\label{fig:UpwardSwallow}
\end{figure}

Recall that a {\bf $\Z$-valued Maslov potential} on $\Lambda$ is a locally constant function $\Lambda_0 \rightarrow \Z$ whose value increases by one when passing a cusp edge from the lower sheet to the upper sheet (in terms of the $z$-coordinate).  Such a $\Z$-valued Maslov potential exists if and only if $\Lambda$ has vanishing Maslov class, or equivalently the minimal Maslov number of $\Lambda$ is $0$; see, eg. \cite[Section 2.1.2]{RuSu2}.  A {\bf $\Z$-graded Legendrian surfaces} $(\Lambda, \mu)$ in $J^1S$ is an ordered pair consisting of a Legendrian surface $\Lambda  \subset J^1S$ equipped with a choice of $\Z$-valued Maslov potential, $\mu$.

\medskip

\noindent{\bf Standing assumption:}  Except where specified, {\it all Legendrian surfaces in this article are assumed to be $\Z$-graded.}  Frequently, the Maslov potential will be suppressed from notation.

\medskip

\begin{remark}
	\begin{enumerate}
		\item For considering $\Z$-graded augmentation numbers (terminology as in Remark \ref{rem:mgraded}), no generality is lost by restricting attention to $\Z$-graded Legendrians.  Indeed, 
		$\Z$-graded augmentations of the LCH dg-algebra can exist only if $\Lambda$ has vanishing Maslov class (since otherwise $\mathcal{A}(\Lambda)$ has invertible elements in non-zero graded degree that have the form $e^A$ with $A \in H_1(\Lambda)$),  cf. \cite[Section 2.2]{EES05a}. 
		\item  Moreover, with the base surface $S$ assumed oriented, all $\Z$-graded Legendrian surfaces in $J^1S$ are orientable as the mod $2$ reduction of the Maslov potential specifies an orientation.  See \cite[Remark 2.9]{PanRu}.
		\end{enumerate}
\end{remark}

\subsection{The classical LCH dg-algebra}  \label{sec:classicLCH}
We assume some familiarity with the LCH dg-algebra of a ($\Z$-graded) Legendrian surface $\Lambda \subset J^1S$ using coefficients in the group ring $\Z[H_1(\Lambda)]$ as constructed in \cite{EES05a,EES05c,EES07}, and we will denote it as $(\mathcal{A}(\Lambda; \Z[H_1(\Lambda)]),\partial)$ or simply $\mathcal{A}(\Lambda; \Z[H_1(\Lambda)])$.  Given a homology class $A \in H_1(\Lambda)$ we write $e^A$ for the corresponding element in $\Z[H_1(\Lambda)]$ so that $\Z[H_1(\Lambda)]$ is the set of finite $\Z$-linear combinations of the form $\sum_{i} c_i e^{A_i}$.  The algebra $\mathcal{A}(\Lambda; \Z[H_1(\Lambda)])$ is the based dg-algebra over $R = \Z[H_1(\Lambda)]$ 
with generators $a_1, \ldots, a_n$ corresponding to Reeb chords of $\Lambda$.
 The Maslov potential $\mu$ is used to specify a grading of generators, $|a_i| \in \Z$ (as in \cite[Section 2.1.2]{RuSu2}) which produces the $\Z$-grading on $\mathcal{A}(\Lambda; \Z[H_1(\Lambda)])$ with the group ring $\Z[H_1(\Lambda)] \cong \Z[H_1(\Lambda)]\cdot 1 \subset \alg(\Lambda)$ sitting in graded degree $0$ (due to the assumption of vanishing Maslov class).
The differential of a Reeb chord $a$ is computed by summing over all rigid moduli spaces of boundary punctured disks, $D^2_* = D^2 \setminus\{ q_0, \ldots, q_r\}$,  mapped holomorphically to the symplectization, $u: (D^2_*, \partial D^2_*) \rightarrow (\R \times J^1M, \R\times \Lambda))$, having a single positive puncture at (in the asymptotic sense) $a$ and an arbitrary number of negative punctures.  We have
\begin{equation} \label{eq:LCHoriginal}
\partial a = \sum_{u} \iota(u) e^{A(u)} w(u)
\end{equation}
where $w(u) = b_1\cdots b_r$ is the product of Reeb chords appearing at the negative punctures of $u$ in counter-clockwise order around the boundary of the domain of $u$ starting from the positive puncture; $A(u) \in H_1(\Lambda)$ is the homology class represented by the sum of the  boundary segments of $u$ (projected to $\Lambda$ and compactified) and appropriately chosen capping paths (or basepoint paths in the disconnected case, see \cite[Section 2.1.3]{RuSu2});  and $\iota(u) \in \{\pm1\}$ is a sign arising from a coherent orientation of moduli spaces of disks that depends on a choice of spin structure on $\Lambda$.

\begin{remark}
\begin{enumerate}
\item When $\Lambda$ is connected $\mathcal{A}(\Lambda; \Z[H_1(\Lambda)])$ (and so also its augmentation numbers) is independent of the choice of Maslov potential.  In the disconnected case, the grading on $\mathcal{A}(\Lambda; \Z[H_1(\Lambda)])$ may depend on the choice of Maslov potential.
\item LCH dg-algebras, $\alg_1$ and $\alg_2$, arising from two different choices of spin structure on $\Lambda$ are isomorphic {\it as dg-algebras over $\Z$} via an isomorphism that is the identity on Reeb chords and has the form $e^A \mapsto \pm e^A$ on homology classes; see \cite[Section 4.4]{EES05c}.  In particular, the augmentation numbers are independent of the choice of spin structure.  For this reason, for much of the discussion we will suppress the dependence of $\mathcal{A}(\Lambda; \Z[H_1(\Lambda)])$ on a choice of spin structure.
\end{enumerate}
\end{remark}

\subsection{The multi-curve LCH dg-algebra}  \label{sec:multiDGA}

Recall (following Poincar\'e) that when $S$ is a closed surface equipped with a smooth triangulation (or polygonal decomposition), $\mathcal{E}$, the cohomology of $S$ may be computed using a chain complex $(C_{co}^\cdot(\mathcal{E}), d)$ such that $C_{co}^k$ is the set of {\it co-oriented} $(2-k)$-chains made up of $(2-k)$-cells of $\mathcal{E}$; the differential $d$ assigns to a co-oriented $(2-k)$-cell, $e$, its $2-(k+1)$ dimensional boundary chain with boundary cells co-oriented using an outward pointing tangent vector to $e$ followed by the co-orientation of $e$.  [Indeed, there is a canonical isomorphism $C_{co}^{k}(\mathcal{E}) = \mathit{Hom}_{\Z}(C_k(\mathcal{E}^*), \Z)$ obtained by intersecting co-oriented $(2-k)$-chains in $\mathcal{E}$ with oriented $k$-chains in the dual polygonal decomposition (aka the barycentric star decomposition) $\mathcal{E}^*$ and comparing the co-orientation and orientation to obtain a sign $\pm1$ at each intersection point.  Under this isomorphism, the standard cohomology differential on $C^k(\mathcal{E}^*) = \mathit{Hom}(C_k(\mathcal{E}^*), \Z)$ matches (up to a sign depending on $k$) the differential $d$ on $C_{co}^{k}(\mathcal{E})$.]

\begin{definition} With the previous discussion in mind, we adopt the terminology {\bf geometric cocycle} of simply {\bf cocycle} to refer to a union or $\Z$-linear combination of co-oriented closed curves.
\end{definition}

Let $\Lambda \subset J^1S$ be a spin $\Z$-graded Legendrian surface, and let $\Gamma = \{\Gamma_i\} = \{ \Gamma_1, \ldots, \Gamma_\ell\}$ be a collection (technically an {\it ordered set}) of geometric cocycles on $\Lambda$.
As above, each $\Gamma_i$ represents a cohomology class $[\Gamma_i] \in H^1(\Lambda)$.  

\begin{definition}  \label{def:multicurveDGA} We associate an LCH dg-algebra to $(\Lambda, \Gamma)$ that we denote by $\mathcal{A}(\Lambda, \Gamma)$ and call the {\bf multi-curve LCH dg-algebra} as follows.  The coefficient ring is $R= \Z[t_1^{\pm1}, \ldots, t_\ell^{\pm1}]$ (where $\ell$ is the number of $\Gamma_i$ cocycles) and the algebra $\mathcal{A}(\Lambda, \Gamma)$ is the based dg-algebra over $R$ generated by the  Reeb chords $a_1, \ldots, a_n$ of $\Lambda$.  The grading of the $a_i$ is the same as in the ordinary LCH dg-algebra.  The differential is defined by summing over the same rigid holomorphic disks as in the ordinary LCH DGA but with the terms modified to
\[
\partial a = \sum_{u} \iota(u)\, t_1^{\mu_1} \cdots t_\ell^{\mu_l} w(u)
\]
where $\iota(u)$ and $w(u)$ are as in (\ref{eq:LCHoriginal}) and $\mu_{i}$ is the intersection number of the union of (oriented) boundary arcs\footnote{These intersection numbers involve {\it only} the boundary arcs of $u$.  In contrast with the formation of $A(u)$ in (\ref{eq:LCHoriginal}), no capping/basepoint paths are involved.} of $u$ with (co-oriented) $\Gamma_i$. 
\end{definition}

See Figure \ref{fig:DiskEx} for a schematic of a rigid disk contributing to the differential in $\mathcal{A}(\Lambda, \Gamma)$.

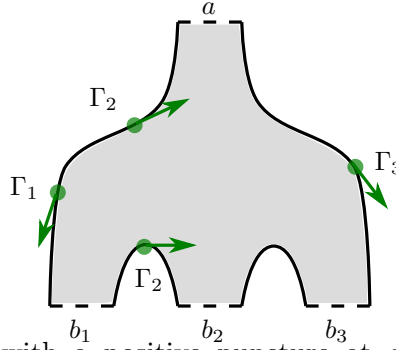
\begin{figure}
	
\labellist
\small
\pinlabel $a$ [b] at 80 134
\pinlabel $\Gamma_2$ [br]  at 40 90
\pinlabel $\Gamma_1$ [r] at 4 56
\pinlabel $b_1$ [t] at 20 -2
\pinlabel $\Gamma_2$ [t] at 52 22
\pinlabel $b_2$ [t] at 82 -2
\pinlabel $b_3$ [t] at 140 -2
\pinlabel $\Gamma_3$ [l] at 154 68

\endlabellist
\centerline{\includegraphics[scale=.8]{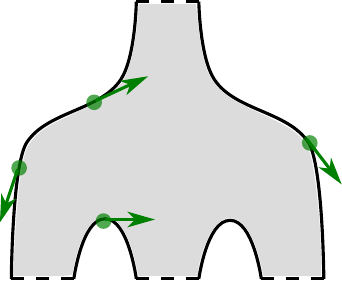} }

\caption{A rigid disk with a positive puncture at $a$ contributing a term $\partial a = \pm t_2^{-1}t_1b_1t_2b_2b_3t_3^{-1}+\cdots= \pm t_1t_3^{-1}b_1b_2b_3+\cdots$.  Intersections with the $\Gamma_i$ curves are indicated by green dots with arrows specifying the co-orientation. 
}
\label{fig:DiskEx}
\end{figure}

\begin{proposition} \label{prop:multiDGA} The $\alg(\Lambda, \Gamma)$ is a dg-algebra, and the stable tame isomorphism type (as a based dg-algebra over $R$) of $\alg(\Lambda, \Gamma)$ is invariant under Legendrian isotopy of $(\Lambda, \Gamma)$.  Moreover, up to tame isomorphism $\alg(\Lambda, \Gamma)$ only depends on the cohomology classes $([\Gamma_1], \ldots, [\Gamma_\ell]) \in (H^1(\Lambda))^\ell$.
\end{proposition}

\begin{proof}
Up to a tame isomorphism, $\alg(\Lambda, \Gamma)$, is obtained from the ordinary LCH dg-algebra $\alg(\Lambda)$ via changing coefficients from $\Z[H_1(\Lambda)]$ to $\Z[t_1^{\pm1}, \ldots, t_\ell^{\pm 1}]$ via the ring homomorphism induced by the group homomorphism
\[
\phi: H_1(\Lambda) \rightarrow (\Z[t_1^{\pm1}, \ldots, t_\ell^{\pm1}])^\times, \quad A \mapsto t_1^{\langle \Gamma_i, A\rangle}\cdots t_\ell^{\langle \Gamma_\ell , A \rangle}
\]
where $\langle \Gamma_i, A \rangle$ is the intersection pairing.  [This is a special case\footnote{In the terminology of \cite{RuSu5}, decorating the co-oriented curve $\Gamma_i$ with $t_i$ makes $\Gamma$ into a $G$-valued homotopy structure where $G$ is the free abelian multiplicative group generated by $t_1, \ldots, t_\ell$.  The associated homomorphism $\Theta_\Gamma: \pi_1(\Lambda, x_0) \rightarrow G$ factors through $H_1(\Lambda)$ as the homomorphism $\phi$.} of a result from \cite{RuSu5}.]
Moreover, the intersection pairing is simply the pairing of cohomology and homology, $\langle \Gamma_i , A \rangle = [\Gamma_i](A)$, so that $\phi$ only depends on the cohomology classes of the $\Gamma_i$. 
\end{proof}

\begin{remark} \label{rem:Gammabasis}
\begin{enumerate} 
\item In the special case where $[\Gamma_1], \ldots, [\Gamma_{2g}]$ is a basis for $H^1(\Lambda)$ the group ring map induced by $\phi$ is an isomorphism ($\Lambda$ is assumed to be a closed surface), so that with the coefficient rings identified in this manner we have $\mathcal{A}(\Lambda; \Z[H_1(\Lambda)]) \cong \alg(\Lambda, \Gamma)$.  This method of computing the LCH dg-algebra using intersections with a basis of curves appears in \cite{EENS}.  
\item  The case where $\Gamma =\{\Gamma_i\}$ spans $H^1(\Lambda)$ but is not linearly independent is analogous to the multi-basepoint DGA used for $1$-dimensional Legendrian knots from \cite{NgR}.
\end{enumerate}
\end{remark}

\subsection{Definition and invariance of augmentation numbers of Legendrian surfaces}  \label{sec:defauginvariance}

Next we provide a more flexible definition of augmentation numbers for Legendrian surfaces that generalizes the equation (\ref{eq:1dimskeinrel}) used in the introduction.

\begin{definition}  \label{def:augLambda} For any prime power, $q$, to   define the {\bf augmentation number of $\Lambda$}, denoted $\aug_\Lambda(q)$, we choose $\Gamma = \{\Gamma_i\}$ so that $[\Gamma_1], \ldots, [\Gamma_\ell]$ spans $H^1(\Lambda)$ and set 
\[
\aug_\Lambda(q)= (q-1)^{\dim H_0(\Lambda)-\ell} \aug(\alg(\Lambda, \Gamma), q)
\]
where $\aug(\alg(\Lambda, \Gamma), q)$ is the augmentation number of the dg-algebra $\alg(\Lambda, \Gamma)$ as defined in Definition \ref{def:DGalgAugNum} above.
\end{definition}

\begin{proposition} \label{prop:augInvariance} The augmentation numbers $\mathit{Aug}_\Lambda(q)$  only depend on the Legendrian isotopy type of $(\Lambda, \mu)$ and in the connected case are independent of the Maslov potential $\mu$.
\end{proposition}

\begin{proof}
With $\Lambda$ fixed, for verifying that $\mathit{Aug}_\Lambda(q)$ is independent of the choice of spanning set $\Gamma,$ we temporarily use the notation $\aug_{(\Lambda, \Gamma)}(q) = (q-1)^{\dim H_0(\Lambda)-|\Gamma|} \aug(\alg(\Lambda, \Gamma), q)$.  To see the independence of $\Gamma$, it suffices to establish the following.

\medskip

\noindent {\bf Claim:}  If $[\Gamma_{\ell+1}] \in \mbox{Span}_{\Z}\{[\Gamma_1], \ldots, [\Gamma_\ell]\}$, then
\[
\aug_{(\Lambda, \Gamma)}(q)= \aug_{(\Lambda, \Gamma \cup \{\Gamma_{\ell+1}\})}(q).
\]

\medskip

\noindent Indeed, using the claim we see that if $\Gamma$ and $\Gamma'$ are two spanning sets for $H^1(\Lambda)$, then
\[
\aug_{(\Lambda, \Gamma)}(q) = \aug_{(\Lambda, \Gamma \cup \Gamma')}(q) = \aug_{(\Lambda, \Gamma')}(q).
\]

To verify the claim, supposing $[\Gamma_{\ell+1}] = \sum_{i=1}^\ell b_i [\Gamma_i]$ we can replace $\Gamma_{\ell+1}$ with $\Gamma'_{\ell+1} = \sum_{i=1}^\ell b_i \Gamma_i$ without changing the augmentation numbers (by Proposition \ref{prop:multiDGA}).  Once this is done we have an isomorphism of dg-algebras over $\Z$ given by 
\[\begin{array}{ccc}
\Z[t^{\pm1}_{\ell +1}] \otimes_\Z \alg(\Lambda, \Gamma) & \rightarrow & \alg(\Lambda, \Gamma \cup \{ \Gamma'_{\ell +1}\}) \\
t_{\ell+1} \otimes 1 & \mapsto & t_{\ell+1} \\
 1 \otimes t_i & \mapsto & t_it_{\ell+1}^{b_i} \\
1 \otimes a_j & \mapsto & a_j
\end{array}
\]
for all Reeb chords $a_j$.  As an $\F$-valued augmentation of $\Z[t^{\pm1}_{\ell +1}] \otimes_\Z \alg(\Lambda, \Gamma)$ is nothing more than an augmentation $\epsilon$ for $\alg(\Lambda, \Gamma)$ together with an arbitrary choice of value for $t_{\ell +1}$ in $\mathbb{F}^*$, we see from Definition \ref{def:augLambda} that
\begin{align*}
\aug_{(\Lambda, \Gamma \cup \{\Gamma_{\ell+1}\})}(q) &= (q-1)^{\dim H_0(\Lambda)-(\ell+1)} \aug(\alg(\Lambda, \Gamma \cup \{\Gamma_{\ell+1}\}), q) \\
 &= (q-1)^{\dim H_0(\Lambda)-\ell-1} q^{-\chi^*(\alg(\Lambda))/2} |V(\alg(\Lambda, \Gamma \cup \{\Gamma_{\ell+1}\}), \mathbb{F}_q)|  \\
 &= (q-1)^{\dim H_0(\Lambda)-\ell-1} q^{-\chi^*(\alg(\Lambda))/2} |\mathbb{F}_q^*\times V(\alg(\Lambda, \Gamma), \mathbb{F}_q)|  \\
&= (q-1)^{\dim H_0(\Lambda)-\ell} q^{-\chi^*(\alg(\Lambda))/2} |V(\alg(\Lambda, \Gamma), \mathbb{F}_q)|  \\
& = \aug_{(\Lambda, \Gamma)}(q). 
\end{align*}

Finally, with the independence of $\Gamma$ established, the Legendrian isotopy invariance of $\aug_{(\Lambda, \Gamma)}(q)$ follows from Propositions \ref{prop:multiDGA} and \ref{prop:stableinv}.
\end{proof}

\begin{remark}
The augmentation numbers defined in Definition \ref{def:augLambda} coincide with those defined in equation (\ref{eq:1dimskeinrel}) of the introduction.  To see this, take $\Gamma$ to be a basis for $H^1(\Lambda)$ and apply the isomorphism from Remark \ref{rem:Gammabasis} (1) to get
\begin{align*}
(q-1)^{\dim H_0(\Lambda)-|\Gamma|} \aug(\alg(\Lambda, \Gamma), q) & = (q-1)^{\dim H_0(\Lambda)-\dim H_1(\Lambda)} \aug(\mathcal{A}(\Lambda; \Z[H_1(\Lambda)]), q) \\
& =(q-1)^{\dim H_0(\Lambda)-\dim H_1(\Lambda)} q^{-\chi^*(\alg(\Lambda))/2} |V(\Lambda, \mathbb{F}_q)|.
\end{align*}
\end{remark}

\section{Cellular computation of the LCH dg-algebras}  \label{sec:CWDGA}

For Legendrian surfaces in $1$-jet spaces, the works \cite{RuSu1, RuSu2, RuSu25, RuSu5} provide a cellular computation of LCH dg-algebras that we will make use of in proving the skein relations from Theorem \ref{thm:main}.  We provide here a brief overview of the version of the cellular DGA from \cite{RuSu5} that allows for computation of the multi-curve LCH dg-algebra $\alg(\Lambda, \Gamma)$ from Definition \ref{def:multicurveDGA}.  We postpone some details related to swallowtail points until Section \ref{sec:Cone} where they become necessary for proofs.

\subsection{Cellular computation of the multi-curve DGA}  \label{sec:CWDGA1}

Let $\Lambda \subset J^1S$ be a (closed) Legendrian surface equipped with a collection of geometric cocycles $\Gamma = \{ \Gamma_1,\ldots, \Gamma_\ell\}$.  The front projection of $\Lambda$ is assumed to have only generic singularities; these include crossing arcs (singularity type $A_1^2$), cusp edges (type $A_2$) of codimension $1$ in $\pi_{xz}(\Lambda)$, as well as triple points (type $A_1^3$), cusp-sheet intersections (type $A_1A_2$), and swallowtail points (type $A_3$) appearing at isolated points of codimension $2$ in $\pi_{xz}(\Lambda)$.  See, e.g., \cite{RuSu1}.

The cellular DGA relies on two additional choices:  
\begin{enumerate}
	\item  A {\bf compatible polygonal decomposition}, $\mathcal{E}$, of the base projection $\pi_x(\Lambda) \subset S$.
\end{enumerate}
Here, 
$\mathcal{E}$ must be a CW-decomposition of $\pi_{x}(\Lambda)$ that contains in its $1$-skeleton the base projection of all front singularities of $\Lambda$; this implies that the  projections of all codimension $2$ singularities are contained in the $0$-skeleton.  In addition, $\mathcal{E}$ includes choices of orientations for $1$- and $2$-cells; choices of initial and terminal vertices for each $2$-cell; and at each swallowtail point $w$ a choice of labeling by $S$ and $T$ of the two $2$-cell corners that border the crossing arc at $w$.  The cells of $\mathcal{E}$ are required to be transverse to the cocycles from $\Gamma$; see also \cite{RuSu1, RuSu5}.

\begin{enumerate}
	\item[(2)]  A {\bf combinatorial spin structure}, $\xi$, for $\Lambda$.
\end{enumerate}
By definition, a combinatorial spin structure
is a $1$-chain in $\Lambda$ with coefficients in $\Z/2$ (i.e., a union of unoriented paths and loops) whose boundary points are in bijection with the set of swallowtail points of $\Lambda$.  For technical reasons, rather than having the boundary of $\xi$ coincide precisely with the swallowtail points of $\Lambda$, it is required that 
\[
\partial \xi = \sum_{w} a_w
\] 
where the sum is over all swallowtail
points $w$ of $\Lambda$, and the terms $a_w$ are defined as follows:  At a swallowtail point $w$, let $e_S$ be the  $2$-cell of $\mathcal{E}$ that contains the $S$ labelled corner at $w$.  In the case that $w$ is an upward (resp. downward) swallowtail point, we lift $e_S$ to the lower (resp. upper) of the two crossing sheets of $\Lambda$ that meets $w$ and take $a_w$ to be the barycenter of this lifted $2$-cell.  See also \cite{RuSu5}.

\begin{definition}  \label{def:Sheets}
A {\bf sheet of $\Lambda$ above a cell} $e^d_\alpha \subset S$ is a connected component of $\Lambda \cap \pi_x^{-1}(e^d_\alpha)$ not contained in the set of cusp singularities of $\pi_{xz}|_{\Lambda}$.  Note that a swallowtail point {\it is} considered to be a valid sheet above a $0$-cell.
\end{definition}

\begin{remark} \label{rem:sheetnum}
\begin{enumerate}
\item[(i)] The base projection $\pi_{x}: \Lambda \rightarrow S$ maps each sheet above $e^d_\alpha$ homeomorphically to $e^d_\alpha$.
\item[(ii)] Crucially, sheets are subsets of $\Lambda$ and {\it not}  subsets of the front projection, $\pi_{xz}(\Lambda)$.  E.g., when a $1$-cell or $0$-cell $e^d_\alpha$ lies in the projection of a crossing arc of the front projection, there are two distinct sheets above $e^d_\alpha$ that both map to the same subset of $\pi_{xz}(\Lambda)$.
\item[(iii)]  We typically enumerate the sheets above a given $e^d_\alpha$ as $S^\alpha_1, \ldots, S^\alpha_n$ with weakly descending $z$-coordinate, i.e. $z(S^\alpha_1) \geq z(S^\alpha_2) \geq \cdots \geq z(S^\alpha_n)$, where $z(S^\alpha_m)$ is the $z$-coordinate along $S^\alpha_m$ viewed as a function $e^d_\alpha \rightarrow \R$.  However, when two or more sheets of $\Lambda$ cross in the front projection above $e^d_\alpha$ this convention does not uniquely specify the numbering, and a choice is required for ordering crossing sheets.  A convenient way to specify this is by choosing a neighboring $2$-cell $e'$ that contains $e^d_\alpha$ in its closure and ordering the crossing sheets in the order that their continuous extensions appear above $e'$.  (As all singularities of the front projection project into the $1$-skeleton, above any (open) $2$-cell sheets are always totally ordered by $z$-coordinates.)  In figures, we indicate the choices of $e'$ with small arrows pointing from a crossing cell into the interior of the chosen $e'$. 
\end{enumerate}
\end{remark}

Given such a choice of compatible polygonal decomposition $\mathcal{E}$ and combinatorial spin structure $\xi$ we define the {\bf cellular DGA} $(\algCW(\Lambda, \Gamma, \mathcal{E}, \xi), \partial)$, often denoted  simply as $\algCW(\Lambda, \Gamma)$, as a based DGA in the following manner.  The coefficient ring is $R = \Z[t_1^{\pm1}, \ldots, t_\ell^{\pm 1}]$.  The generating set consists of one generator, $x^\alpha_{i,j}$, for each triple $(e^d_\alpha, S^\alpha_i, S^\alpha_j)$ consisting of a cell, $e^d_\alpha$, of $\mathcal{E}$ together with two sheets of $\Lambda$ above $e^d_\alpha$ satisfying $z({S_{i}^\alpha})> z(S_j^\alpha)$, i.e. at all points $x \in e^d_\alpha$ the $z$-coordinate of $S_i^\alpha$ at $x$ is strictly above the $z$-coordinate of $S_j^\alpha$.  
A $\Z$-grading (with $R$ placed in degree $0$) arises from the $\Z$-valued Maslov potential, $\mu$, on $\Lambda$ according to the formula
\begin{equation} \label{eq:CWgrading}
|x_{i,j}^\alpha| = \mu(S_i^\alpha)- \mu(S_j^\alpha) + \dim(e^d_\alpha) -1.
\end{equation}
We indicate the dimension of the underlying cell $e^d_\alpha$ in our notation for generators by writing $a^\alpha_{i,j}$, $b^\alpha_{i,j}$, or $c^\alpha_{i,j}$ rather than $x^\alpha_{i,j}$ depending whether $\dim(e^d_\alpha) = 0, 1$, or $2$ respectively.  

\begin{remark} 
	\begin{enumerate}
		\item 
	Due to the numbering of sheets $S_1^\alpha, \ldots, S_{n_\alpha}^\alpha$ above each $e^d_\alpha$ with (weakly) {\it decreasing} $z$-coordinate, i.e., $z(S_i^\alpha) \geq z(S_j^\alpha)$ whenever $i<j$, generators $x^\alpha_{i,j}$ always have $i<j$.  

	\item When sheets $S_k$ and $S_l$ cross above $e^d_\alpha$, there is no generator of the form $x^\alpha_{k,l}$.
	
	\item When $S_i^\alpha=w$ is a swallowtail point $w$ above a $0$-cell, the term $\mu(S_i^\alpha)$ from (\ref{eq:CWgrading}) should be taken to be the value of $\mu$ on the crossing sheets near $w$.
\end{enumerate}
\end{remark}

The differential $\partial: \algCW(\Lambda, \Gamma) \rightarrow \algCW(\Lambda, \Gamma)$ is characterized by the following  matrix formulas where in all cases we apply $\partial$ to matrices in an {\it entry-by-entry} manner.  For simplicity, we restrict our discussion here to the case of cells that {\it do not contain swallowtail points} in their closure.  The adjustments for swallowtail points are provided in Section \ref{sec:Cone} below where they are needed; see also \cite{RuSu5}.

\medskip

\noindent{\it $0$-cells:}  Using the ordering of sheets above a $0$-cell $e^0_\alpha$ as  $S_1^\alpha,\ldots, S_n^\alpha$ with $n=n_\alpha$, place the generators into an $n \times n$ strictly upper triangular matrix $A$ whose $i,j$-entry is $a^\alpha_{i,j}$ (resp. $0$) whenever $a^\alpha_{i,j}$ exists (resp. does not exist).  In addition, a diagonal matrix is formed with entries $\pm 1$ reflecting the parity of the Maslov potential $\mu$ as
\begin{equation} \label{eq:Jmatrix}
J_\alpha = \mbox{diag}\left((-1)^{\mu(S_1^\alpha)}, \ldots, (-1)^{\mu(S_n^\alpha)}\right)
\end{equation}
Then,
\begin{equation} \label{eq:dA}
\partial A = J_\alpha \cdot A^2.
\end{equation}

\medskip

\noindent{\it $1$-cells:}  Let $e^1_\alpha$ be a $1$-cell of $\mathcal{E}$ and notate the boundary cells of $e^1_\alpha$ as $e^0_\pm$ so that $\partial e^1_\alpha = e^0_+-e^0_-$.  Place the generators $b_{i,j}^\alpha$ associated to $e^1_\alpha$ into an upper triangular matrix $B= (b_{i,j}^\alpha)$, and form a diagonal matrix $J_\alpha$ as in (\ref{eq:Jmatrix}).  Form a diagonal matrix $\Delta = \mbox{diag}(\delta_1, \ldots, \delta_n)$ using $\Gamma$ and the combinatorial spin stucture $\xi$ by setting
\[
\delta_i  = (-1)^{\langle \xi, S_i^\alpha \rangle} \prod_{m =1}^\ell t_m^{\langle \Gamma_m, S_i^\alpha \rangle}
\]
where $\langle \xi, S_i^\alpha\rangle$ is a mod $2$ intersection number and $\langle \Gamma_m, S_i^\alpha \rangle$ is a $\Z$-valued intersection number arising from the co-orientation of $\Gamma_m$  and the orientation of $S_i^\alpha$ as a lift of $e^1_\alpha$.  
In addition, form a $n_\alpha \times n_\alpha$ {\bf boundary matrix} $A_+$ by (i) placing each generator $a^{+}_{p,q}$ associated to $e^0_+$ into the $i,j$-position specified by $S^+_{p} \subset \overline{S_{i}}$ and $S^+_q \subset \overline{S_j}$ and by (ii) placing $1$ into each $k,l$-position for which sheets $S^\alpha_k$ and $S^\alpha_l$ with $z(S^\alpha_k)>z(S^\alpha_l)$ above $e^1_\alpha$ meet at a cusp point above $e^0_+$.  A second boundary matrix $A_-$ associated to $e^0_-$ is formed similarly, and the differentials $\partial b^\alpha_{i,j}$ are specified by
\begin{equation} \label{eq:dB}
\partial B = J_\alpha \cdot\left[A_+(\Delta+B)-(\Delta+B)A_-\right].
\end{equation}

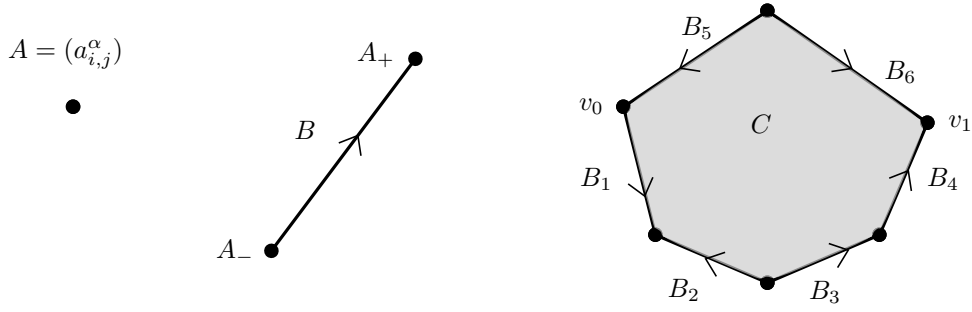
\begin{figure}
	
\labellist
\small
\pinlabel $A=(a^\alpha_{i,j})$ [b] at 0 100
\pinlabel $B$ [br]  at 120 68
\pinlabel $A_+$ [r] at 158 110
\pinlabel $A_-$ [r] at 92 18
\pinlabel $v_0$ [r] at 256 86
\pinlabel $v_1$ [l] at 410 78
\pinlabel $C$ [b] at 326 70
\pinlabel $B_1$ [r] at 260 52
\pinlabel $B_2$ [tr] at 300 8
\pinlabel $B_3$ [t] at 356 8
\pinlabel $B_4$ [l] at 400 52
\pinlabel $B_5$ [b] at 296 114
\pinlabel $B_6$ [l] at 380 102

\endlabellist
\centerline{\includegraphics[scale=.8]{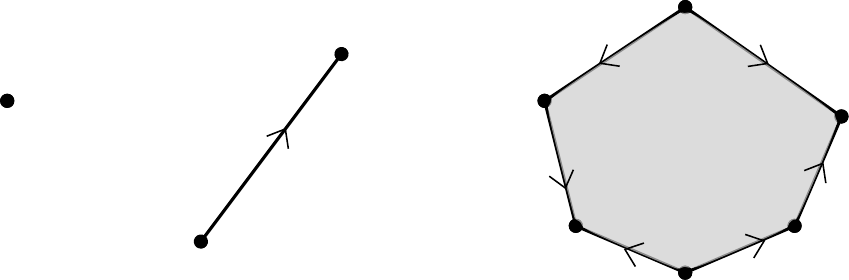} }

\caption{The matrices of Reeb chords associated to cells and their boundaries that appear in (\ref{eq:dA}), (\ref{eq:dB}), and (\ref{eq:dC}).  For a $2$-cell as pictured, (\ref{eq:dC}) becomes $\partial C = J_{\alpha}\cdot[A_{v_1}C+CA_{v_0} + (\Delta_4+B_4)(\Delta_3+B_3)(\Delta_2+B_2)^{-1}(\Delta_1+B_1) - (\Delta_6+B_6)(\Delta_5+B_5)^{-1}]$. 
}
\label{fig:Diff}
\end{figure}

\medskip

\noindent{\it $2$-cells:}  The differential for generators $c_{i,j}^\alpha$ associated to a $2$-cell, $e^2_\alpha$, are characterized by the formula
\begin{equation} \label{eq:dC}
\partial C = J_\alpha\cdot\left[ A_{v_1} C + CA_{v_0} + (\Delta_j+B_j)^{\eta_j} \cdots(\Delta_1 +B_1)^{\eta_1} - (\Delta_m +B_m)^{\eta_m} \cdots (\Delta_{j+1} + B_{j+1})^{\eta_{j+1}}\right].
\end{equation}
where the notations are as follows:  

\begin{itemize}
\item The $C$ and $J_\alpha$ are  defined as usual.

\item The $A_{v_0}$ and $A_{v_1}$ are boundary matrices formed from the generators associated to the initial and terminal vertices $v_0$ and $v_1$ of $e^2_\alpha$ by the same procedure that was used to form $A_\pm$ from (\ref{eq:dB}) but with $e^2_\alpha$ playing the role of $e^1_\alpha$.

\item Label the $1$-cells that appear around the boundary of $e^2_\alpha$ (as parametrized by a characteristic map $D^2 \rightarrow \overline{e^2_\alpha}$) in order along the counter-clockwise path from $v_0$ to $v_1$ as $e_1,\ldots, e_j$ and those along the clockwise path as $e_{j+1},\ldots, e_m$. For $1 \leq i \leq m$, the matrix $B_i$ is formed by placing 
the $b^i_{pq}$ generator in position $k,l$ when $S^i_p \subset \overline{S^\alpha_{k}}$ and $S^i_q \subset \overline{S^\alpha_{l}}$.  Note that in contrast to the procedure used when forming the $A_{v_1}$ and $A_{v_0}$  boundary matrices we do {\it not} include additional entries of $1$ when pairs of sheets above $e^2_\alpha$ meet at cusp edges above $e_i$.

\item The $\Delta_{i}$ is formed as in the matrix $\Delta$ above using a small shift of $e_i$ into the interior of $e^\alpha_2$ 
and then taking its lift to each of the sheets $S^\alpha_1, \ldots, S^\alpha_n$ before intersecting with $\xi$ and $\Gamma$ to form the diagonal entries.  
\item The exponent $\eta_i$ is $+1$ or $-1$ depending on whether the orientation of $e_i$ points from $v_0$ to $v_1$ or from $v_1$ to $v_0$.
\end{itemize}

See Figure \ref{fig:Diff}.

\begin{remark}\label{rem:diffs}
\begin{enumerate}
\item 	The case where $v_0=v_1$ is allowed and convenient in several of our later calculations.  In this case, a choice is made for a path around $\partial \overline{e^2_\alpha}$ from $v_0$ to $v_1$ and the other path is considered to be constant (so that one of the two products involving $B_i$'s in (\ref{eq:dC}) is just the identity matrix $I$).
\item It is crucial to distinguish between the matrix $A$ used in (\ref{eq:dA}) to define $\partial a_{i,j}^\alpha$ for generators associated to $0$-cell $e^0_\alpha$ and the boundary matrices $A_\pm, A_{v_0}$ and $A_{v_1}$ that arise when $e^0_\alpha$ appears in the boundary of a neighboring $1$- or $2$-cell, $e'$.  The boundary matrices may differ from the original matrix $A$ by (i) interchanging rows and columns to reflect the difference between the choice of numbering of sheets above $A$ and the numbering of their continuous extensions above $e'$, and by (ii) the presence of additional rows and columns in the boundary matrices that correspond to sheets above the $e'$ that limit to cusp points above $e^0_\alpha$.  Consequently, the boundary matrices can be formed from $A$ by first conjugating $QA Q^{-1}$ where $Q$ is an appropriate permutation matrix, and then taking a direct sum with some number of $2\times 2$ matrices of the form $N = \left[\begin{array}{cc} 0  & 1 \\ 0 & 0 \end{array}\right]$ where the positioning of the summands is determined by the location of the cusp sheets above $e'$. 

A parallel remark applies to the matrices $B$ and $\Delta$ from (\ref{eq:dB}) and the matrices $B_i$ and $\Delta_i$ from (\ref{eq:dC}), though with different $2\times 2$-blocks  playing the role of $N$. 
\end{enumerate}
\end{remark}

The following theorem is from\footnote{The referenced article \cite{RuSu5} works with the fully non-commutative version of the DGA where the $R$ coefficients do not commute with Reeb chords.} \cite{RuSu5}; see also  \cite{RuSu1,RuSu2, RuSu25} for the case of $\Z/2$-coefficients. For the given collection of geometric co-cycles $\Gamma=\{\Gamma_1, \ldots, \Gamma_\ell\}$ on $\Lambda$, we let $-\Gamma = \{-\Gamma_1, \ldots, -\Gamma_\ell\}$ denote the collection resulting from reversing the co-orientation of each of the $\Gamma_i$.  

\begin{theorem} \label{thm:CWmulti}  For any $\mathcal{E}$ and $\xi$, there exists a choice of spin structure on $\Lambda$ such that there is a stable tame isomorphism over $R= \Z[t_1^{\pm1}, \ldots, t_\ell^{\pm1}]$
	\[
	\algCW(\Lambda, \Gamma, \mathcal{E}, \xi) \cong \alg(\Lambda, -\Gamma)
	\]
	where $\alg(\Lambda, -\Gamma)$ denotes the multi-curve LCH DGA.
\end{theorem}

In particular, in view of Proposition \ref{prop:stableinv} and Proposition \ref{prop:augInvariance} the augmentation numbers from $\aug_{\Lambda}(q)$ can be computed using $\algCW(\Lambda, \Gamma, \mathcal{E}, \xi)$ in place of $\alg(\Lambda, \Gamma)$ where $\Gamma$ is chosen so that $[\Gamma_1], \ldots, [\Gamma_\ell]$ spans $H^1(\Lambda)$.  Note that the $[-\Gamma_i]$ will then span $H^1(\Lambda)$ as well.

\begin{remark}
	The need for using $-\Gamma$ rather than $\Gamma$ on the right side of the isomorphism from Theorem \ref{thm:CWmulti} can be avoided by modifying the definition of the cellular DGA to replace all occurrences of the matrices $\Delta$ and $\Delta_i$ in  (\ref{eq:dB}) and (\ref{eq:dC}) with $\Delta^{-1}$ and $\Delta^{-1}_i$, respectively.  This is the approach taken in \cite{RuSu5}.  In the present article, using $\Delta$ and $\Delta_i$ rather than their inverses is convenient for simplifying considerations in later calculations.
\end{remark}

\subsubsection{Triangular property of the cellular DGA}

When working with the cellular DGA, we will often apply the Propostion \ref{prop:cancel} which requires a suitable ordering of generators to make $\mathcal{A}^{CW}(\Lambda)$ triangular (as in Definition \ref{def:DGAs}).  We provide here a construction of many such orderings.

\begin{construction} \label{const:order}  For each $d = 0,1,2$, choose a total ordering, $\stackrel{d}{<}$ of the $d$-dimensional cells of $\mathcal{E}$.  Define a partial ordering $\prec$ of the generators of $\mathcal{A}^{CW}(\Lambda, \mathcal{E})$ by declaring that for generators $x^\alpha_{i,j}$ and $x^\beta_{i',j'}$ associated to cells $e^d_\alpha$ and $e^{d'}_\beta$ we have
\[
 x^\alpha_{i,j} \prec x^\beta_{i',j'}
\]
if and only if one of the following conditions holds
\begin{itemize}
\item[(i)]  $d < d'$, or
\item[(ii)] $d=d'$ and $e^d_\alpha \stackrel{d}{<} e^{d}_\beta$, or
\item[(iii)] $e^d_\alpha = e^{d'}_\beta$ and $|i-j| < |i'-j'|$.
\end{itemize}
Finally, extend $\prec$ to a total ordering.
\end{construction}

\begin{proposition}  \label{prop:triangleCW}
The cellular DGA is a triangular DGA  with respect to any ordering of the generating set arising from Construction \ref{const:order}.
\end{proposition} 

\begin{proof}
From (\ref{eq:dA}), (\ref{eq:dB}), (\ref{eq:dC}), we see that the differential of a generator $x^\alpha_{i,j}$ associated to a cell $e^d_\alpha$ involves only generators that either (i) are associated to other cells of lesser dimension, or (ii) are associated to $e^d_\alpha$ but have strictly smaller $|i-j|$.  [The reason for (ii) is that the matrices $A$, $B$, and $C$ from (\ref{eq:dA}), (\ref{eq:dB}), (\ref{eq:dC}) are strictly upper triangular and are always multiplied by at least one other strictly upper triangular matrix when they appear on the right hand sides of these equations.]
\end{proof}

\begin{remark} \label{rem:reorder}
\begin{enumerate}
\item The requirement (ii) is not necessary to prove triangularity but can be  useful when applying Proposition \ref{prop:cancel}.
\item When a generator $x$ satisfies $\partial x =0$, modifying the ordering of generators to make $x$ the smallest element maintains the  triangularity property.
\end{enumerate} 
\end{remark}

\section{Proof of the index $0/2$ skein relation}  \label{sec:5}

In this section we establish the proof of (SR1), aka the ``index 0/2 skein relation'' in reference to the Morse index of the Reeb chords (viewed as critical points of local difference functions) that appear on the LHS of the equation.  In outline, after equipping each of the four Legendrians with a collection of geometric cocycles (Section \ref{sec:cocycles}), a combinatorial spin structure, and a compatible polygonal decomposition, we compute the cellular DGAs in Section \ref{sec:5-2}.  In Section \ref{sec:5-3}, we establish stable tame isomorphisms between various specializations of these DGAs (Proposition \ref{prop:specialiso}), and use them to complete the proof of (SR1).

\medskip

Let $\Lambda_1, \Lambda_2, \Lambda_3, \Lambda_4 \subset J^1S$ denote four Legendrian surfaces corresponding to the four terms of a skein relation (SR1) from Theorem \ref{thm:main} numbered in the order they appear in the relation.  The front projections of the $\Lambda_a$, $1 \leq a \leq 4$ are assumed to agree outside of the pictured $3$-ball $B \subset S\times \R$.  As well, the $\Z$-valued Maslov potentials of $\Lambda_1$ and $\Lambda_2$ are assumed to agree outside of $B$ and to take the values $r$ and $s$ in $\pi_{xz}^{-1}(B)$ as pictured in (SR1). We will only consider $\Lambda_3$ and $\Lambda_4$ when $r=s+1$, and in this case $\Lambda_3$ and $\Lambda_4$ can (and are assumed to) be uniquely equipped with Maslov potentials agreeing with those of $\Lambda_1$ and $\Lambda_2$ outside of $\pi_{xz}^{-1}(B)$.

\subsection{Cocycles}  \label{sec:cocycles}

For proving the relation (SR1) we will specify for each $\Lambda_a$, $1\leq a \leq 4$, a collection of geometric cocycles $\Gamma_{a}$ spanning $H^1(\Lambda_a)$ and work with the multi-curve cellular DGA from Theorem \ref{thm:CWmulti}.  For the computation of $H^1(\Lambda_3)$ and $H^1(\Lambda_4)$ and at other points later in the section, we will need to distinguish between the following two cases.

\medskip

\noindent {\bf Connected Case.}  The two disks  of $\Lambda_1$ pictured in (SR1) belong to the {\it same} connected component of $\Lambda_1$.

\medskip

\noindent {\bf Disconnected Case.}  The two disks of $\Lambda_1$ pictured in (SR1) belong to {\it different} connected components of $\Lambda_1$.

\medskip

We have
\begin{align*} & \rk(H^1(\Lambda_1)) = \rk(H^1(\Lambda_2)), \\
& \rk(H^1(\Lambda_3)) = \rk(H^1(\Lambda_4)) = \left\{\begin{array}{cl} \rk(H^1(\Lambda_1)) +2, & \mbox{Connected Case} \\ \rk(H^1(\Lambda_1)), & \mbox{Disconnected Case.}
\end{array}\right.
\end{align*}

The $\Gamma_a$ are as follows:  Start with a common choice $\Gamma$ of co-oriented multi-curves within $\Lambda_1 \setminus \pi_{xz}^{-1}(B) = \Lambda_2 \setminus \pi_{xz}^{-1}(B)$ that
\begin{itemize}
	\item[(i)] span $H^1(\Lambda_1)$ and $H^1(\Lambda_2)$, and
	\item[(ii)] have their base projections disjoint from $\pi_x(B)$.
\end{itemize}
[This is possible since $\Lambda_1 \cap \pi^{-1}_{x}(\pi_x(B))$ is a union of disks.] Let $\mu_a$ be an additional cocycle on each $\Lambda_a$ running parallel to the upper boundary circle of $\Lambda_a \cap \pi_{xz}^{-1}(B)$ that is visible in (SR1), but shifted slightly into $\Lambda_a \cap \pi_{xz}^{-1}(B)$ with co-orientation pointing into $\Lambda_a \cap \pi_{xz}^{-1}(B)$.  We then put
\[
\Gamma_1 = \Gamma \cup \{\mu_1\} \quad \mbox{and} \quad \Gamma_2 = \Gamma \cup \{\mu_2\};
\] 
see Figures \ref{fig:CharMap1} and \ref{fig:CharMap2}.

In the Connected Case, we choose for both $\Lambda_3$ and $\Lambda_4$ additional cocycles $\lambda_3$ and $\lambda_4$ that intersect $\Lambda_a \cap \pi^{-1}_{xz}(B)$ as pictured
in Figures \ref{fig:CharMap2} and \ref{fig:CharMap1} respectively and agree with one another outside of $\pi^{-1}_{xz}(B)$.  We then have spanning sets for $H^1(\Lambda_3)$ and $H^1(\Lambda_4)$, 
\[
\Gamma_3 = \left\{\begin{array}{cl} \Gamma\cup\{\mu_3, \lambda_3\}, & \mbox{Connected Case} \\ \Gamma\cup\{\mu_3\}, & \mbox{Disconnected Case}\end{array}\right. \quad \mbox{and} \quad \Gamma_4 = \left\{\begin{array}{cl} \Gamma\cup\{\mu_4, \lambda_4\}, & \mbox{Connected Case} \\ \Gamma\cup\{\mu_4\}, & \mbox{Disconnected Case.} \end{array}\right.
\]

\begin{remark}  The co-orientations of $\lambda_3$ and $\lambda_4$ agree at their upper endpoints in Figures \ref{fig:CharMap1} and \ref{fig:CharMap2}.  To see this, one should note that a continuous normal vector field to $\lambda_3$ within $\Lambda_3$ will appear to reverse in the front projection of $\Lambda_3$ when it passes through the cone point (whose pre-image in $\Lambda_3$ is an $S^1$).  See also the appearance of $\lambda_3$ when the cone point of $\Lambda_3$ is perturbed in Section \ref{sec:Cone} below.

\end{remark}

\subsection{Cellular DGAs of the four Legendrians}  \label{sec:5-2}
Next, we fix compatible polygonal decompositions and combinatorial spin structures for each of the $\Lambda_a$, $1\leq a \leq 4$.  We then form cellular DGAs for the $\Lambda_a$, and in some cases pass to stable tame isomorphic quotients in order to allow for an easier comparison of augmentation numbers.  The resulting DGAs will be denoted as simply $\alg(\Lambda_1), \alg(\Lambda_2), \alg(\Lambda_3),$ and $\alg(\Lambda_4)$.

\subsubsection{Combinatorial spin structures} First, fix such a structure $\xi_1$ for $\Lambda_1$ that has its base projection disjoint from $\pi_{x}(B)$.  We then can put $\xi_a = \xi_1$  for $a=1, 2, 4$ to obtain combinatorial spin structures for $\Lambda_1, \Lambda_2,$ and $\Lambda_4$.  The $\Lambda_3$ contains a cone point in its front projection which is a non-generic singularity that needs to be resolved before the cellular DGA can be applied.  This is done in the Section \ref{sec:Cone} where a combinatorial spin structure $\xi_3$ is introduced that agrees with $\xi_1$ outside of $\Lambda_3 \cap \pi_{xz}^{-1}(B)$.

\medskip

\subsubsection{The DGA $\alg(\Lambda_1)$}  
 We begin with a compatible polygonal decomposition $\mathcal{E}_1$ for $\Lambda_1$ that has the form pictured in Figure \ref{fig:CharMap1} within $\pi_x(B)$, and we consider the cellular DGA $\alg^{CW}(\Lambda_1, \Gamma_1, \mathcal{E}_1, \xi_1)$.  Here and henceforth, we will often simply label cells by their corresponding matrices of generators.  Note that the base projection of the visible circle of crossings of $\pi_{xz}(\Lambda_1)$ is the union of the cells labeled with $A_1$ and $B_2$;  the arrows pointing from the $A_1$ and $B_2$ cells into $C_1$ indicate that we number the $A_1$ and $B_2$ sheets in the order they appear above $C_1$, cf. Remark \ref{rem:sheetnum} (iii).  Recall that this numbering of sheets above each cell specifies the subscripts in the notations for the generators;  we will take the superscript of generators to match the subscript of the matrix they belong to.  E.g., the generators associated to cells labeled $A_0$ and $C_2$ will be respectively labeled as $a^0_{i,j}$ and $c^2_{i,j}$.

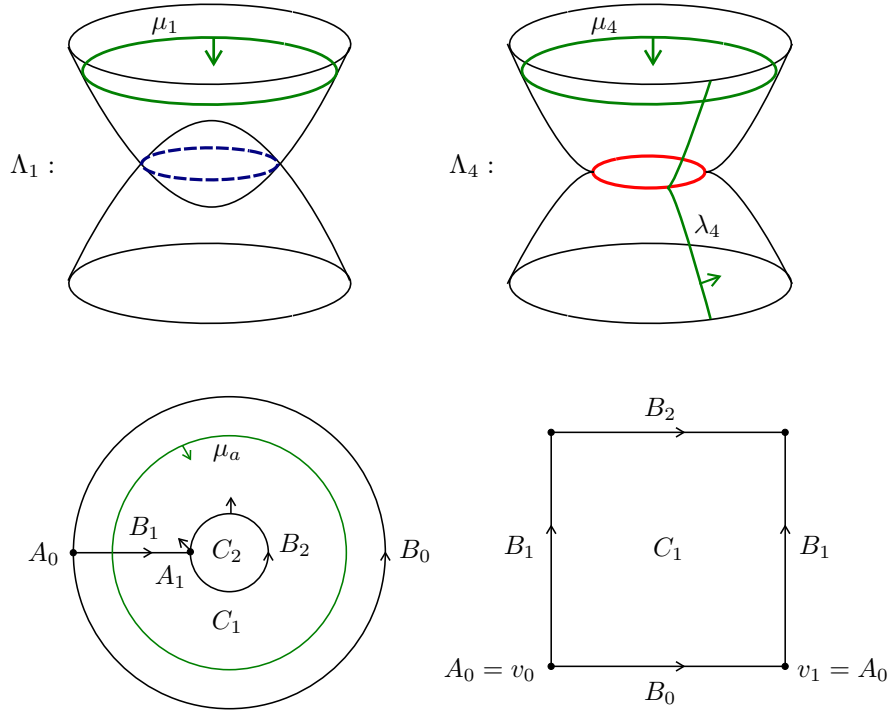
\begin{figure}
	
\labellist
\small
\pinlabel $\mu_1$ [b] at 46 132
\pinlabel $\Lambda_1:$ [r]  at -2 74
\endlabellist
\centerline{\includegraphics[scale=.8]{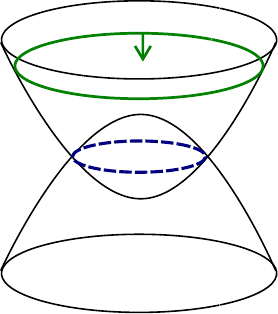} \quad \quad \quad \quad \quad
	\labellist
	\small
	\pinlabel $\mu_4$ [b] at 46 132
	\pinlabel $\Lambda_4:$ [r] at -2 74
	\pinlabel $\lambda_4$ [l] at 84 46
	\endlabellist
	\raisebox{0cm}{\includegraphics[scale=.8]{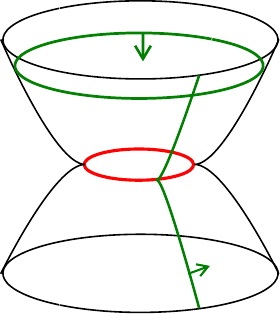}}
} 
	
	\quad
	
	\quad
	
\labellist
\small
\pinlabel $B_0$ [l] at 310 152
\pinlabel $C_1$  at 152 84
\pinlabel $A_0$ [r] at -2 150
\pinlabel $A_1$ [tr] at 118 144
\pinlabel $B_1$ [b] at 72 156
\pinlabel $C_2$ at 152 152
\pinlabel $B_2$ [l] at 194 158
\pinlabel $\mu_a$ [t] at 152 262
\endlabellist
\centerline{\includegraphics[scale=.4]{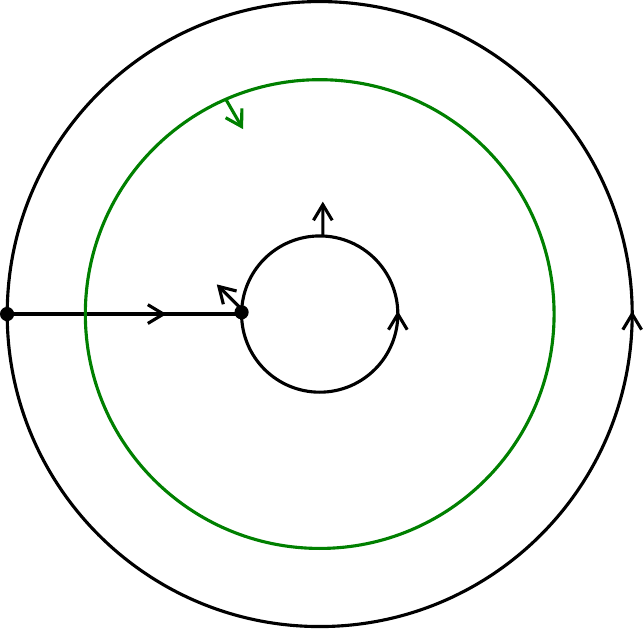} \quad \quad \quad \quad \quad
\labellist
\small
\pinlabel $B_0$ [t] at 110 -2
\pinlabel $B_1$ [l] at 236 120
\pinlabel $B_1$ [r] at -2 120
\pinlabel $B_2$ [b] at 110 234
\pinlabel $C_1$ at 118 118
\pinlabel $A_0=v_0$ [r] at -2 2
\pinlabel $v_1=A_0$ [l] at 234 2
\endlabellist
\raisebox{.5cm}{\includegraphics[scale=.4]{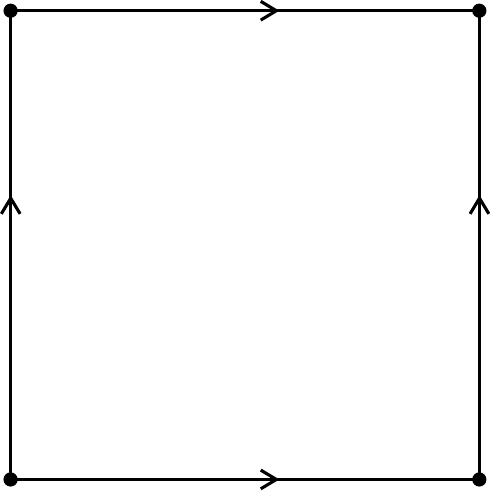}}
} 

\caption{(Top)  Front projections for $\Lambda_1$ and $\Lambda_4$ with the co-oriented curves $\mu_1$ from $\Gamma_1$ and  $\mu_4$ and $\lambda_4$ from $\Gamma_4$ pictured in green.  (Bottom left)  The decomposition $\mathcal{E}_a$ for $a=1$ or $4$.  
(Bottom right) The domain of the characteristic map for $C_1$ with initial and terminal vertices, $v_0$ and $v_1$, labeled.
}
\label{fig:CharMap1}
\end{figure}

Let $n$ denote the common number of sheets of $\Lambda_1$ that appear above any of the cells pictured in Figure \ref{fig:CharMap1}; let $k$ and $k+1$ be the indices of the two sheets above $A_1$ and $B_2$ that cross in $\pi_{xz}(\Lambda_1)$.  Each of the matrices $A_0, A_1, B_0, B_1, B_2, C_1, C_2$ is strictly upper-triangular, and, with two exceptions,  a unique generator of $\alg^{CW}(\Lambda_1, \Gamma_1, \mathcal{E}_1, \xi_1)$ appears in each $i,j$-entry with $1\leq i< j \leq n$, e.g., the $i,j$ entry of $A_0$ is $a^0_{i,j}$ if $i<j$ and $0$ if $i\geq j$. The two exceptions are that the $k,k+1$-entries of $A_1$ and $B_2$ are $0$.  Indeed, no generator of the form $a^1_{k,k+1}$  exists since the sheets $S^{A_1}_{k}$ and $S^{A_1}_{k+1}$ share the same $z$-coordinate above $A_1$ due to the crossing arc in $\pi_{xz}(\Lambda_1)$.  For the same reason, $b^2_{k,k+1}$ does not exist.

Next, we pass from $\alg^{CW}(\Lambda_1, \Gamma_1, \mathcal{E}_1, \xi_1)$ to a stable tame isomorphic quotient that we will denote as simply $\mathcal{A}(\Lambda_1)$.

\begin{proposition} \label{prop:ALambda1}  The cellular DGA $\alg^{CW}(\Lambda_1, \Gamma_1, \mathcal{E}_1, \xi_1)$ has a stable tame isomorphic quotient, $\mathcal{A}(\Lambda_1)$, satisfying:
\begin{itemize}
\item All $B_2$ generators are removed from the generating set.
\item All $C_2$ generators except for $c := c_{k,k+1}^2$ are removed from the generating set, and $\partial c=0$.
\item There is no generator of the form $a^1_{k,k+1}$.
\item In the quotient the differentials satisfy
\begin{align}
J_{A_1} \cdot \partial A_1 &= A_1^2 \label{eq:JA1Lambda1}  \\
J_{B_1} \cdot \partial B_1 &= A_1(\Theta^{\mu_1}_k +B_1)- (\Theta^{\mu_1}_k +B_1)A_0 \label{eq:JB1Lambda1} \\
J_{C_1} \cdot \partial C_1 & = A_0C_1+C_1A_0 + (I+B_0) \label{eq:JC1Lambda1}  \\ & \quad \quad -(\Theta^{\mu_1}_k+B_1)^{-1}\left(I+ A_1(cE_{k+1,k}) + (cE_{k+1,k})A_1\right)(\Theta^{\mu_1}_k+B_1). \notag  
\end{align}
where $\Theta^{\mu_1}_k$ is the diagonal matrix with $k$-th diagonal entry  $\mu_1$ and other diagonal entries $1$. 
\item All other generators remain and their differentials are the same as in $\alg^{CW}(\Lambda_1, \Gamma_1, \mathcal{E}_1, \xi_1)$.
\end{itemize}
\end{proposition}

\begin{proof}

We will apply Proposition \ref{prop:cancel} to 
$\alg^{CW}(\Lambda_1, \Gamma_1, \mathcal{E}_1, \xi_1)$ with the goal of eliminating all of the $B_2$ and $C_2$ generators except for $c^2_{k,k+1}$.  To this end, with $Q=Q_{(k\,k+1)}$ denoting the permutation matrix of the transposition $(k\,k+1)$ and making use of the identity $(J_{C_2})^2=I$, the equation (\ref{eq:dC}) leads to
\begin{align*}
J_{C_2} \cdot \partial C_2 = (QA_1Q)C_2 + C_2 (Q A_1Q) + I -(I + QB_2Q).
\end{align*}
The matrices $A_1$ and $B_2$ appear conjugated by $Q$ since the $S_{k}$ and $S_{k+1}$ sheets above $A_1$ and $B_2$ belong to the closures of the sheets  $S^{C_2}_{k+1}$ and $S^{C_2}_k$ above $C_2$ respectively, cf. Remark \ref{rem:diffs} (2).
Now, as $a^1_{k,k+1}$ and $b^2_{k,k+1}$ do not exist, the conjugated matrices $QA_1Q$ and $QB_2Q$ are themselves upper triangular with $k,k+1$-entry of $0$.  Thus, writing\footnote{To help the reader follow the arguments, we place a pair of boxes around generators that we will intend to cancel using Proposition \ref{prop:cancel}.}
\begin{equation} \label{eq:100}
J_{C_2} \cdot \partial \framebox{$C_2$} = -Q\framebox{$B_2$}Q + (QA_1Q)C_2 + C_2 (Q A_1Q)
\end{equation}
shows that $\partial c^2_{k,k+1}=0$.   Following Proposition \ref{prop:triangleCW} and  Remark \ref{rem:reorder}, there exists an ordering of the generators of $\alg^{CW}(\Lambda_1, \Gamma_1, \mathcal{E}_1, \xi_1)$ for which $\partial$ is triangular such that $ c^2_{k,k+1}$ is smallest and is followed immediately by the $A_1$ generators.  
Moreover, for each $i<j$ with $(i,j) \neq (k,k+1)$ we can write
\[
\iota \cdot \partial \framebox{$c^2_{i,j}$} = -\framebox{$b^2_{\sigma(i),\sigma(j)}$} + v+ w
\]
where $\iota \in \{\pm1\}$;  $w$ consists of those terms from $(QA_1Q)C_2 + C_2 (Q A_1Q)$ that involve the generator $c^2_{k,k+1}$; and $v$ consists of all the remaining terms.  Note that due to the strict upper triangularity of $QA_1Q$ and $C_2$, the $(i,j)$-entry of the product $(QA_1Q)C_2$ (resp. $C_2 (Q A_1Q)$) involves only entries from $C_2$ of the from $(i',j)$ with $i<i'<j$ (resp. of the form $(i,j')$ with $i<j'<j$).  Therefore, we are able to apply Proposition \ref{prop:cancel} by taking the $y_1,\ldots, y_r$ and $z_1, \ldots, z_r$ appearing in the statement of Proposition \ref{prop:cancel} to be the $\{c_{i,j} \,|\, i<j, (i,j) \neq (k,k+1)\}$ and $\{b_{\sigma(i),\sigma(j)}\,|\, i<j, (i,j) \neq (k,k+1)\}$ ordered with $|i-j|$ non-decreasing.
With the ordering of generators of $\alg^{CW}(\Lambda_1, \Gamma_1, \mathcal{E}_1, \xi_1)$ described above $v$ and $w$ indeed satisfy the requirements of the proposition.  We take the resulting quotient dg-algebra of $\alg^{CW}(\Lambda_1, \Gamma_1, \mathcal{E}_1, \xi_1)$ to be $\alg(\Lambda_1)$.  

The formulas for the differentials (\ref{eq:JA1Lambda1}) and (\ref{eq:JB1Lambda1}) 
are satisfied already in $\alg^{CW}(\Lambda_1, \Gamma_1, \mathcal{E}_1, \xi_1)$; see (\ref{eq:dA}) and (\ref{eq:dB}) above.  
For verifying (\ref{eq:JC1Lambda1}), 
we note that $\partial C_2 \doteq 0$ and $C_2 \doteq cE_{k,k+1}$
where $\doteq$ is used to indicate equality of classes in the quotient algebra, $\alg(\Lambda_1)$.  Thus,  using that $Q^2=I$, (\ref{eq:100}) shows that
\begin{align*}
B_2 & \doteq A_1 QC_2Q + QC_2QA_1 \\
 & \doteq A_1(cE_{k+1,k}) + (c E_{k+1,k})A_1 \\
 & = \sum_{i<k} a^1_{i,k+1}cE_{i,k} + \sum_{k+1<j} c a^1_{k,j} E_{k+1,j},
\end{align*}
and (\ref{eq:JC1Lambda1}) follows from 
substituting this into the identity
\[
J_{C_1} \cdot \partial C_1 = A_0C_1+C_1A_0 + (I+B_0) -(\Theta^{\mu_1}_k+B_1)^{-1}\left(I+B_2\right)(\Theta^{\mu_1}_k+B_1). 
\]
that holds in $\alg^{CW}(\Lambda_1, \Gamma_1, \mathcal{E}_1, \xi_1)$ from (\ref{eq:dC}), minding that $\Delta_1 = \Theta^{\mu_1}_k$ and  $\Delta_0 =\Delta_2 = 1$ by our construction of $\Gamma_1$.
\end{proof}

\begin{remark}  Throughout this section as well as in Section \ref{sec:Cone}, we continue using the box notation to indicate pairs of generators to be canceled, e.g., writing $\partial \framebox{$y$} = \framebox{$z$} + v +w$ indicates that an application of Proposition \ref{prop:cancel} will be used to remove $y$ and $z$ from the generating set.  In addition, we maintain the notation $\doteq$ for equality of classes in a quotient DGA.
\end{remark}

\begin{remark}
Morally, 
the quotienting procedure used to pass from $\alg^{CW}(\Lambda_1, \Gamma_1, \mathcal{E}_1, \xi_1)$ to $\mathcal{A}(\Lambda_1)$ is analogous to modifying the CW decomposition by the simple homotopy equivalence collapsing $C_2 \cup B_2$ to a point.  However, due to the presence of the crossing arc, in the quotient DGA some remnant of $C_2$ remains via the relation $B_2 \doteq A_1(cE_{k+1,k}) + (cE_{k+1,k})A_1$.
\end{remark}

\subsubsection{The DGA $\alg(\Lambda_2)$}  A compatible polygonal decomposition, $\mathcal{E}_2$, for $\Lambda_2$ is introduced in Figure \ref{fig:CharMap2}.  We set $\alg(\Lambda_2) = \alg^{CW}(\Lambda_1, \Gamma_1, \mathcal{E}_1, \xi_1)$ and record some of its properties.

\begin{figure}
		\labellist
	\small
	\pinlabel $\mu_2$ [b] at 46 132
	\pinlabel $\Lambda_2:$ [r]  at -2 74
	\endlabellist
	\centerline{\includegraphics[scale=.8]{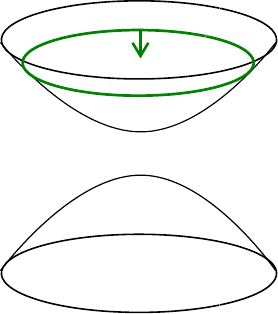} \quad \quad \quad \quad \quad
		\labellist
		\small
		\pinlabel $\mu_3$ [b] at 46 132
		\pinlabel $\Lambda_3:$ [r] at -2 74
		\pinlabel $\lambda_3$ [l] at 52 22
		\endlabellist
		\raisebox{0cm}{\includegraphics[scale=.8]{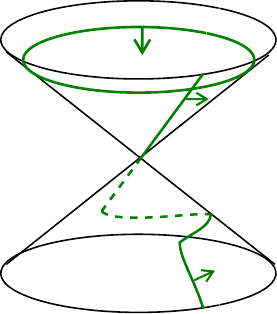}}
	} 
	
	\quad
	
	\quad
	\labellist
	\small
	\pinlabel $B_0$ [l] at 310 152
	\pinlabel $C_1$  at 152 84
	\pinlabel $A_0$ [r] at -2 150
	\pinlabel $A_1$ [t] at 152 144
	\pinlabel $B_1$ [b] at 84 156
	\pinlabel $\mu_a$ [t] at 152 262
	\endlabellist
	\centerline{\includegraphics[scale=.4]{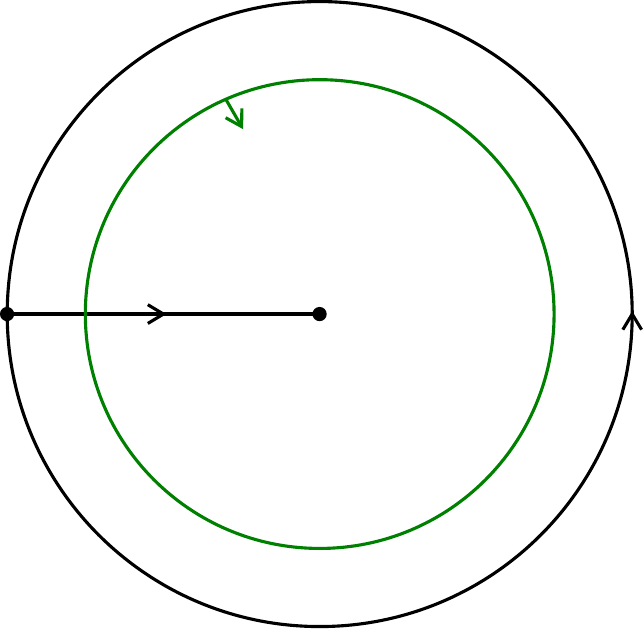} \quad \quad \quad \quad \quad
		\labellist
		\small
		\pinlabel $B_0$ [t] at 110 -2
		\pinlabel $B_1$ [l] at 180 120
		\pinlabel $B_1$ [r] at 50 120
		\pinlabel $C_1$ at 118 80
		\pinlabel $A_0=v_0$ [r] at -2 2
		\pinlabel $v_1=A_0$ [l] at 234 2
		\endlabellist
		\raisebox{.5cm}{\includegraphics[scale=.4]{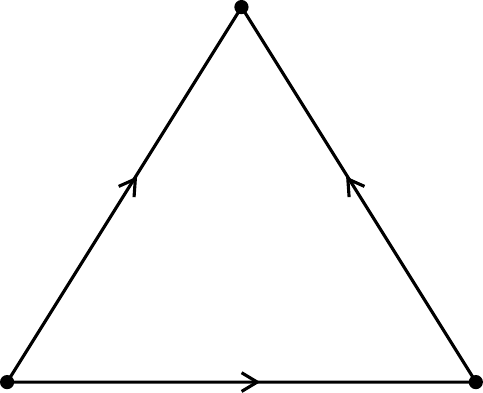}}
	} 
	
	\caption{(Top)  Front projections for $\Lambda_2$ and $\Lambda_3$ with the co-oriented curves $\mu_2$ from $\Gamma_2$, and $\mu_3$ and $\lambda_3$ from $\Gamma_3$ pictured in green.  Note that the co-orientation of $\lambda_3$, while continuous in $\Lambda_3$ itself, appears to reverse in the front projection when passing through the cone point singularity.  In particular, near the boundary the coorientation of $\lambda_3$ is consistent with that of the curve $\lambda_4$ from Figure \ref{fig:CharMap1}.  (Bottom left)  The decomposition $\mathcal{E}_a$ for $a=2$ or $3$.  
(Bottom right) The domain of the characteristic map for $C_1$ with initial and terminal vertices, $v_0$ and $v_1$, labeled.
	}
	\label{fig:CharMap2}
\end{figure}
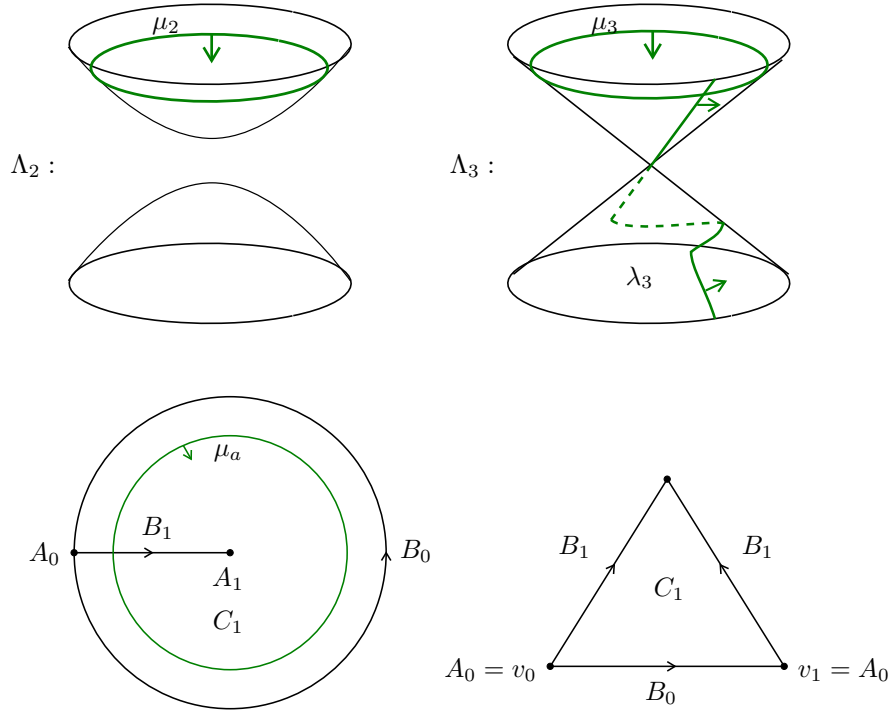

\begin{proposition}  \label{prop:ALambda2}  The generators and differentials of $\alg(\Lambda_2)$ coincide with the those of $\alg(\Lambda_1)$ with the following exceptions:
	\begin{itemize}
		\item The generator $a:= a^1_{k,k+1}$ exists in $\alg(\Lambda_2)$.
		\item The generator $c = c^2_{k,k+1}$ does not exist in $\alg(\Lambda_2)$.
	\item The differentials satisfy
\begin{align}
	J_{A_1} \cdot \partial A_1 &= A_1^2 \label{eq:2JA1Lambda1}  \\
	J_{B_1} \cdot \partial B_1 &= A_1(\Theta^{\mu_2}_k +B_1)- (\Theta^{\mu_2}_k +B_1)A_0 \label{eq:2JB1Lambda1} \\
	J_{C_1} \cdot \partial C_1 & = A_0C_1+C_1A_0 + (I+B_0) - I  \label{eq:2JC1Lambda1}
\end{align}
where $\Theta^{\mu_2}_k$ is the diagonal matrix with $k$-th diagonal entry  $\mu_2$ and other diagonal entries $1$.
\end{itemize}
In particular, differentials for the generators associated to all other cells ($A_0$, $B_0$, and all cells outside of the pictured disk) coincide with differentials from $\mathcal{A}(\Lambda_1)$.   
\end{proposition}

\subsubsection{The DGA $\alg(\Lambda_3)$}
Due to the presence of the non-generic cone point singularity, we  cannot immediately employ the version of the cellular DGA described above in Section \ref{sec:CWDGA} to $\Lambda_3$.  
An extension of the cellular DGA allowing for cone point front singularities is established Section \ref{sec:Cone}.  Using this extension with the cellular decomposition $\mathcal{E}_3$ as in Figure \ref{fig:CharMap2} produces a DGA $\alg(\Lambda_3)$ as follows. 

\begin{proposition}  \label{prop:ALambda3}  The generators and differentials of $\alg(\Lambda_3)$ have the following properties:
	\begin{itemize}
		\item The generators of $\alg(\Lambda_3)$ coincide with those of $\alg(\Lambda_1)$ except that $c = c^2_{k,k+1}$ does not exist in $\alg(\Lambda_3)$.
		\item In particular, the generator $a= a^1_{k,k+1}$ that exists in $\alg(\Lambda_2)$ does not exist in $\alg(\Lambda_3)$.
	\item Write
	\[
	\lambda = \left\{\begin{array}{cl} \lambda_3, & \mbox{Connected Case,} \\ 1, & \mbox{Disconnected Case,} \end{array} \right.
	\]
	\[
	A_1^\lambda = A_1 + (\lambda^{-1}-1)E_{k,k+1},
	\]
	and
	\[
	U_\lambda = 
\Theta^\lambda_{k,k+1}+ \sum_{i <k} a^1_{i,k+1} \lambda E_{i,k} + \sum_{k+1 <j} \lambda a^1_{k,j} E_{k+1,j} 
	\]
	where $\Theta^\lambda_{k,k+1}$ is diagonal with $\lambda$ appearing in positions $k$ and $k+1$ and $1$ in all other diagonal entries.
	The differentials satisfy
	\begin{align}
		J_{A_1} \cdot \partial A_1 &= (A^\lambda_1)^2 \label{eq:3JA1Lambda1}  \\
		J_{B_1} \cdot \partial B_1 &= A_1^\lambda(\Theta^{\mu_3}_k +B_1)- (\Theta^{\mu_3}_k +B_1)A_0 \label{eq:3JB1Lambda1} \\
		J_{C_1} \cdot \partial C_1 & = A_0C_1+C_1A_0 + (\Theta^\lambda_{k,k+1}+B_0)- (\Theta^{\mu_3}_k+B_1)^{-1}U_\lambda(\Theta^{\mu_3}_k+B_1)   \label{eq:3JC1Lambda1}
	\end{align}
	where $\Theta^{\mu_3}_k$ is the diagonal matrix with $k$-th diagonal entry  $\mu_3$ and other diagonal entries $1$.
	\item Differentials for the generators associated to all other cells ($A_0$, $B_0$, and all cells outside of the pictured disk) coincide with differentials from $\mathcal{A}(\Lambda_1)$ after replacing all occurrences of $\lambda_3$ with $1$.  
\end{itemize} 
\end{proposition}

The proof of Proposition \ref{prop:ALambda3} appears in Section \ref{sec:Cone}.

\subsubsection{The DGA $\alg(\Lambda_4)$}  
Finally, we turn to $\Lambda_4$ making use of the polygonal decomposition from Figure \ref{fig:CharMap1}.

\begin{proposition} \label{prop:ALambda4}  The cellular DGA $\alg^{CW}(\Lambda_4, \Gamma_4, \mathcal{E}_4, \xi_4)$ has a stable tame isomorphic quotient, $\mathcal{A}(\Lambda_4)$, with the following properties:
	\begin{itemize}
\item 	The generators of $\alg(\Lambda_4)$ coincide with the those of $\alg(\Lambda_1)$ except that  
	\begin{itemize}
		\item[(i)] the generator $c = c^2_{k,k+1}$ does not exist in $\alg(\Lambda_4)$
		and
		\item[(ii)] there are no generators of the form $a^1_{i,j}$ with $\{i,j\} \cap \{k,k+1\} \neq \emptyset$. 
	\end{itemize}
	\item In particular,
	 the generator $a:= a^1_{k,k+1}$ that exists in $\alg(\Lambda_2)$ does not exist in $\alg(\Lambda_4)$.
	\item 	Let $\widehat{A}_1$ denote the $n\times n$ matrix with $i,j$ entry of $a^1_{i,j}$ when $a^1_{i,j}$ exists as a generator of $\alg(\Lambda_4)$, with $k,k+1$ entry of $1$, and having all other entries $0$.  Write
	\[
	\lambda = \left\{\begin{array}{cl} \lambda_4, & \mbox{Connected Case,} \\ 1, & \mbox{Disconnected Case.} \end{array} \right.
	\]
	The differentials in $\mathcal{A}(\Lambda_4)$
	satisfy
		\begin{align}
			J_{A_1} \cdot \partial\widehat{A}_1 &= \widehat{A}_1^2 \label{eq:4JA1Lambda1}  \\
			J_{B_1} \cdot \partial B_1 &= \widehat{A}_1(\Theta^{\mu_4}_k +B_1)- (\Theta^{\mu_4}_k +B_1)A_0 \label{eq:4JB1Lambda1} \\
			J_{C_1} \cdot \partial C_1 & = A_0C_1+C_1A_0 + (\Theta^{\lambda}_{k,k+1}+B_0)-(\Theta^{\mu_4}_k+B_1)^{-1}\Theta^{\lambda}_{k,k+1}(\Theta^{\mu_4}_k+B_1).  \label{eq:4JC1Lambda1}
		\end{align}
		where $\Theta^{\mu_4}_k$ is the diagonal matrix with $k$-th diagonal entry  $\mu_1$ and other diagonal entries $1$, and $\Theta^\lambda_{k,k+1}$ is diagonal with $\lambda$ appearing in positions $k$ and $k+1$ and $1$ in all other diagonal entries.
\item Differentials for the generators associated to all other cells ($A_0$, $B_0$, and all cells outside of the pictured disk) coincide 
\begin{itemize}
	\item[(i)] with differentials from $\mathcal{A}(\Lambda_3)$ after replacing $\lambda_4$ with $\lambda_3$, and
	\item[(ii)] with  differentials from $\mathcal{A}(\Lambda_2)$ after replacing $\lambda_4$ with $1$.
\end{itemize}
\end{itemize}
\end{proposition}

\begin{proof}
The differential of $C_2$ in $\alg^{CW}(\Lambda_4, \Gamma_4, \mathcal{E}_4, \xi_4)$ is computed from (\ref{eq:dC}) as
\[
J \cdot \partial \framebox{$C_2$} = A_1C_2 + C_2A_1 + \framebox{$B_2$}
\]
where the matrices $A_1,B_2,C_2$ appearing in the formula all have size $(n-2) \times (n-2)$ and have distinct generators appearing in every $i,j$ entry with $i<j$.  We can then apply Proposition \ref{prop:cancel} where the $y_1, \ldots, y_r$ (resp. $z_1, \ldots, z_r$) are the $c_{i,j}^2$ (resp. the $b_{i,j}^2$) ordered with $|i-j|$ non-decreasing; the $v$ term in Proposition \ref{prop:cancel} arises from $A_1C_2+C_2A_1$; and $w=0$.  The resulting stable tame isomorphic quotient has all $C_2$ and $B_2$ generators removed and $C_2 \doteq B_2 \doteq 0$.  It matches the description of $\mathcal{A}(\Lambda_4)$ from the statement of Proposition 5.7 with the caveat that the generators $a^2_{i,j}$ from $\alg^{CW}(\Lambda_4, \Gamma_4, \mathcal{E}_4, \xi_4)$ are renamed as $a^2_{\tau(i), \tau(j)}$ in $\mathcal{A}(\Lambda_4)$ where $\tau(l) = \left\{\begin{array}{cr} l, &  l < k, \\ l+2, & k \leq l. \end{array}\right.$  [This is due to the inconsistency of the numbering of sheets of $\Lambda_4$ above $C_1$ and $A_2$ (with the latter matching the enumeration above $C_2$) due to the presence of the cusp edge.]

The final statement relating differentials of other generators with those of $\alg(\Lambda_2)$ and $\alg(\Lambda_3)$ follows since outside of the pictured disk $\pi_x(B)$ the only difference between the setups $(\Lambda_a, \mathcal{E}_a, \Gamma_a, \xi_a)$, $1 \leq a \leq 4$ is that $\lambda_3$ and $\lambda_4$ exist (in the Connected Case) in $\Gamma_3$ and $\Gamma_4$ respectively (and coincide outside of $B$), but do not exist in $\Gamma_1$ or $\Gamma_2$.
\end{proof}

\subsection{Relating the four DGAs}  \label{sec:5-3}

The skein relation (SR1) is established via a combination of stable tame isomorphisms between various specializations stated in the following Proposition \ref{prop:specialiso}.  We prepare for the statement with some preliminaries.  Recall the Constructions \ref{cons:spec1} and \ref{cons:spec2}, and let $\mathbb{F}$ be a field.

\begin{itemize}
\item In a mild extension of notation, for $a=3$ or $4$ we set
\[
\alg(\Lambda_a)|_{\lambda =1} = \left\{\begin{array}{cc} \alg(\Lambda_a)|_{\lambda_a =1}, & \mbox{Connected Case}, \\
\alg(\Lambda_a), & \mbox{Disconnected Case}; \end{array}\right.
\]
note that this is consistent with the use of $\lambda$ in Propositions \ref{prop:ALambda3} and \ref{prop:ALambda4}.  (Recall that the geometric cocycles $\lambda_3$ and $\lambda_4$ are only defined in the Connected Case; see Section \ref{sec:cocycles}.)

\item By identifying $\mu= \mu_1=\mu_2=\mu_3=\mu_4$ we can view the DGAs $\alg(\Lambda_1)$, $\alg(\Lambda_2)$ as well as any specializations of the form $\alg(\Lambda_3)|_{\lambda =1}$, $\alg(\Lambda_4)|_{\lambda =1}$, $\alg(\Lambda_3)|_{\lambda_3 = \overline{\lambda}}$, or $\alg(\Lambda_4)|_{\lambda_4 = \overline{\lambda}}$ with $\overline{\lambda} \in \mathbb{F}^*$ as being defined over the same coefficient ring, $R = \mathbb{F}[t_1^{\pm1}, \ldots, t_\ell^{\pm1}, \mu^{\pm1}]$.

\item Note that the generators $c \in \alg(\Lambda_1)$ and $a \in \alg(\Lambda_2)$ identified in Propositions \ref{prop:ALambda1} and \ref{prop:ALambda2} have their degrees determined from the values of the Maslov potential indicated in (SR1) as $r$ and $s$ by
\begin{equation} \label{eq:acMaslov}
|c| = s-r+1, \quad |a|= r-s-1.
\end{equation}
In particular, $|a|=|c|=0$ holds if and only if $r=s+1$ so that the RHS of (SR1) can be non-zero.  Recall that we only consider the $\Lambda_3$ and $\Lambda_4$ in this situation (as otherwise, $\Lambda_3$ and $\Lambda_4$ do not inherit a well defined $\Z$-valued Maslov potential from $\Lambda_1$ and $\Lambda_2$.)

\end{itemize}

\begin{proposition} \label{prop:specialiso}   Let $\mathbb{F}$ be a field.  
We have the following semi-tame isomorphisms, $\cong$, and stable semi-tame isomorphisms, $\stackrel{s.t.}{\cong}$, of based DGAs over $R=\mathbb{F}[t_1^{\pm1}, \ldots, t_\ell^{\pm1}, \mu^{\pm1}]$.
\begin{enumerate}
\item  We have
\begin{equation} \label{eq:58i}
\alg(\Lambda_1)|_{c=0} \cong \alg(\Lambda_2)|_{a=0}.
\end{equation}
\item If $|a|=|c| =0$, then
\begin{equation}  \label{eq:58ii}
\forall \, \overline{c} \in \mathbb{F}^*, \quad \alg(\Lambda_1)|_{c=\overline{c}} \cong \alg(\Lambda_3)|_{\lambda=1}
\end{equation}
and
\begin{equation}  \label{eq:58iii}
\forall \, \overline{a} \in \mathbb{F}^*, \quad \alg(\Lambda_2)|_{a=\overline{a}} \stackrel{s.t.}{\cong} \alg(\Lambda_4)|_{\lambda=1}.
\end{equation}
Moreover, in the Connected Case
\begin{equation}  \label{eq:58iv}
\forall \, \overline{\lambda} \in \mathbb{F}^*\setminus\{1\}, \quad \alg(\Lambda_3)|_{\lambda_3=\overline{\lambda}} \stackrel{s.t.}{\cong} \alg(\Lambda_4)|_{\lambda_4=\overline{\lambda}}.
\end{equation}
\end{enumerate}

\end{proposition}

\begin{proof}  Throughout the proof we will use notations $\partial_1, \partial_2, \partial_3$, and $\partial_4$ for the differentials on the DGAs arising from $\alg(\Lambda_1), \alg(\Lambda_2), \alg(\Lambda_3)$, and $\alg(\Lambda_4)$ respectively.
	
	\medskip
	
\noindent 	{\bf  Proof of (\ref{eq:58i}):}  Once these specializations are made, the generating sets of the two dg-algebras (as computed in Propositions \ref{prop:ALambda1} and \ref{prop:ALambda2}) coincide, and the formulas (\ref{eq:JA1Lambda1})-(\ref{eq:JC1Lambda1}) and (\ref{eq:2JA1Lambda1})-(\ref{eq:2JC1Lambda1}) that characterize the differentials of generators agree. E.g., upon specializing $c =0$, the product in the second line of (\ref{eq:JC1Lambda1}) becomes simply $I$.

\medskip

\noindent {\bf  Proof of (\ref{eq:58ii}):} We will  establish an isomorphism of dg-algebras $\Phi:\alg(\Lambda_3)|_{\lambda=1} \rightarrow \alg(\Lambda_1)|_{c = \overline{c}}$ over $\mathbb{F}$ that when restricted to the coefficient rings is the isomorphism 
\begin{equation} \label{eq:phioverlinec}
\phi_{\overline{c}}: \mathbb{F}[t_1^{\pm1}, \ldots, t_\ell^{\pm1}, \mu_3^{\pm1}] \stackrel{\cong}{\rightarrow} \mathbb{F}[t_1^{\pm1}, \ldots, t_\ell^{\pm1}, \mu_1^{\pm1}]
\end{equation}
 that preserves $\F$ and satisfies $\phi_{\overline{c}}(t_i) = t_i$ and $\phi_{\overline{c}}(\mu_3) = \overline{c} \mu_1$.  Set $D= \mbox{diag}(1, \ldots, \overline{c}, 1, \ldots, 1)$ to be the diagonal matrix with $\overline{c}$ in the $k$-th position and $1$'s at all other diagonal entries.  There is then a unique $\F$-algebra homomorphism $\Phi$ that (i) restricts to $\phi_{\overline{c}}$ on $R$, (ii) satisfies the matrix equations (as usual maps are applied to matrices entry-by-entry) 
\begin{equation} \label{eq:PhiA1527}
\Phi(A_1) = D A_1 D^{-1}, \quad \Phi(\Theta_k^{\mu_3}+ B_1) = D(\Theta_k^{\mu_1} + B_1),
\end{equation}
and acts as the identity on all generators not belonging to $A_1$ or $B_1$.  [In defining $\Phi$, it is important to note that $D\Theta_k^{\mu_3} = \phi_{\overline{c}}(\Theta_k^{\mu_3})$.]  It then remains to verify that $\Phi \circ \partial_3 = \partial_1 \circ \Phi$ holds on all generators of $\alg(\Lambda_3)|_{\lambda=1}$ (and hence on all of $\alg(\Lambda_3)|_{\lambda=1}$).  As the generators from $A_1$ or $B_1$ and elements of the coefficient ring involving  $\mu_3$ or $\mu_1$ only appear in the differentials of $A_1, B_1$, and $C_1$, the identity is clear when applied to any generator not belonging to $A_1,B_1,$ or $C_1$.  We are left only to check that $\Phi \circ \partial_3 = \partial_1 \circ \Phi$ when applied to entries of $A_1, B_1,$ and $C_1$.  This is done via routine, but somewhat involved, matrix calculations.  Frequently, we will use that when $M \in \mathit{Mat}_n(\F)$ is a matrix of constants and $X \in \mathit{Mat}_n(\alg(\Lambda_3)|_{\lambda=1})$ the algebra homomorphism property of $\Phi$ and the Leibniz rule for $\partial_a$, $a=1,3$, allow us to make computations such as
\[
\Phi(MX) = M \Phi(X), \quad \Phi(XM) = \Phi(X) M,  \quad \partial_a(MX) = M \partial_a(X), \quad \partial_a(XM) = \partial_a(X)M, \quad \mbox{etc.}
\]

Now, noting that when $\lambda=1$ the matrix $A_1^\lambda$ from (\ref{eq:3JA1Lambda1}) becomes simply $A_1$, we compute
\begin{align*}
J_{A_1} \cdot (\Phi \circ \partial_3)(A_1) & = \Phi( J_{A_1} \cdot \partial_3(A_1) ) = \Phi( (A_1)^2) = (D A_1D^{-1})^2 = D (A_1)^2D^{-1} \\
 & = D(J_{A_1} \cdot \partial_1(A_1))D^{-1} = J_{A_1} \cdot \partial_1(D A_1 D^{-1}) = J_{A_1} \cdot (\partial_1 \circ \Phi)(A_1)
\end{align*}
where at the 6th equality we used that $D$ and $J_{A_1}$ commute since they are both diagonal.  As $J_{A_1}$ is invertible, this shows $(\Phi \circ \partial_3)(A_1) = (\partial_1 \circ \Phi)(A_1)$.

Next, compute from (\ref{eq:PhiA1527}) that
\begin{align*}
J_{B_1} \cdot(\Phi \circ \partial_3)(\Theta_k^{\mu_3} + B_1) & = \Phi\left[ A_1(\Theta_k^{\mu_3} + B_1) - (\Theta_k^{\mu_3} + B_1) A_0 \right] \\
&= (DA_1D^{-1})\big( D(\Theta_k^{\mu_1} + B_1)\big) - \big( D(\Theta_k^{\mu_1} + B_1)\big) A_0 \\
 & = D\left[ A_1 (\Theta_k^{\mu_1} + B_1)- (\Theta_k^{\mu_1} + B_1)A_0 \right] \\
& = D\left[J_{B_1} \cdot \partial_1(\Theta_k^{\mu_1} + B_1) \right] = J_{B_1} \cdot \partial_1( D (\Theta_k^{\mu_1} + B_1)) \\ & = J_{B_1} \cdot (\partial_1 \circ \Phi)(\Theta_k^{\mu_3} + B_1).
\end{align*}
This verifies $\Phi\circ \partial_3 = \partial_1 \circ \Phi$ on all generators from $B_1$.

Turning towards $C_1$, we first observe that in the specialization  $\alg(\Lambda_3)|_{\lambda=1}$ the matrices $U_\lambda$ and $\Theta^\lambda_{k,k+1}$ from (\ref{eq:3JC1Lambda1}) becomes
\[
U_\lambda= I + \sum_{i<k} a^1_{i,k+1}E_{i,k}+ \sum_{k+1<j} a_{k,j}^1 E_{k+1,j} \quad \quad \mbox{and} \quad \quad \Theta^\lambda_{k,k+1}=I.
\]
We can then compute
\begin{align*}
J_{C_1} \cdot (\Phi \circ\partial_3)(C_1) & = \Phi\bigg[A_0C_1+C_1A_0 +(I+B_0) \\ & \quad -(\Theta_k^{\mu_3} + B_1)^{-1}\left[I + \sum_{i<k} a^1_{i,k+1}E_{i,k}+ \sum_{k+1<j} a_{k,j}^1 E_{k+1,j} \right](\Theta_k^{\mu_3} + B_1)\bigg] \\
& = A_0C_1+C_1A_0 +(I+B_0) \\ & \quad -(\Theta_k^{\mu_1} + B_1)^{-1}D^{-1}\left[I + \sum_{i<k} a^1_{i,k+1}E_{i,k}+ \sum_{k+1<j} \overline{c} a_{k,j}^1 E_{k+1,j} \right]D(\Theta_k^{\mu_1} + B_1) \\
& = A_0C_1+C_1A_0 +(I+B_0) \\ & \quad -(\Theta_k^{\mu_1} + B_1)^{-1}\left[I + \sum_{i<k} a^1_{i,k+1} \overline{c} E_{i,k}+ \sum_{k+1<j} \overline{c} a_{k,j}^1 E_{k+1,j} \right](\Theta_k^{\mu_1} + B_1) \\
& = J_{C_1} \cdot \partial_1C_1 = J_{C_1} \cdot (\partial_1 \circ \Phi)(C_1).
\end{align*}
At the 2nd equality we used that $\Phi(a_{k,j}^1) = \overline{c} a^1_{k,j}$ for $k+1<j$ and  $\Phi(a^1_{i,k+1})= a^1_{i,j}$ for $i<k$; at the 3rd equality we used that for $i<k$, $D^{-1}a^1_{i,k+1} E_{i,k}D = a^1_{i,k+1}\overline{c}E_{i,k}$, and for $k+1<j$, $D^{-1}a^1_{k,j}  E_{k+1,j}D = a^1_{k,j}E_{k+1,j}$.

\medskip

\noindent {\bf  Proof of (\ref{eq:58iii}):}

\medskip

We begin by applying Proposition \ref{prop:cancel} to remove all generators of the form $a^1_{i,k}, a^1_{i, k+1}, a^1_{k,j}$ and $a^1_{k+1,j}$ with $i<k$ and $k+1 <j$ from $\alg(\Lambda_2)|_{a=\overline{a}}$; notice that as $a^1_{k,k+1} = a$ is also removed from the generating set of $\alg(\Lambda_2)$ due to the specialization, the remaining generators are in bijection with those of $\alg(\Lambda_4)|_{\lambda=1}$.  This application of Proposition \ref{prop:cancel} is done in two parts: First, take the $y_1, \ldots, y_r$ and $z_1, \ldots, z_r$ to respectively be the $a^1_{i,k+1}$ and $a^1_{i,k}$ ordered with $i$ decreasing from $k-1$ to $1$.  Here, equation (\ref{eq:2JA1Lambda1}) provides for $i<k$
\[
\partial_2 \framebox{$a^1_{i,k+1}$} = \pm\left( \framebox{$a^1_{i,k}$}\, \overline{a} + \sum_{i<l<k} a^1_{i,l}a^1_{l,k+1} \right)
\]
which indeed demonstrates the required form from equation (\ref{eq:yzvw}) where $\alpha = \pm \overline{a} \in \mathbb{F}^*$ and $w=0$.  Second, take the $y$ and $z$ generators to be $a^1_{k,j}$ and $a^1_{k+1,j}$ in the order of increasing $j$; we have
\[
\partial_2 \framebox{$a^1_{k,j}$} = \pm\left( \overline{a}\,\framebox{$a^1_{k+1,j}$} + \sum_{k+1<l<j} a^1_{k,l}a^1_{l,j} \right)
\]
again verifying (\ref{eq:yzvw}).  Denote by $\mathcal{B}(\Lambda_2, \overline{a})$ the resulting stable tame isomorphic quotient of $\alg(\Lambda_2)|_{a = \overline{a}}$.  Notice that in this quotient (since in the applications of Proposition \ref{prop:cancel} the $w$ term was always $0$), the cancelled generators all become $0$.  Moreover, the matrix $A_1$ from $\alg(\Lambda_2)$ becomes
\begin{equation}  \label{eq:A1doteq}
A_1 \doteq \widetilde{A}_1 + \overline{a}E_{k,k+1}
\end{equation}
where $\widetilde{A}_1$ denotes $A_1$ with all entries in rows and columns $k$ and $k+1$ replaced with $0$.  
With this notation the matrix $\widehat{A}_1$ from Proposition \ref{prop:ALambda4} is $\widehat{A}_1= \widetilde{A}_1+ E_{k,k+1}$.  In particular, the differential on $\mathcal{B}(\Lambda_2, \overline{a})$ has 
\begin{equation}  \label{eq:1102}
J_{A_1} \cdot \partial_2 \widetilde{A}_1 = J_{A_1} \cdot \partial_2 (\widetilde{A}_1 + \overline{a}E_{k,k+1}) = J_{A_1} \cdot \partial_2 (A_1) = A_1^2 = (\widetilde{A}_1 + \overline{a}E_{k,k+1})^2 = (\widetilde{A}_1)^2.
\end{equation}

A dg-algebra isomorphism,  $\Phi: \mathcal{B}(\Lambda_2, \overline{a}) \rightarrow \alg(\Lambda_4)|_{\lambda=1}$ now arises from requiring $\Phi|_{R} = \phi_{\overline{a}}$ (defined as in Proof of (5.15)); 
\[
\Phi(\Theta_k^{\mu_2}+B_1) = D(\Theta_k^{\mu_4}+B_1)
\]
where $D$ is now a diagonal matrix with $\overline{a}$ in the $k$-th position and all other diagonal entries $1$; and $\Phi$ acts as the identity on all generators not belonging to $B_1$.  To verify that $\Phi$ is a chain map, note that using (\ref{eq:1102}) and (\ref{eq:4JA1Lambda1}) we have 
\[
J_{A_1} \cdot (\Phi \circ \partial_2)(\widetilde{A}_1) = \Phi[(\widetilde{A}_1)^2] = (\widetilde{A}_1)^2 = (\widehat{A}_1)^2 = J_{A_1} \cdot \partial_4 \widehat{A}_1 = J_{A_1} \cdot \partial_4 \widetilde{A}_1 = J_{A_1} \cdot (\partial_4 \circ \Phi)( \widetilde{A}_1).
\]
In addition, observing that
\[
\widetilde{A}_1D = \widetilde{A}_1 = D \widetilde{A}_1 \quad \mbox{and} \quad \overline{a}E_{k,k+1} D = \overline{a}E_{k,k+1} = D E_{k,k+1}
\]
and using equations (\ref{eq:2JB1Lambda1}), (\ref{eq:4JB1Lambda1}), and (\ref{eq:A1doteq}), we can compute
\begin{align*}
J_{B_1} \cdot (\Phi \circ \partial_2)(\Theta_k^{\mu_2}+ B_1) & = \Phi\left[ (\widetilde{A}_1 + \overline{a}E_{k,k+1})(\Theta_k^{\mu_2}+B_1) - (\Theta_k^{\mu_2}+B_1)A_0 \right] \\
& = (\widetilde{A}_1 + \overline{a}E_{k,k+1}) D (\Theta_k^{\mu_4}+B_1) - D(\Theta_k^{\mu_4}+B_1)A_0 \\
& = D\left[(\widetilde{A}_1 + E_{k,k+1})(\Theta_k^{\mu_4}+B_1) - (\Theta_k^{\mu_4}+B_1)A_0 \right] \\
& = D\left[ J_{B_1} \cdot \partial_4(\Theta_k^{\mu_4}+B_1) \right]  \\ & = J_{B_1} \cdot \left[ \partial_4( D(\Theta_k^{\mu_4}+B_1))\right]  = J_{B_1} \cdot (\partial_4 \circ \Phi)(\Theta_k^{\mu_2}+B_1).
\end{align*}
For the rest of the generators (including entries of $C_1$), the differentials agree in $\mathcal{B}(\Lambda_2, \overline{a})$ and $\alg(\Lambda_4)|_{\lambda=1}$ and do not involve $B_1$, $\mu_2$ or $\mu_4$.  [In (\ref{eq:4JC1Lambda1}), use that $\Theta^\lambda_{k,k+1}=I$ due to the specialization of $\lambda=1$.]

\medskip

\noindent {\bf  Proof of (\ref{eq:58iv}):}  Since the $k,k+1$ entry of the matrix $A^\lambda_1$ from Proposition \ref{prop:ALambda3} becomes the non-zero element $\overline{\lambda}^{-1} -1 \in \mathbb{F}^*$, we can begin by cancelling all generators of the form $a_{i,k}, a_{i,k+1}, a_{k,j}, a_{k+1,j}$ precisely as was done in Proof of (\ref{eq:58iii}) above.  Each of these generators becomes equal to $0$ in the resulting quotient dg-algebra that we denote as $\mathcal{B}(\Lambda_3, \overline{\lambda})$.  The proof is then completed via an isomorphism $\Phi: \mathcal{B}(\Lambda_3, \overline{\lambda}) \rightarrow \alg(\Lambda_4)|_{\lambda = \overline{\lambda}}$ defined to have $\Phi|_{R} = \phi_{\overline{\lambda}^{-1}-1}$; to have 
\[
\Phi(\Theta_{k}^{\mu_3}+B_1) = D(\Theta_{k}^{\mu_4}+B_1)
\]
where $D$ is now taken to be diagonal with $\overline{\lambda}^{-1}-1$ in the $k$-th position and $1$'s elsewhere;  and to act as the identity on all other generators.  That $\Phi \circ \partial_3 = \partial_4 \circ \Phi$ is verified on generators from $A_1$ and $B_1$ in the same manner as in the Proof of (\ref{eq:58iii}).  For consideration of $C_1$, note that in $\mathcal{B}(\Lambda_3, \overline{\lambda})$  the matrices $U_\lambda$ and $\Theta_{k,k+1}^\lambda$ from Proposition \ref{prop:ALambda3} becomes equal.  Then, we see that
\begin{align*}
J_{C_1} \cdot (\Phi \circ \partial_3)(C_1)  & = \Phi\left[ A_0 C_1 + C_1A_0 +(\Theta_{k,k+1}^\lambda+B_0) -(\Theta_k^{\mu_3}+B_1)^{-1}\Theta^\lambda_{k,k+1} (\Theta_k^{\mu_3}+B_1) \right] \\
 & = A_0 C_1 + C_1A_0 +(\Theta_{k,k+1}^\lambda+B_0) - (\Theta_k^{\mu_4}+B_1)^{-1}D^{-1}\Theta^\lambda_{k,k+1} D(\Theta_k^{\mu_4}+B_1) \\
 & = A_0 C_1 + C_1A_0 +(\Theta_{k,k+1}^\lambda+B_0) - (\Theta_k^{\mu_4}+B_1)^{-1}\Theta^\lambda_{k,k+1}(\Theta_k^{\mu_4}+B_1) \\
 & = J_{C_1} \cdot \partial_4(C_1) = J_{C_1} \cdot (\partial_4 \circ \Phi)(C_1).
 \end{align*}

\end{proof}

We are now prepared to prove the index $0$/$2$ skein relation (SR1).

\begin{proof}[Proof of (SR1)]  {\bf Case 1.} {\it $|a|$ and $|c|$ are non-zero.}  Then, any augmentation of $\alg(\Lambda_1)$ (resp. $\alg(\Lambda_2)$) must map $c \mapsto 0$ (resp. $a\mapsto 0$), so that via Proposition \ref{prop:specbiject} we have bijections  
\[
V(\alg(\Lambda_1), \mathbb{F}_q) \cong V(\alg(\Lambda_1)|_{c=0}, \mathbb{F}_q) \quad \mbox{and} \quad V(\alg(\Lambda_2), \mathbb{F}_q) \cong V(\alg(\Lambda_2)|_{a=0}, \mathbb{F}_q).
\]
As the shifted Euler characteristics (from Definition \ref{def:DGalgAugNum}) are related by
\[
\chi^*(\alg(\Lambda_1)|_{c=0}) = \chi^*(\alg(\Lambda_1)) - \eta(|c|)  \quad \mbox{and} \quad \chi^*(\alg(\Lambda_2)|_{a=0}) = \chi^*(\alg(\Lambda_2)) - \eta(|a|)
\]
where $\eta$ is the shifted parity function from (\ref{eq:parity}), the augmentation numbers of the {\it dg-algebras} (as in Definition \ref{def:DGalgAugNum}) satisfy
\[
q^{\eta(|c|)/2}\aug(\alg(\Lambda_1), q) = \aug(\alg(\Lambda_1)|_{c=0}, q) \quad \mbox{and} \quad q^{\eta(|a|)/2}\aug(\alg(\Lambda_2), q) = \aug(\alg(\Lambda_2)|_{a=0}, q).
\]
Now, since $H_0(\Lambda_1) \cong H_0(\Lambda_2)$ and $|\Gamma_1| = |\Gamma_2|$, working from Definition \ref{def:augLambda} and using (\ref{eq:acMaslov}) we verify
\begin{align*}
& \quad \quad q^{\eta(s-r+1)/2}\aug_{\Lambda_1}(q) - q^{\eta(r-s-1)/2}\aug_{\Lambda_2}(q) \\ & = q^{\eta(|c|)/2} (q-1)^{\dim H_0(\Lambda_1)- |\Gamma_1|}\aug(\alg(\Lambda_1),q)  - q^{\eta(|a|)/2} (q-1)^{\dim H_0(\Lambda_2)- |\Gamma_2|}\aug(\alg(\Lambda_2),q) \\
& = (q-1)^{\dim H_0(\Lambda_1)- |\Gamma_1|}\big(\aug(\alg(\Lambda_1)|_{c=0},q) - \aug(\alg(\Lambda_2)|_{a=0},q) \big) = 0
\end{align*}
where the last equality used (\ref{eq:58i}).  Proposition \ref{prop:stableinv} and Theorem \ref{thm:CWmulti}  have been applied as well to allow the use of the dg-algebra $\alg(\Lambda_1)$ in computing $\aug_{\Lambda_1}(q)$, and similarly for $\Lambda_2$.

\medskip

\noindent {\bf Case 2.} {\it $|a|=|c|=0$ in Connected Case.}  We need to establish the identity
\[
q^{1/2} \aug_{\Lambda_1}(q) - q^{1/2} \aug_{\Lambda_2}(q) = (q-1)^2 \big[ \aug_{\Lambda_3}(q) - \aug_{\Lambda_4}(q) \big].
\]
Begin by using that $H_0(\Lambda_1) \cong H_0(\Lambda_2)$ and $|\Gamma_1|= |\Gamma_2|$ to compute from Definitions \ref{def:augLambda} and \ref{def:DGalgAugNum} that
\begin{align}
\notag & q^{1/2} \aug_{\Lambda_1}(q) - q^{1/2} \aug_{\Lambda_2}(q)  \\
\notag = \quad & q^{1/2}(q-1)^{\dim H_0(\Lambda_1)-|\Gamma_1|}\big[ \aug(\alg(\Lambda_1)) - \aug(\alg(\Lambda_2)) \big] \\
= \quad & (q-1)^{\dim H_0(\Lambda_1)-|\Gamma_1|}\left[ q^{-[\chi^*(\alg(\Lambda_1)) -1]/2} | V(\alg(\Lambda_1), \mathbb{F}_q)| -q^{-[\chi^*(\alg(\Lambda_2)) -1]/2} | V(\alg(\Lambda_2), \mathbb{F}_q)|. \right. \label{eq:SR1proof1}
\end{align}
[We omit the dependence on $q$ from the notation for augmentation numbers of dg-algebras from Definition \ref{def:DGalgAugNum}.]
Next, 
Proposition \ref{prop:specbiject} provides a bijection
\[
V(\alg(\Lambda_1), \mathbb{F}_q)  \stackrel{\cong}{\longleftrightarrow} \bigsqcup_{ \overline{c} \in \mathbb{F}_q} V(\alg(\Lambda_1)|_{c=\overline{c}}, \mathbb{F}_q) 
\]
It is also the case that $\chi^*(\alg(\Lambda_1)|_{c= \overline{c}}) = \chi^*(\alg(\Lambda_1))-1$.  Using this, and a similar observations for $\Lambda_2$, we can continue our computation with (\ref{eq:SR1proof1}) becoming
\begin{align}
& (q-1)^{\dim H_0(\Lambda_1)-|\Gamma_1|}\left[ \sum_{\overline{c} \in \mathbb{F}_q} \aug(\alg(\Lambda_1)|_{c=\overline{c}}) - \sum_{\overline{a} \in \mathbb{F}_q} \aug(\alg(\Lambda_2)|_{a=\overline{a}}) \right]. 
\end{align}
Now, using equations (\ref{eq:58i}) through (\ref{eq:58iii}) from Proposition \ref{prop:specialiso} with Proposition \ref{prop:stableinv} the $\overline{c}=0$ and $\overline{a}=0$ terms cancel, and the $q-1$ terms with $\overline{c} \neq 0$ (resp. with $\overline{a} \neq 0$) can be replaced with $\aug(\alg(\Lambda_3)|_{\lambda =1})$ (resp. with $\aug(\alg(\Lambda_4)|_{\lambda =1})$) to get
\begin{align}
 & q^{1/2} \aug_{\Lambda_1}(q) - q^{1/2} \aug_{\Lambda_2}(q) \notag \\ \label{eq:sr1proof35} =  \quad & (q-1)^{\dim H_0(\Lambda_1)-|\Gamma_1|} \bigg[ (q-1) \aug(\alg(\Lambda_3)|_{\lambda =1}) - (q-1) \aug(\alg(\Lambda_4)|_{\lambda =1}) \bigg] \\
 \label{eq:sr1proof3} = \quad & (q-1)^{\dim H_0(\Lambda_1)-|\Gamma_1|+1} \bigg[ \aug(\alg(\Lambda_3)) - \aug(\alg(\Lambda_4))  \bigg]  \\
 \label{eq:sr1proof4} = \quad & (q-1)^2(q-1)^{\dim H_0(\Lambda_3)-|\Gamma_3|}  \bigg[ \aug(\alg(\Lambda_3)) - \aug(\alg(\Lambda_4))  \bigg]  \\ 
  \label{eq:sr1proof5} = \quad & (q-1)^2\big[ \aug_{\Lambda_3}(q) - \aug_{\Lambda_4}(q) \big]
\end{align}
as required.  At the equality (\ref{eq:sr1proof3}) we used (\ref{eq:58iv}) and Proposition \ref{prop:specbiject} to see that 
\begin{align*}
\aug(\alg(\Lambda_3)|_{\lambda =1}) - \aug(\alg(\Lambda_4)|_{\lambda =1}) & =  \sum_{\overline{\lambda} \in \mathbb{F}^*_q} \aug(\alg(\Lambda_3)|_{\lambda =\overline{\lambda}}) - \sum_{\overline{\lambda} \in \mathbb{F}^*_q}\aug(\alg(\Lambda_4)|_{\lambda =\overline{\lambda}})  \\
 & =  \aug(\alg(\Lambda_3)) - \aug(\alg(\Lambda_4)).
\end{align*}
  At (\ref{eq:sr1proof4}) and (\ref{eq:sr1proof5}) we used that, due to the Connected Case, $\dim H_0(\Lambda_1) = \dim H_0(\Lambda_3) = \dim H_0(\Lambda_4)$ and $|\Gamma_1 | = |\Gamma_3| -1= |\Gamma_4|-1$.
  
  \medskip
  
  \noindent {\bf Case 3.} {\it $|a|=|c|=0$ in Disconnected Case.}  The computation is the same through (\ref{eq:sr1proof35}).  From there we use that (in the Disconnected Case) $\alg(\Lambda_3)|_{\lambda =1} = \alg(\Lambda_3)$, $\alg(\Lambda_4)|_{\lambda =1} = \alg(\Lambda_4)$,  $\dim H_0(\Lambda_1) = \dim H_0(\Lambda_3) +1 = \dim H_0(\Lambda_4) +1$, and $|\Gamma_1| = |\Gamma_3|=|\Gamma_4|$ to justify the same equalities (\ref{eq:sr1proof3})-(\ref{eq:sr1proof5}) that complete the argument.
  
\end{proof}

\section{The index $1$ skein relation via Reidemeister moves}  \label{sec:SR2}
The next order of business is to deduce the index $1$ skein relation, (SR2),  from the index $0/2$ skein relation, (SR1).  (Again the ``index'' referred to is the Morse index of the Reeb chords as critical points of local difference functions.)  This is based on the following:

\begin{proposition}  \label{prop:index1main}
	Let $\Lambda_1, \Lambda_2, \Lambda_3, \Lambda_4$ be Legendrian surfaces related as in the four terms of (SR2) numbered from left to right.  Then, there exists Legendrian isotopic surfaces
	\[
	\Lambda_1 \cong \Lambda_1', \quad \Lambda_2 \cong \Lambda_2', \quad \Lambda_3 \cong \Lambda_3', \quad \Lambda_4 \cong \Lambda_4'
	\]
	such that $\Lambda_1', \Lambda_2', \Lambda_3', \Lambda_4'$ are related as in the four terms of (SR1).  Moreover, the values of the Maslov potential\footnote{Consistent with the $\delta_{r,s}$ term on the RHS of (SR2), we only consider $\Lambda_3$, $\Lambda_4$, $\Lambda'_3$ and $\Lambda'_4$ in the case that $r=s$.}  $r$ and $s$ for the $\Lambda_i$ that appear in (SR2) are related to the Maslov potential values $r'$ and $s'$ for the $\Lambda_i'$ appearing in (SR1) by $r'=s+1$ and $s' = r$.
\end{proposition}
The deduction of (SR2) from (SR1) then follows from Proposition \ref{prop:index1main} via
\begin{align*}
q^{\eta(r-s)/2} \aug_{\Lambda_1} - q^{\eta(s-r)/2} \aug_{\Lambda_2} & = q^{\eta(s'-r'+1)/2} \aug_{\Lambda'_1} - q^{\eta(r'-s'-1)/2} \aug_{\Lambda'_2} \\
 & = \delta_{r',s'+1}(q-1)^2 \left[ \aug_{\Lambda'_3} - \aug_{\Lambda_4'} \right] \\
 & = \delta_{s,r}(q-1)^2 \left[ \aug_{\Lambda_3} - \aug_{\Lambda_4} \right].
\end{align*} 
The remainder of the section establishes Proposition \ref{prop:index1main} in the following manner.  The surfaces $\Lambda_i$ and $\Lambda_i'$, $1 \leq i \leq 4$ are presented in Section \ref{sec:4surfaces} from which it is clear  that the $\Lambda_i'$ are related as in (SR1) with the claimed Maslov potential values.  After reviewing some elementary isotopies of Legendrian surfaces arising from local moves of Arnold \cite{Arnold1, Arnold} and Goryunov-Alsaeed \cite{Goryunov} in Section \ref{sec:AGmoves} and desingularization of the cone point in Section \ref{sec:desingcone}, we finish the proof of Proposition \ref{prop:index1main} in Section \ref{sec:ProofProp} by obtaining the required isotopies in Lemmas \ref{lem:L1iso}, \ref{lem:L2iso}, \ref{lem:lambda4}, and \ref{lem:L3iso}.

\subsection{The four surfaces $\Lambda_i'$}  \label{sec:4surfaces}
Recall from the statement of (SR2) from Theorem \ref{thm:main} that, within the ball in the front projection $B \subset S\times \R$ where they differ, the Legendrians $\Lambda_1, \Lambda_2, \Lambda_3, \Lambda_4$ have two sheets with crossing arcs that project to $\pi_x(B) \subset S$ as follows:
\[
\begin{array}{cccc} 
\quad \includegraphics[scale=.5]{images/SR2b} \quad &  \quad \includegraphics[scale=.5]{images/SR1b} \quad & \quad \includegraphics[scale=.5]{images/SR4b} \quad & \quad \includegraphics[scale=.5]{images/SR3b} \quad \\
\Lambda_1 & \Lambda_2 & \Lambda_3 & \Lambda_4
\end{array}
\]
At the points where tri-valent vertices appear in the crossing arcs of $\Lambda_3$ and $\Lambda_4$ the front projection has a $D_4^-$ singularity.  
These singularities are non-generic for Legendrian surfaces.  However, they occur generically in $1$-parameter families of Legendrian surfaces at the singular time of the $D_4^-$ Move as pictured in Section \ref{sec:AGmoves} below; see \cite{TZ, CZ} for more discussion.  

We present each of the $\Lambda_i'$ by picturing on the left the image of its crossings (in blue) and cusps (in red) in the base projection,   $\pi_x(B)$, and on the right the movie of slices of the front projection above horizontal lines in $\pi_x(B)$ with increasing $x_2$-coordinate.

\medskip

\centerline{
\begin{tabular}{|c|c|}
\hline
 & \\
$\Lambda_1'$   &  $\Lambda_2'$ \\
 & \\
\hline
 & \\
\quad \includegraphics[scale=.45]{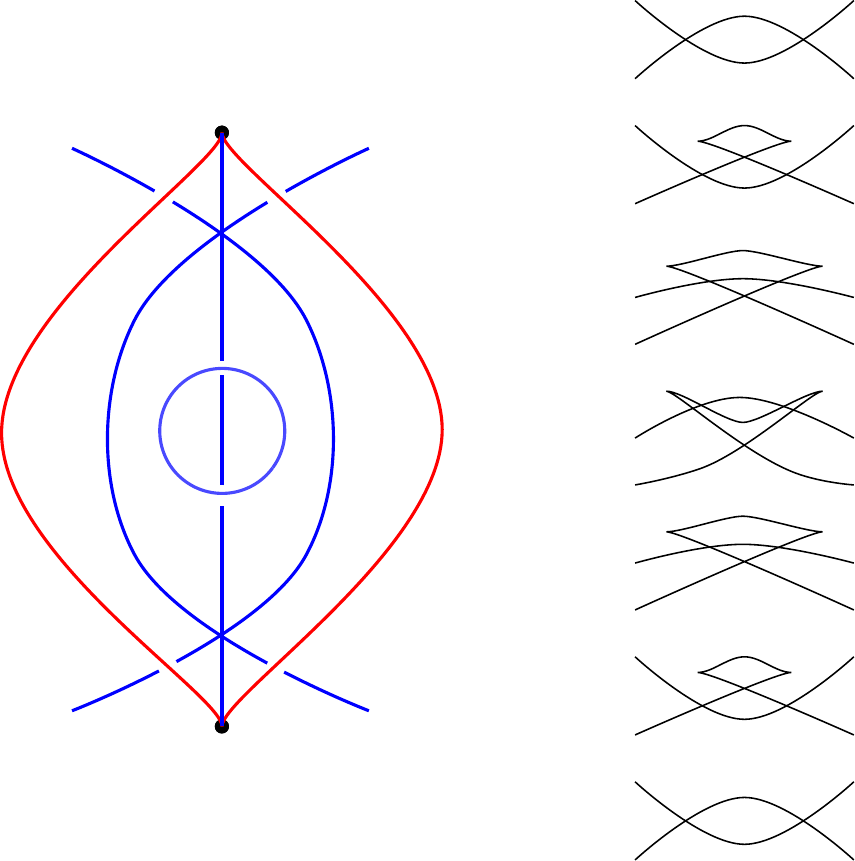}  \quad \quad \quad  &   \quad \quad  \raisebox{1cm}{\includegraphics[scale=.45]{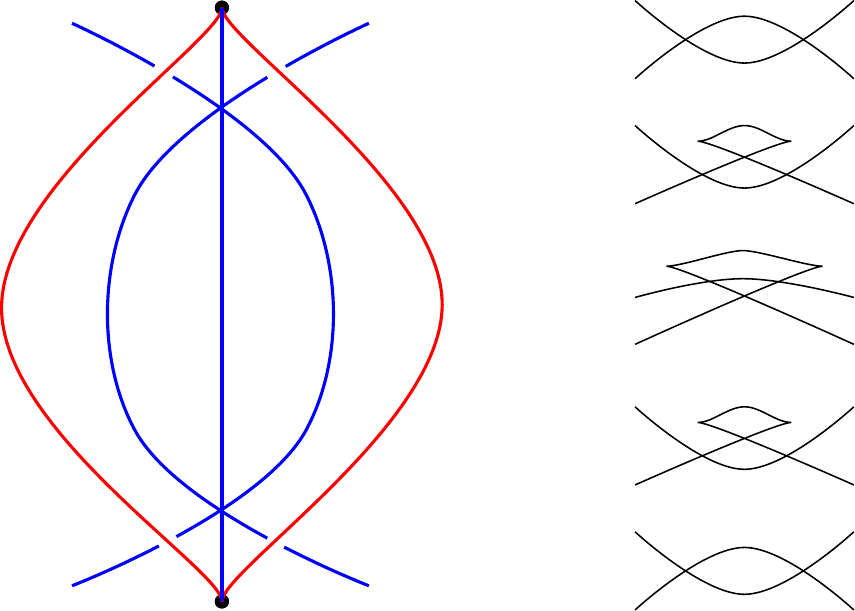}} \quad \\
\hline
\end{tabular}}

\centerline{
\begin{tabular}{|c|c|}
\hline
 & \\
$\Lambda_3'$   &  $\Lambda_4'$ \\
 (a cone point connects the top two sheets) & \\
\hline
 & \\
\quad \labellist
		\small
		\pinlabel $\begin{array}{c}\mbox{Cone}\\\mbox{point}\end{array}$ [b] at -2 188
		\endlabellist
		\raisebox{1cm}{\includegraphics[scale=.45]{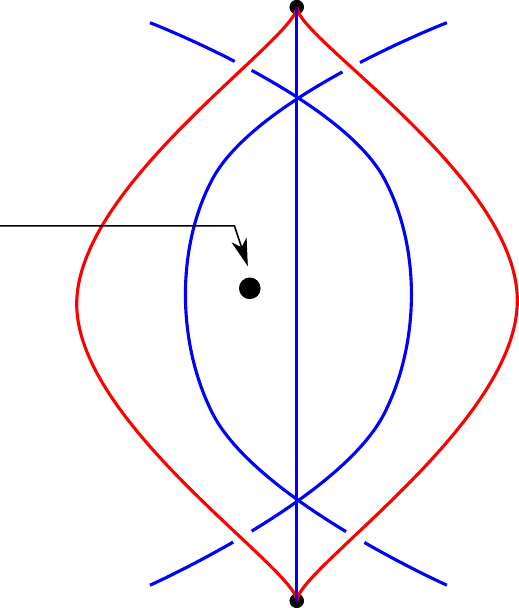}} \quad \quad 
		\labellist
		\small
				\pinlabel $\begin{array}{c}\mbox{Cone}\\\mbox{point}\end{array}$ [b] at 164 256
				\endlabellist
								\raisebox{0cm}{\includegraphics[scale=.45]{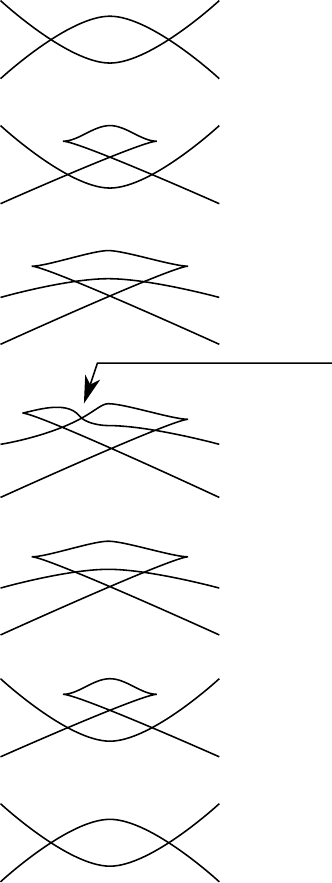}}
		  \quad \quad \quad  &    \quad \quad \raisebox{0cm}{\includegraphics[scale=.45]{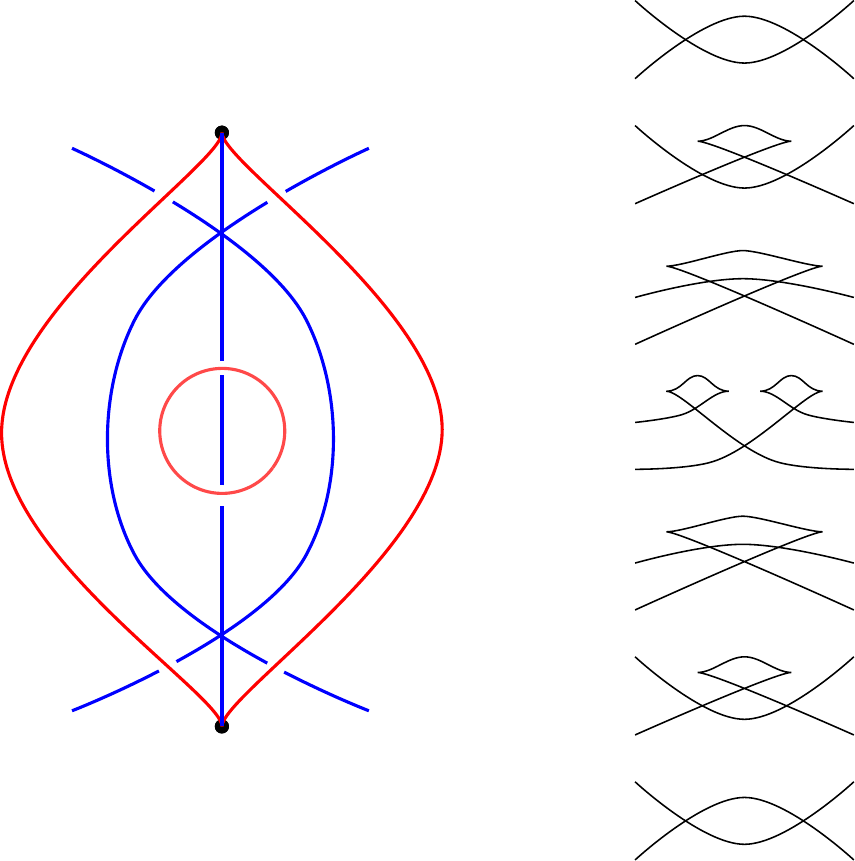}} \quad \\
\hline
\end{tabular}}

\subsection{Some Legendrian surface moves}  \label{sec:AGmoves}

We collect here a series of local modifications to the front projection of a Legendrian surface that can be be achieved via Legendrian isotopy.  We denote these local moves  as $A_1^4$, $A^2_2$, $A_3$-max I, $A_3$-max II, $A_2A_1$-max, $A_3A_1$, $A_1^3$-max, $A_4$, and $D_4^-$, based on Arnold's notation for (multi-germ) singularities.  The notation is explained below in Remark \ref{rem:notate}.  All the moves involve a modification within a 3-ball $B \subset S\times \R$ in the front projection.  In most cases, we have indicated projections of the singular set (crossings in blue and cusps in red) with select front slices pictured in black.  For some of the more complicated fronts, we have included the full movie of $1$-dimensional front slices.

\medskip

\begin{tabular}{|l|c|}

\hline

 & \\
\noindent{\bf $A_2^2$ Move:}   \quad \quad & \raisebox{-2cm}{\includegraphics[scale=.6]{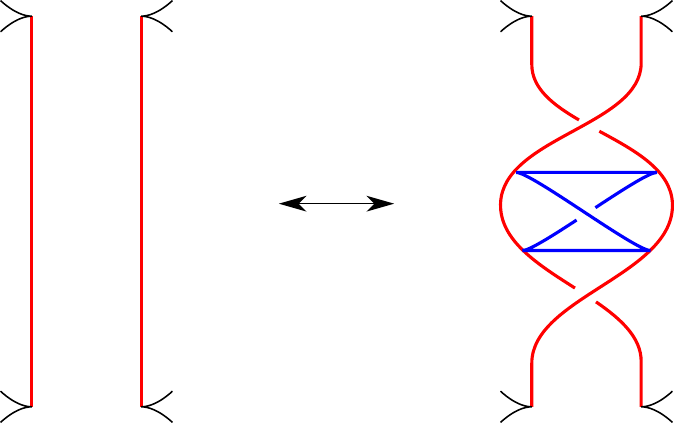}} \\ &

\\
\hline
& \\
\noindent{\bf $A_1^4$ Move:}   \quad \quad & \raisebox{-2cm}{\includegraphics[scale=.6]{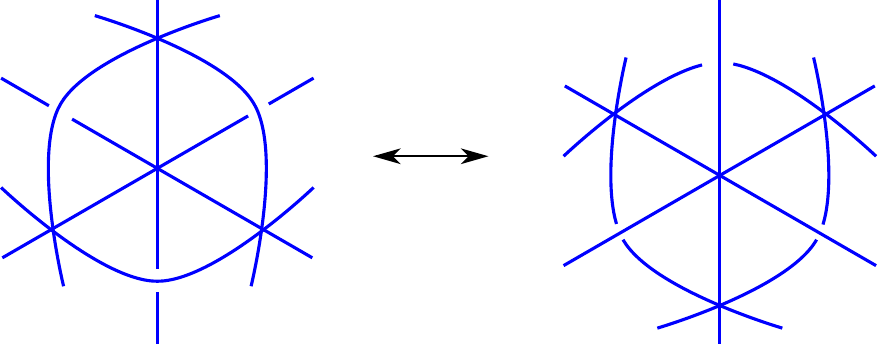}} \\ &
\\
\hline
\end{tabular}

\medskip

\noindent   
We provide a $3$-dimensional description of the $A_1^4$ move.  The front projection on the left side can be taken to be the union of $4$ planes, $P_i$, $1\leq i \leq 4$: Considering a $3$-space $\R^3$ with coordinates $(a_1, a_2,a_3)$, take $P_1, P_2, P_3$ to be the three coordinate planes, and take $P_4$ to be the plane passing through $(1,0,0)$, $(0,1,0)$, and $(0,0,1)$.  The oriented line $\ell: a_1=a_2=a_3$ corresponds to the $z$-direction in $S \times \R$.  The Legendrian isotopy consists of parallel translating $P_4$ downward along $\ell$ until it moves through the origin.   

\medskip

\begin{tabular}{|l|c|}

\hline
& \\
\noindent{\bf $A_3$-max I Move:}   \quad \quad & \raisebox{-2cm}{\includegraphics[scale=.6]{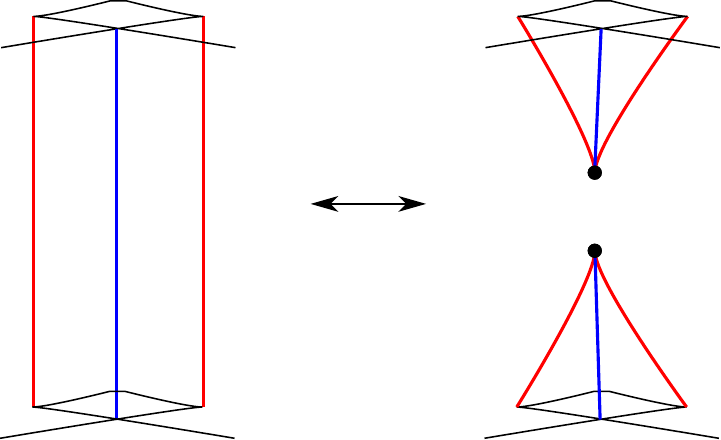}} \\ &
\\
\hline
& \\
\noindent{\bf $A_3$-max II Move:}   \quad \quad & \raisebox{-2cm}{\includegraphics[scale=.6]{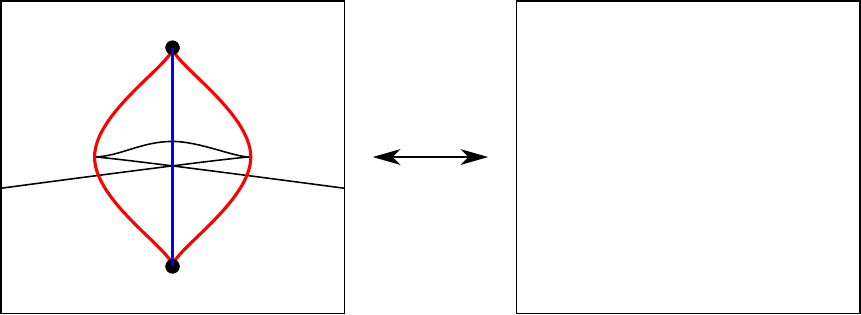}} \\ &
\\
\hline
& \\
\noindent{\bf $A_2A_1$-max Move:}   \quad \quad & \raisebox{-2cm}{\includegraphics[scale=.6]{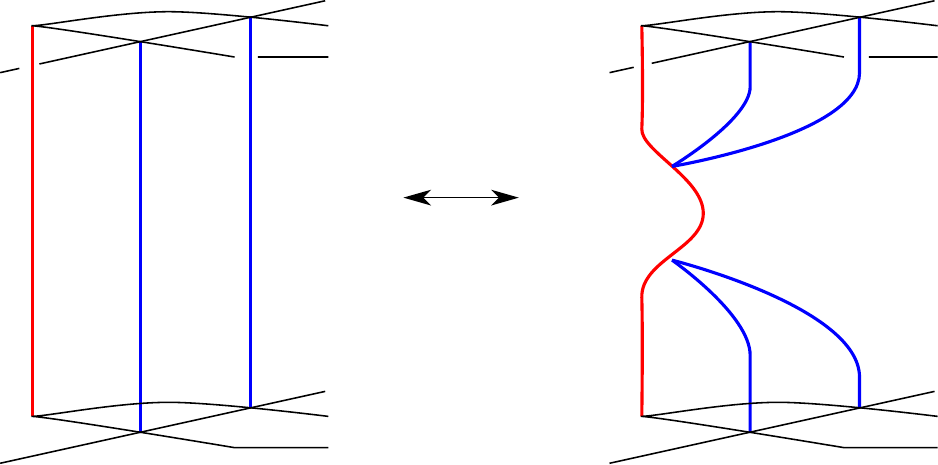}} \\ &
\\
\hline
 & \\
\noindent{\bf $A_3A_1$ Move:}   \quad \quad & \raisebox{-2cm}{\includegraphics[scale=.6]{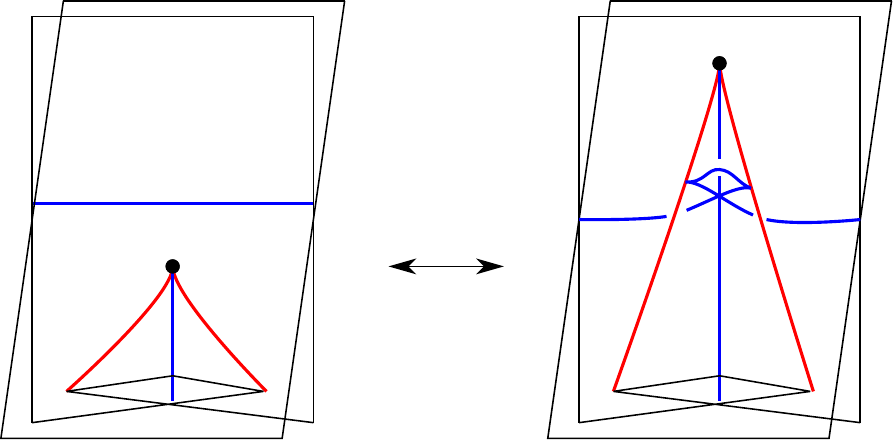}} \\ &
\\
\hline

\end{tabular}

\begin{tabular}{|l|c|}

\hline

 & \\
\noindent{\bf $A_1^3$-max Move:}   \quad \quad & \raisebox{-2cm}{\includegraphics[scale=.6]{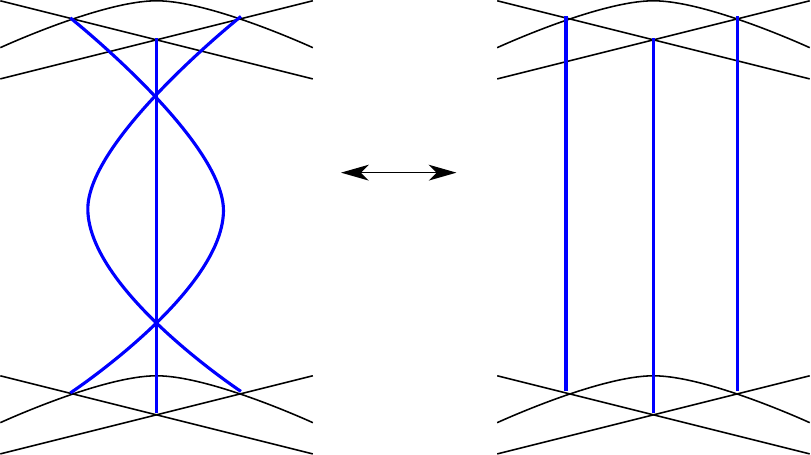}} \\ & 
\\
\hline

& \\
\noindent{\bf $A_4$ Move:}   \quad \quad & \raisebox{-2cm}{\includegraphics[scale=.6]{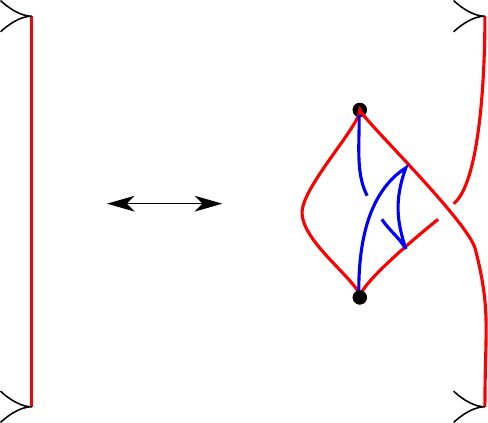}}  \\ &
\\
\hline
\end{tabular}

\medskip

\noindent For the right side of the $A_4$ Move, the movie of $1$-dimensional front slices above horizontal lines in the base projection with increasing $x_2$-coordinate appears as
\[
\includegraphics[scale=.6]{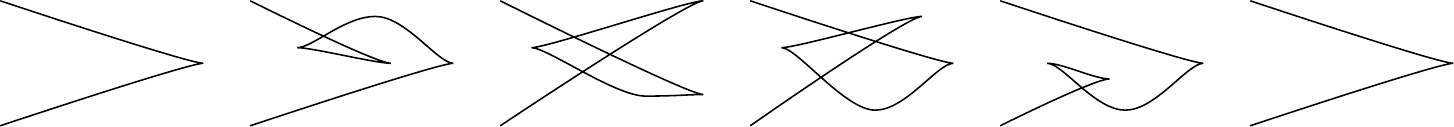}
\]

\medskip

\begin{tabular}{|l|c|}
	
	\hline
	
	& \\
	\noindent{\bf $D_4^-$ Move:}   \quad \quad & \raisebox{-2cm}{\includegraphics[scale=.6]{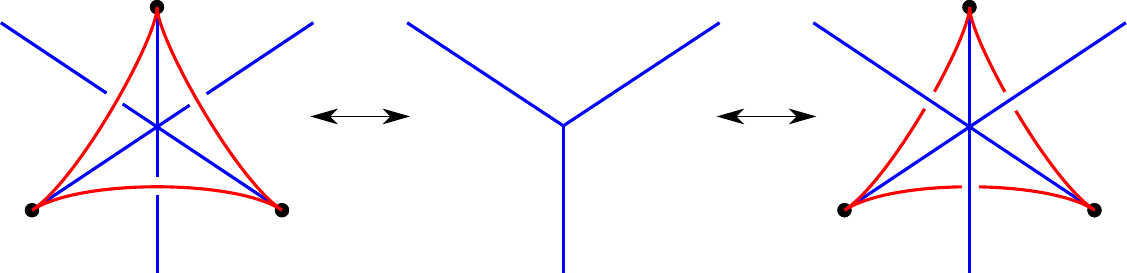}}
	\\ &
	\\
	\hline
\end{tabular}

\medskip

\noindent The left and right sides of the $D_4^-$ Move are related by reflection in the $z$-direction.  The movie of $1$-dimensional front slices for the left side of the move is
\[
\includegraphics[scale=.6]{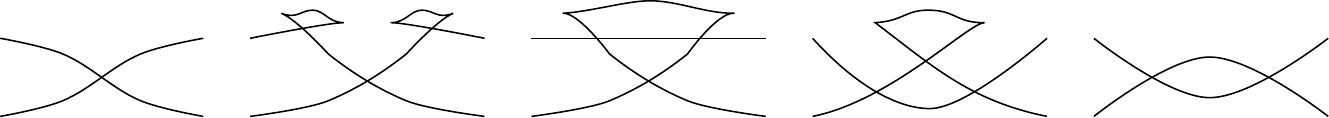}
\]

\begin{remark}[About the notations]  \label{rem:notate}
The notations $A_1$,  $A_2$, and $A_3$ refer points of $\Lambda =\sqcup_{i=0}^2\Lambda_i \subset J^1S$  (the notation is revived from Section \ref{sec:Zgraded}) belonging respectively to $\Lambda_0$,  $\Lambda_1$, and  $\Lambda_2$.  I.e., $A_1$ points are non-singular in the front projection,  $A_2$ points belong to cusp edges, and $A_3$ points are swallowtails.  Note that the subscript $i$ of $A_i$ (resp. $\Lambda_i$) coincides with codimension of the singularity set in the ambient front space, $S \times \R$, (resp. in $\Lambda$).  Products of $A_i$'s denote points in the front projection $\pi_{xz}(\Lambda)$ that are simultaneously the image of multiple such points.  The $A_4$ and $D_4^-$ singularities are two of $3$ generic codimension 3 (in $\Lambda$) singularities of $3$-dimensional front projections.  (The third,   $D_4^+$, is not needed for our proof of Proposition \ref{prop:index1main}.)  

Recall that from a Legendrian isotopy $\{\Lambda_t \}_{t \in [0,1]}$ in $J^1S$, one can construct a $3$-dimensional Legendrian $\Lambda_{\mathit{big}} \subset J^1(S \times [0,1])$ (called the {\it big front} in \cite[Ch. 3]{Arnold})  that is characterized by the property that the slice of the front projection of $\Lambda_{\mathit{big}}$ at $S \times\{t\}$ is the front projection of $\Lambda_t$ for each $t \in [0,1]$.  The notation for the local moves reflects a (multi)-singularity that appears in the front projection of  $\Lambda_{\mathit{big}}$ and corresponds to a unique time $t_0 \in (0,1)$ in the isotopy when the front projection of $\Lambda_{t_0}$ has a non-generic front singularity.  For instance, in the $A^2_2$ move, at the singular time $t_0$, there is a point in the front projection where two cusp edges intersect.  
The moves that include ``max'' in their titles correspond to critical points of the $t$-coordinate function restricted to particular singularity sets.   
\end{remark}

\begin{remark}  The classification of stable germs of Legendrian front singularities in dimensions up to $3$ goes back to Arnold in  \cite{Arnold1}.  See also the expositions in \cite{Arnold} or \cite{AGV}.  Generic 3-dimensional multi-germ singularities and bifurcations of $2$-dimensional Legendrian fronts in $1$-parameter families (i.e. Reidemeister moves for Legendrian surfaces) are exhaustively enumerated in the more recent work of Goryunov and Alsaeed \cite{Goryunov}.  We caution the reader that the collection of moves presented above is tailored towards our needs in proving Proposition \ref{prop:index1main} and is {\it not complete}, cf. \cite{Goryunov}. 
\end{remark}

\begin{proposition}
	The local modification to a front projection achieved by any of the above moves can be realized by a Legendrian isotopy supported in the pre-image in $J^1S$ of the pictured $3$-ball in the front projection. 
\end{proposition}
\begin{proof}  It should be noted that even at the singular time, the $2$-dimensional Legendrian surface is embedded.  For the $A_3$-max I and II, $A_4$, and $D_4^-$-moves, this may be checked using the generating families that produce the isotopies found in Section 3.2 of \cite{Goryunov}.  For the remaining moves, one should note that at the singular time, $t_0$, the tangent planes to distinct points of $\Lambda_{t_0}$ that intersect at the multi-singularity in the front projection are all pair-wise transverse.  (The $A_2$ and $A_3$ points have well-defined tangent planes determined (as a limit) by adjacent $A_1$ points.)  For instance with the $A^2_2$ move, at the singular time $t_0$
when the two cusp edges intersect each cusp edge is transverse to the tangent plane determined (as a limit) by the sheets adjacent to the other cusp edge.  In particular, due to this transversality the pre-image of the $A_2^2$ singular point consists of two distinct points in $\Lambda_{t_0} \subset J^1S$.  See also \cite[Section 4]{CZ} for further visualizations and discussion of the Legendrian isotopies corresponding to many of the above moves.
\end{proof}

\subsection{Desingularization of a cone point}  \label{sec:desingcone} The cone point singularity that appears in the 3rd term of (SR1) and also in $\Lambda_3'$ above is another non-generic singularity of Legendrian surfaces.  
The pre-image of a cone point $c$ in the front projection $\pi_{xz}(\Lambda)$ of a Legendrian $\Lambda \subset J^1S$ is an  embedded $S^1$ in $\Lambda$.  As detailed in \cite[Section 3]{Rizell} and \cite[Section 3]{EENS}, 
a small Legendrian isotopy perturbs a cone point
into a configuration of four swallowtail points, two downward and two upward, connected by crossing arcs and cusp edges as pictured via the base projection 
\[
\includegraphics[scale=.6]{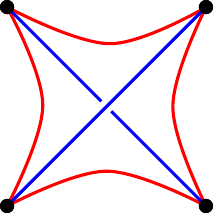}
\]
with the following movie of $1$-dimensional front slices 
\[
\includegraphics[scale=.5]{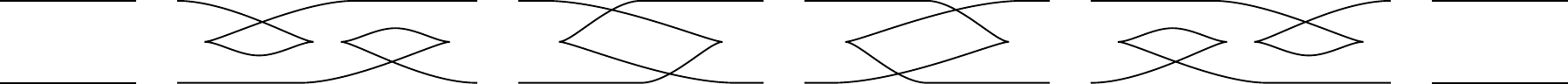}
\]
corresponding to horizontal lines sweeping upward through the pictured base projection.

\subsection{Proof of Proposition \ref{prop:index1main}} \label{sec:ProofProp}

The following Lemmas \ref{lem:L1iso}-\ref{lem:L3iso} establish the Legendrian isotopies required for Proposition \ref{prop:index1main}.

\begin{lemma} \label{lem:L1iso} The surfaces $\Lambda_1$ and $\Lambda_1'$ are Legendrian isotopic.
\end{lemma}

\begin{proof}  Applying successively the $A_2A_1$-max move (twice), the $A_3$-max I move, the $A_3A_1$ move (twice), and the $A_3$-max II move (twice) produces the required Legendrian isotopy:
	\[\raisebox{-2cm}{\includegraphics[scale=.4]{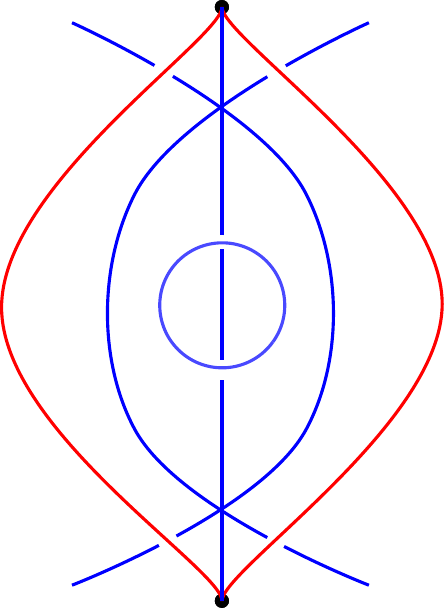}}\quad  \cong \quad \raisebox{-2cm}{\includegraphics[scale=.4]{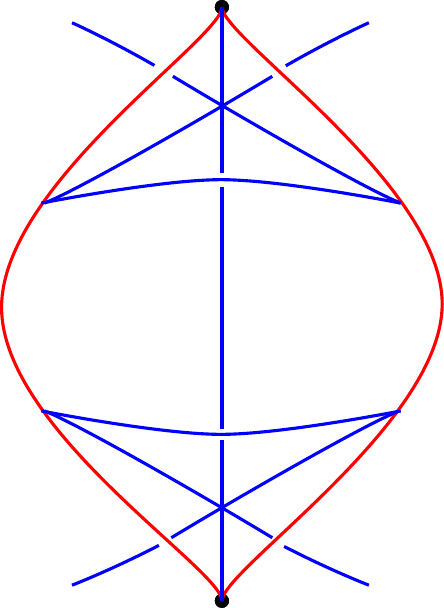}} \quad \cong \quad \raisebox{-2cm}{\includegraphics[scale=.4]{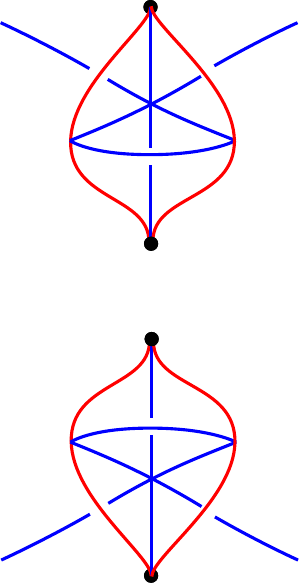}} \quad \cong \quad  \raisebox{-2cm}{\includegraphics[scale=.4]{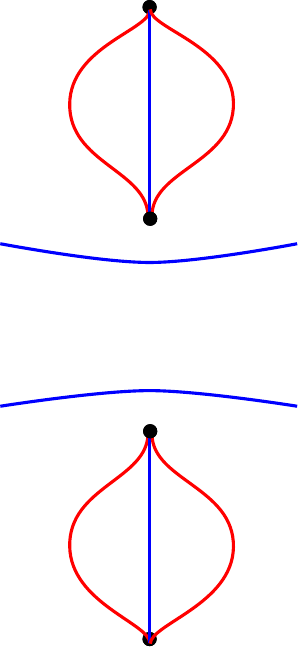}}  \quad \cong \quad \raisebox{-1cm}{\includegraphics[scale=.6]{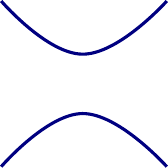}}
	\]	
\end{proof}

\begin{lemma} \label{lem:L2iso} The surfaces $\Lambda_2$ and $\Lambda_2'$ are Legendrian isotopic.
\end{lemma}

\begin{proof}  The proof is from \cite[Section 4.5, Figure 46 5-8]{CZ}.
	The isotopy is constructed making use of the $A_1^3$-max move and the $A_3$-max II move:
	\[\raisebox{-2cm}{\includegraphics[scale=.4]{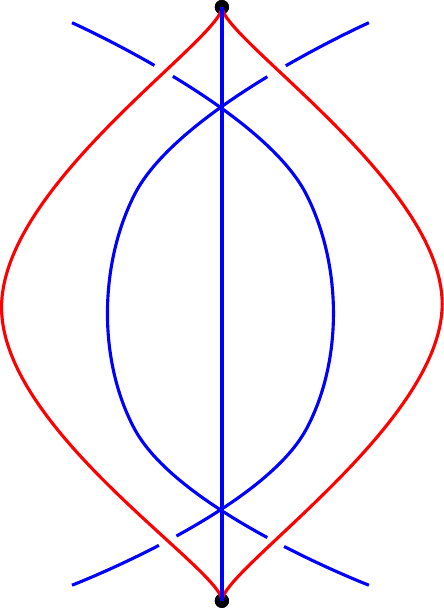}}\quad  \cong \quad \raisebox{-2cm}{\includegraphics[scale=.4]{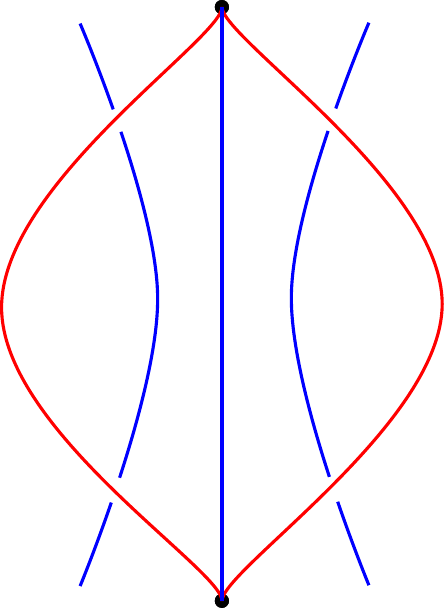}} \quad \cong \quad \raisebox{-1.5cm}{\includegraphics[scale=.4]{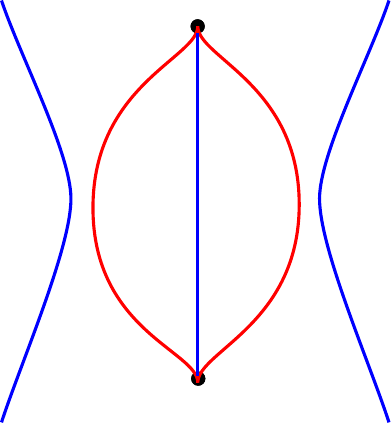}} \quad \cong \quad  \raisebox{-1cm}{\includegraphics[scale=.6]{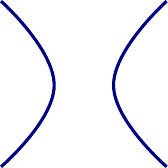}}
\]	
\end{proof}

\begin{lemma} \label{lem:lambda4} The surfaces $\Lambda_4$ and $\Lambda_4'$ are Legendrian isotopic.
\end{lemma}

\begin{proof}   The proof is from \cite[Section 4.5, Figure 46 1-3]{CZ}.  Using the $A_3$-max I move and the $D_4^-$-move we have
	\[\raisebox{-2cm}{\includegraphics[scale=.4]{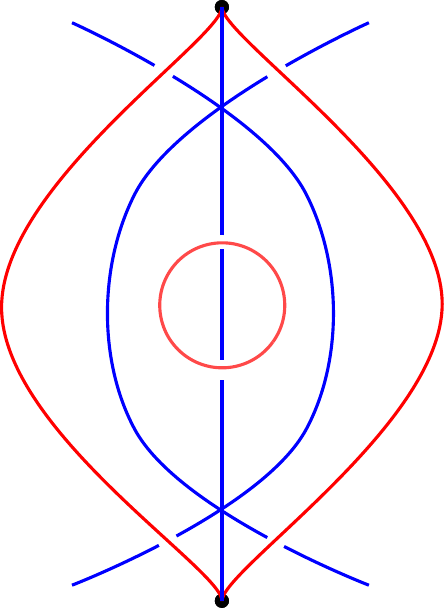}}\quad  \cong \quad \raisebox{-2cm}{\includegraphics[scale=.4]{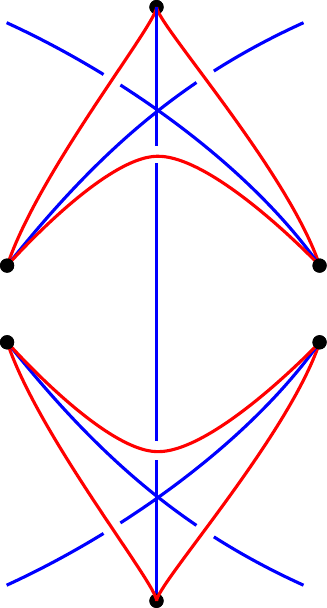}} \quad \cong \quad \raisebox{-1cm}{\includegraphics[scale=.6]{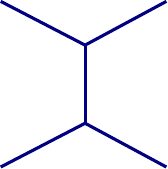}}.
	\]	
\end{proof}

\begin{lemma} \label{lem:L3iso} The surfaces $\Lambda_3$ and $\Lambda_3'$ are Legendrian isotopic.
\end{lemma}

\begin{proof}  We begin by resolving the cone point and then applying two $A_4$-moves:
		\[\Lambda_3' \quad  \cong \quad \raisebox{-2cm}{\includegraphics[scale=.4]{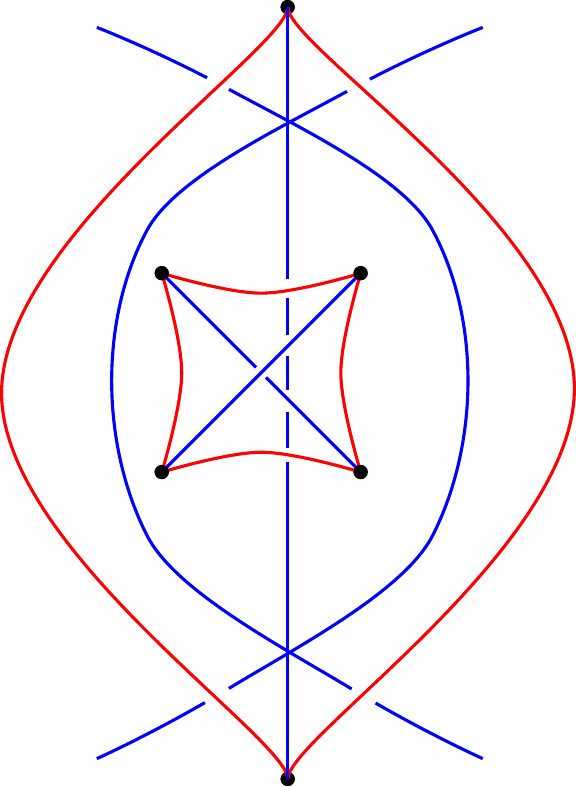}} \quad \cong \quad \raisebox{-2.5cm}{\includegraphics[scale=.4]{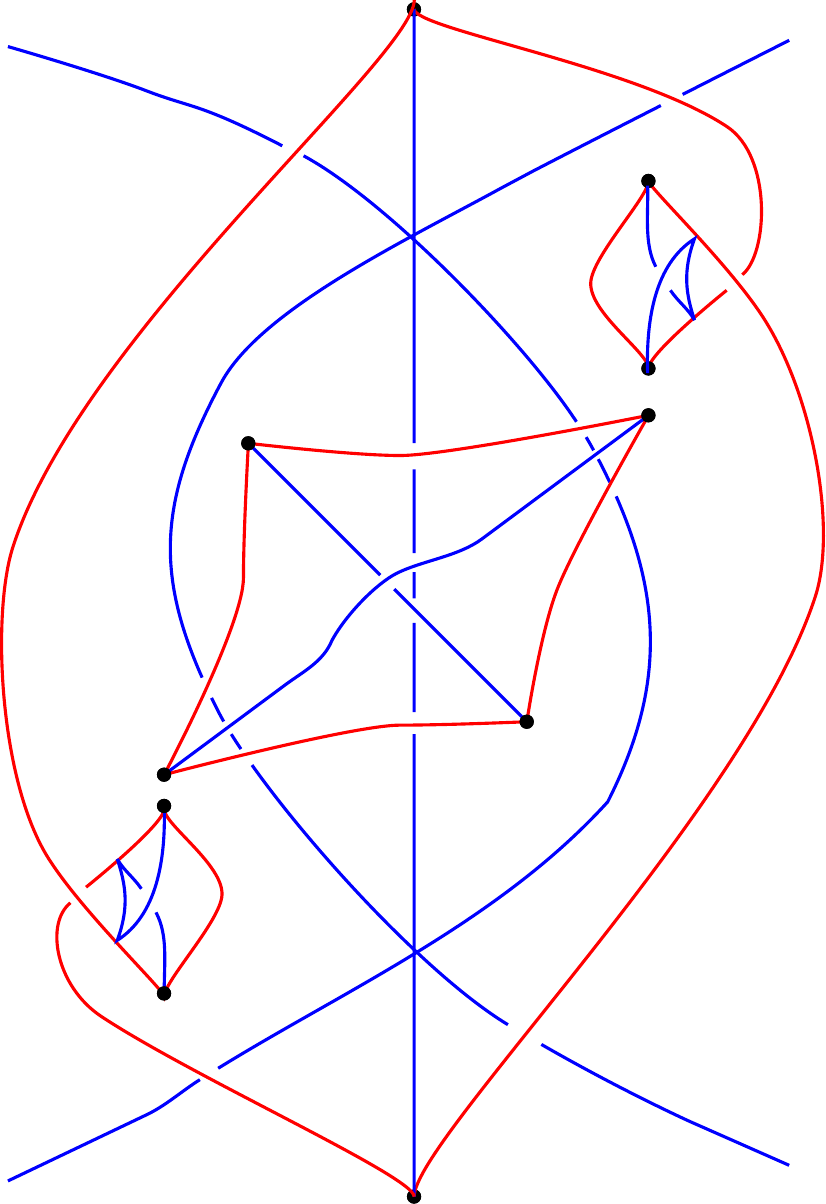}}
	\]	
	With the center of the ``X" from the resolved cone point positioned directly above the straight vertical crossing arc, this base projection exhibits a $180^\circ$ rotational symmetry. For the next few steps, this symmetry is maintained, and  we only present the upper half of the base projection.  
Note that the light green indications that appear in the illustrations are intended to help the reader visualize the change that the base projection undergoes during the isotopy, and are not themselves part of the base projection.
	
	We proceed by applying two $A_3A_1$-moves and the $A_3$-max I move:
	\[
	\raisebox{-2cm}{\includegraphics[scale=.4]{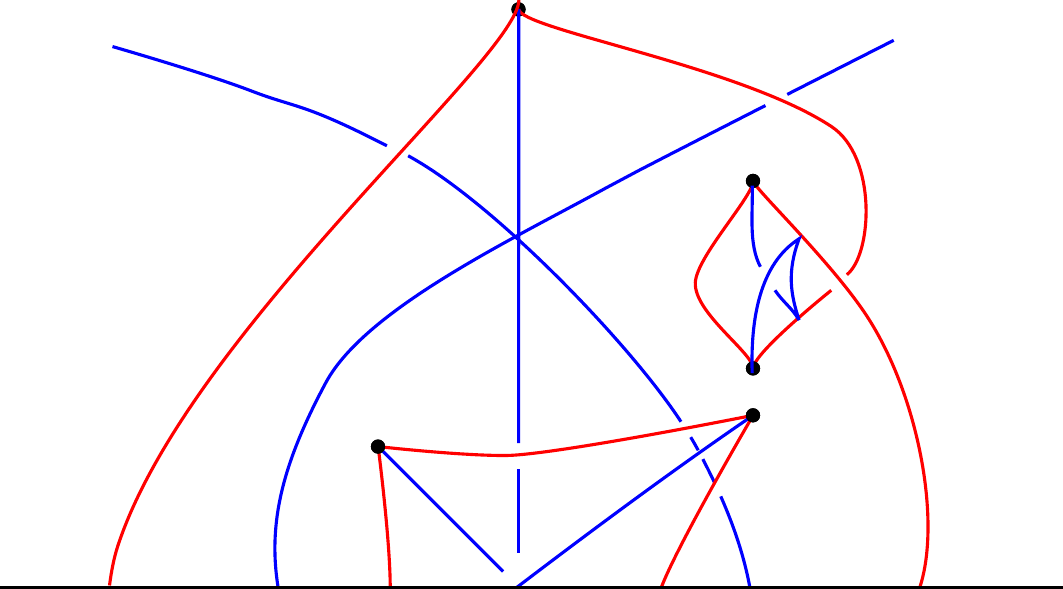}}\quad  \cong \quad \raisebox{-2cm}{\includegraphics[scale=.4]{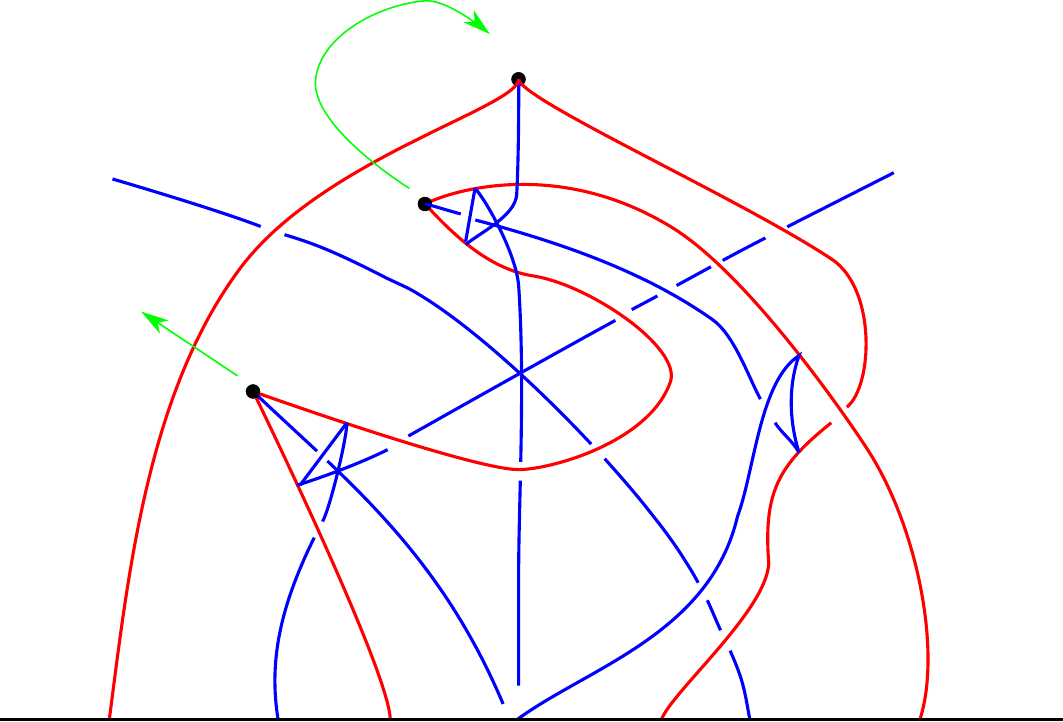}}  
	\]
	One more $A_3$-max I move produces:
	\[ \raisebox{-2cm}{\includegraphics[scale=.4]{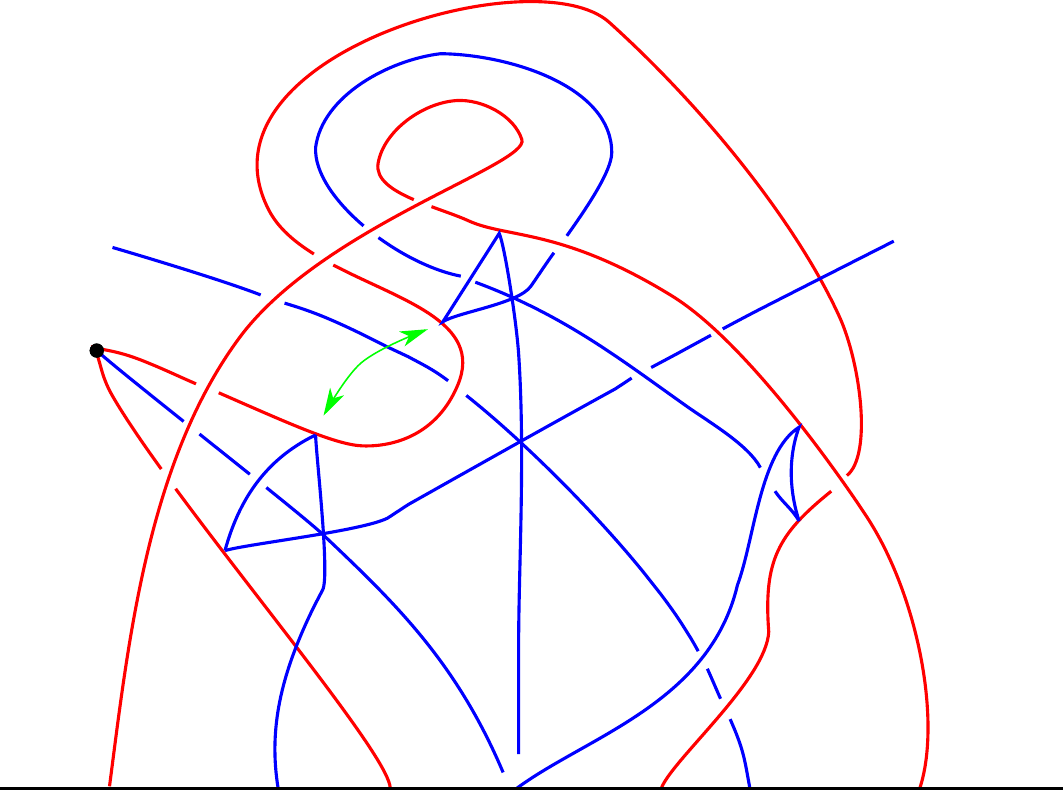}} 
	\]
		We next apply the $A_2A_1$-max move, return to drawing the entire base projection, and then apply the $A_2A_1$-max move again:
		\[
	\raisebox{-4cm}{\includegraphics[scale=.4]{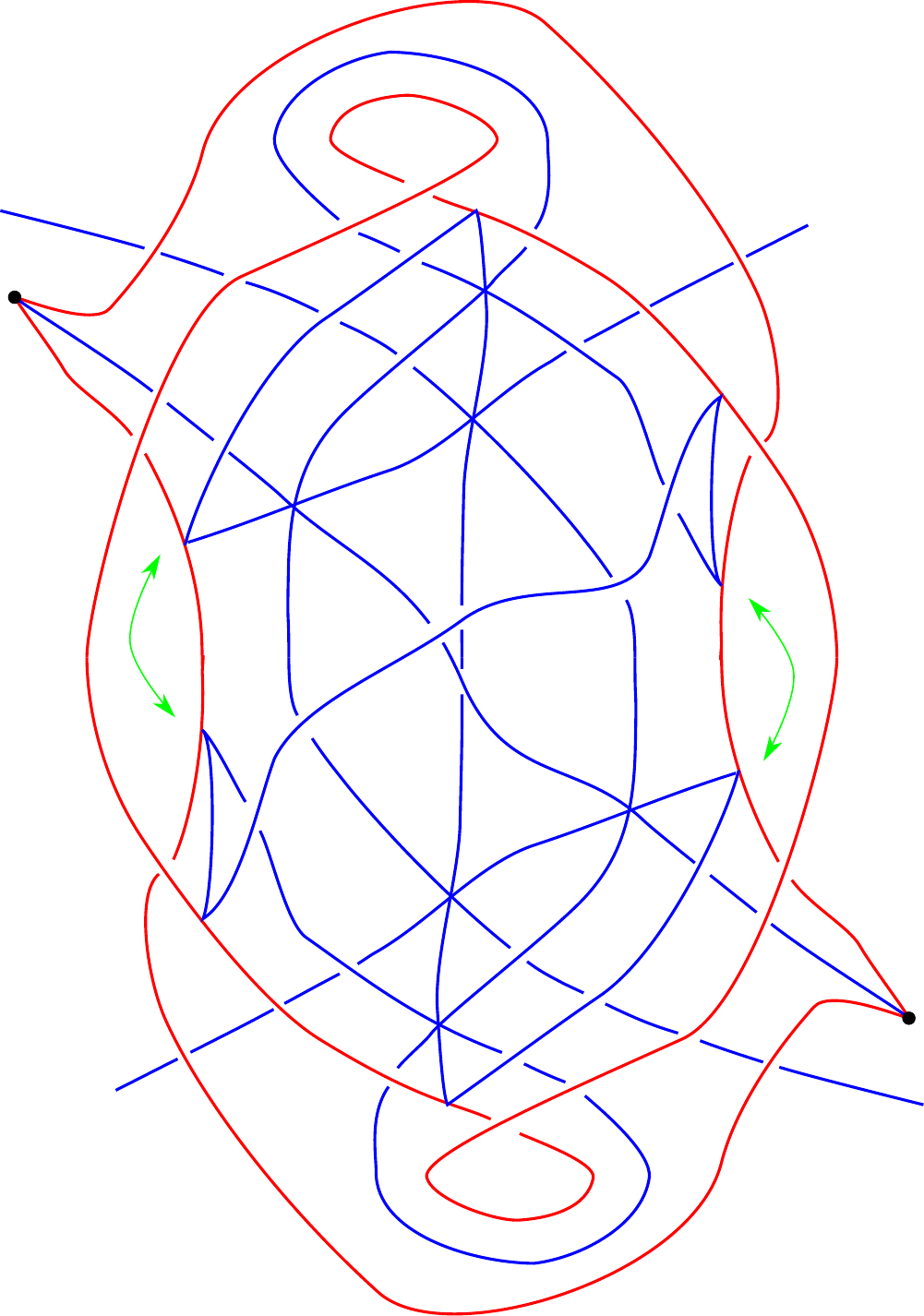}}\quad  \cong \quad \raisebox{-4cm}{\includegraphics[scale=.4]{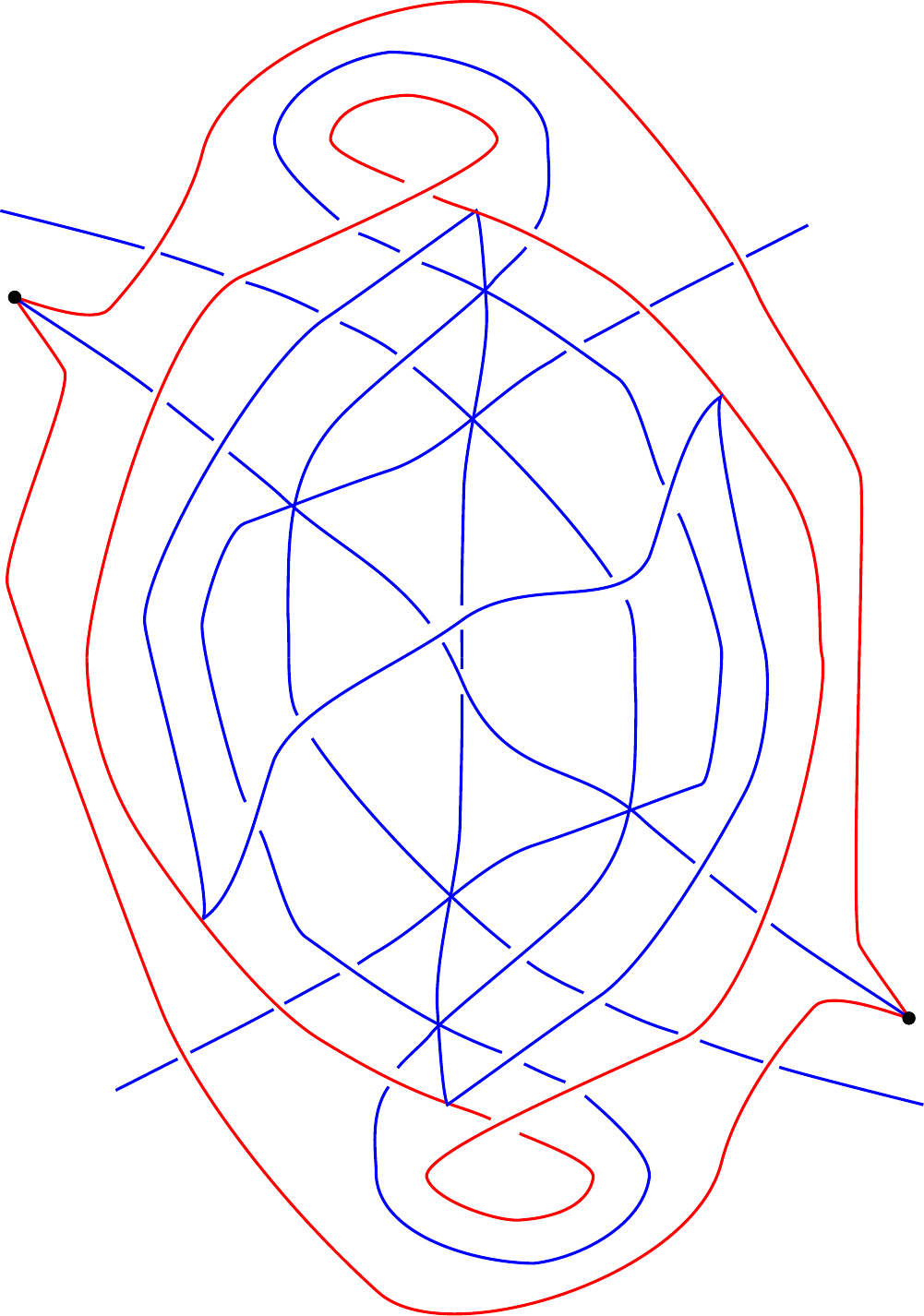}}  
	\]
	At this point in the isotopy we break the rotational symmetry by moving the circled triple point into the upper half of the projection.  
	\[
	\raisebox{-4cm}{\includegraphics[scale=.4]{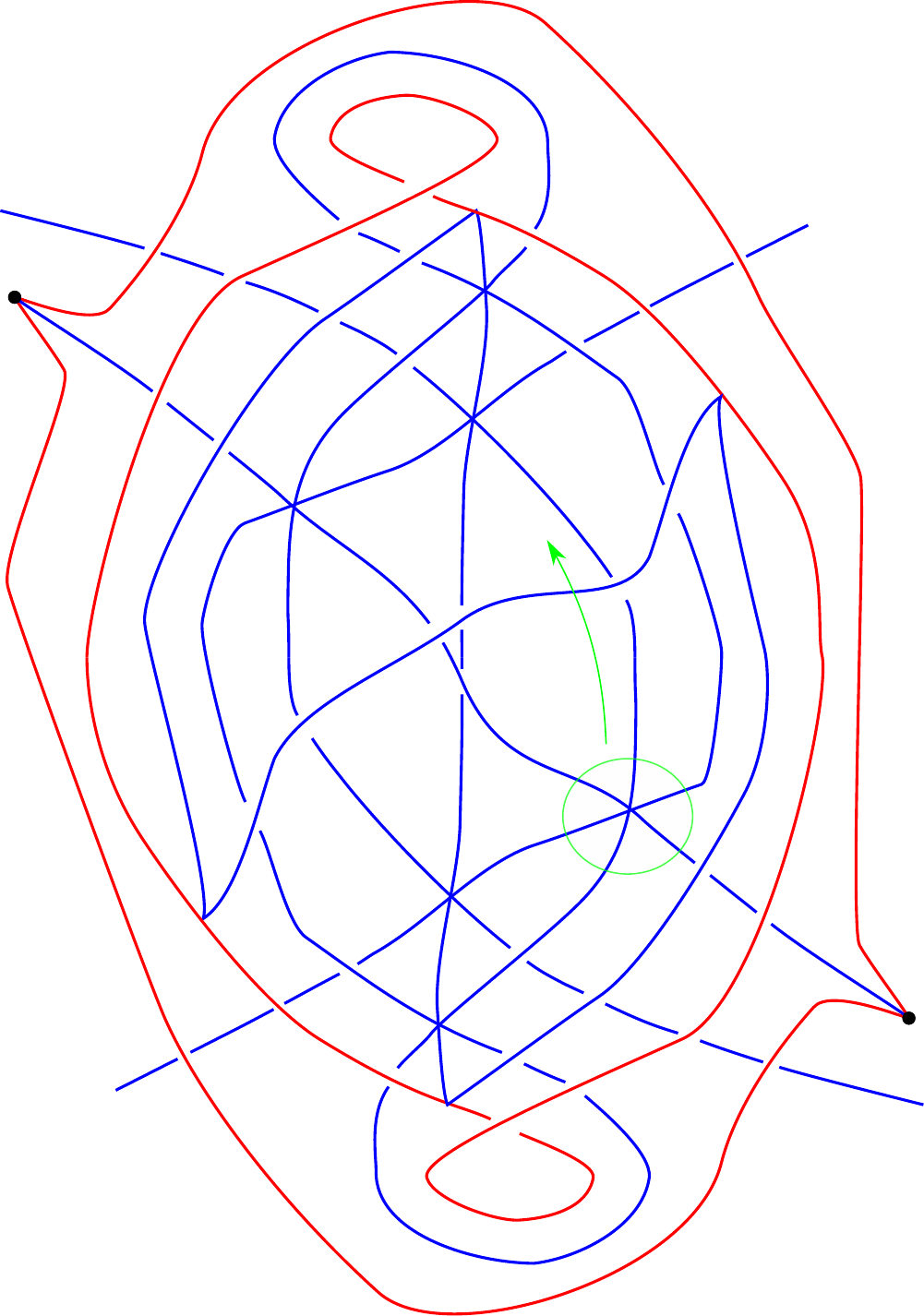}}\quad  \cong \quad \raisebox{-4cm}{\includegraphics[scale=.4]{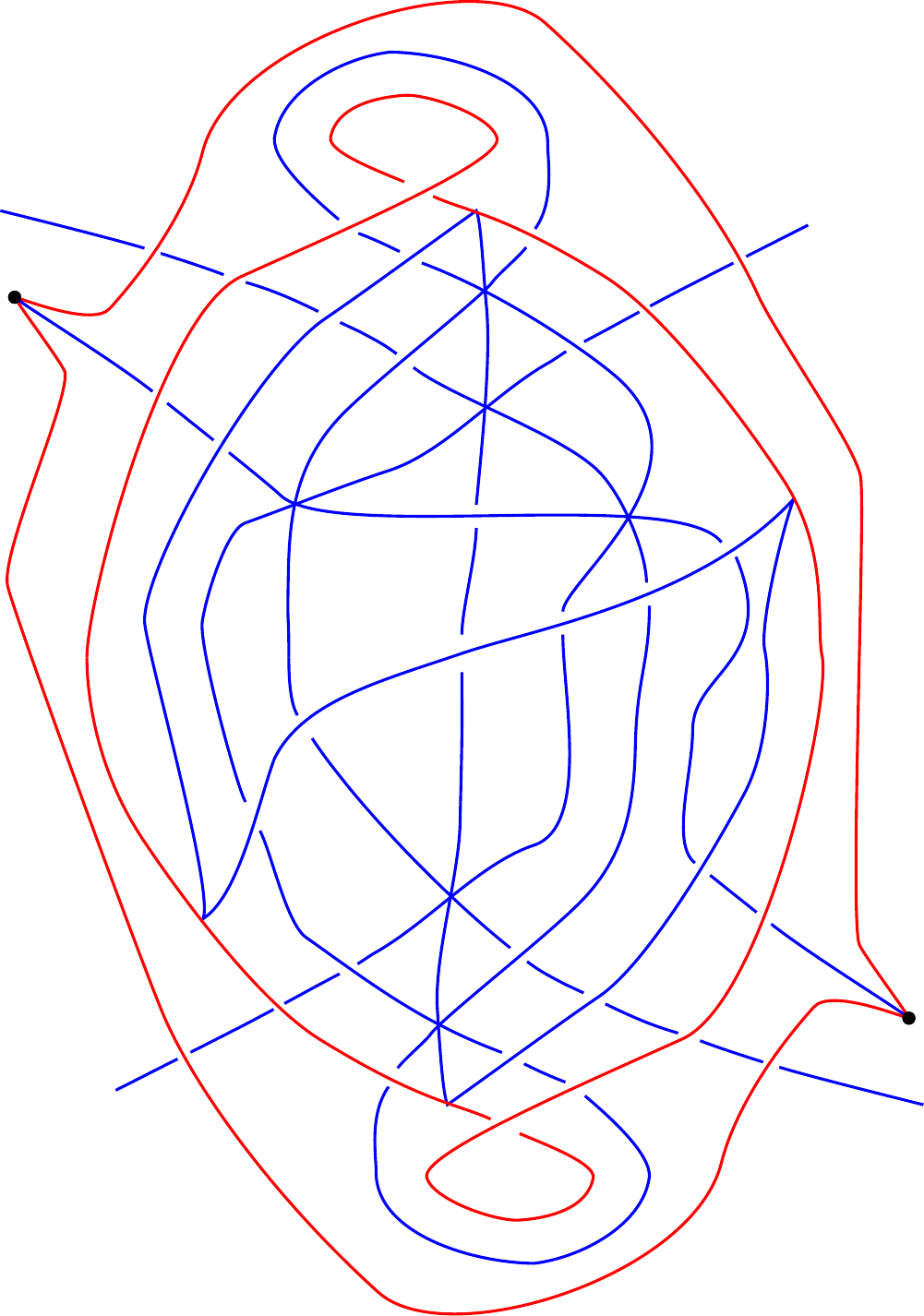}}  \quad 
	\]
	This allows us to then apply the $A_1^4$-move in the region outlined in green followed by a sequence involving two $A_1^3$-max moves.
	\[
	\raisebox{-4cm}{\includegraphics[scale=.4]{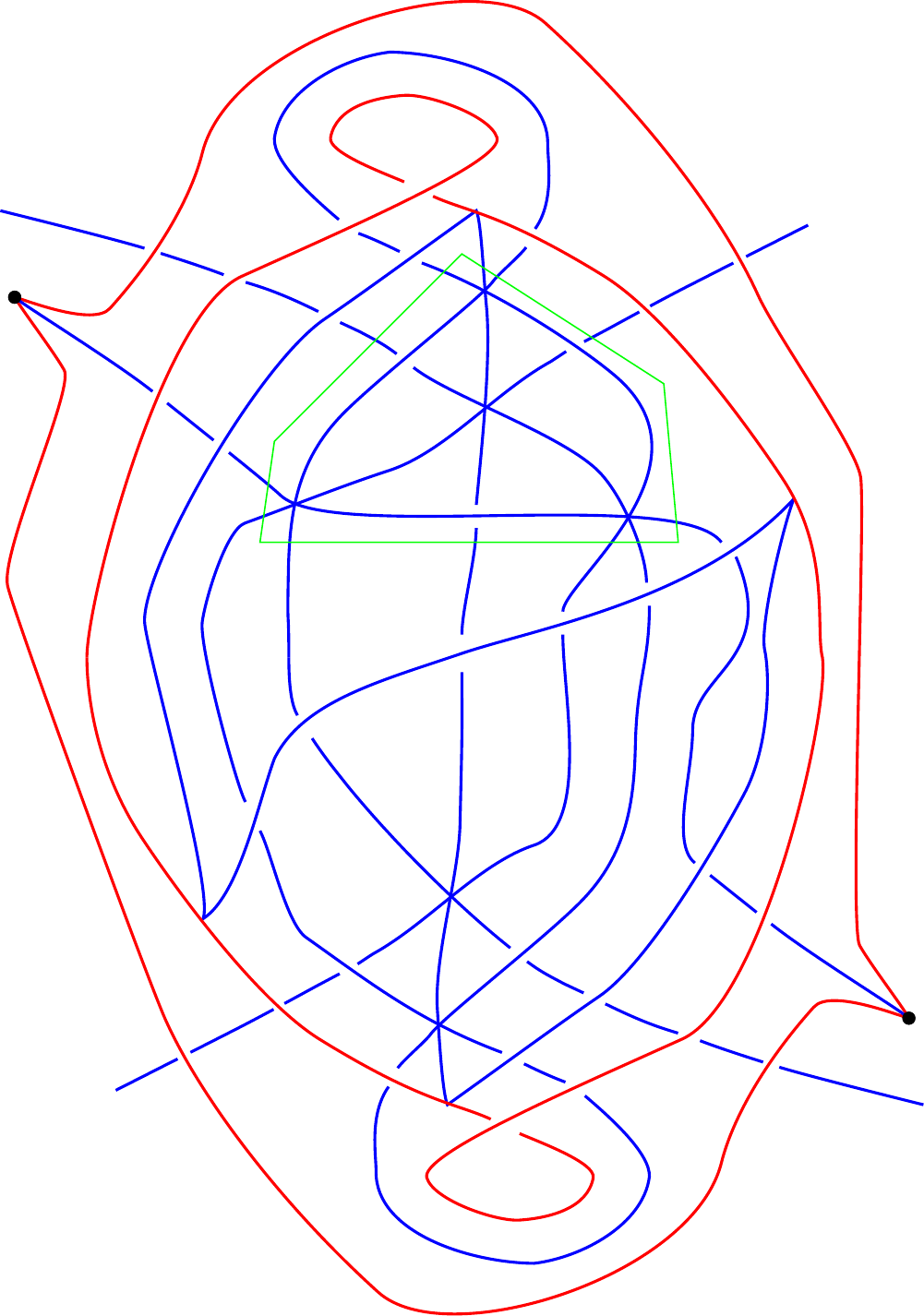}}  \quad  \cong \quad \raisebox{-4cm}{\includegraphics[scale=.4]{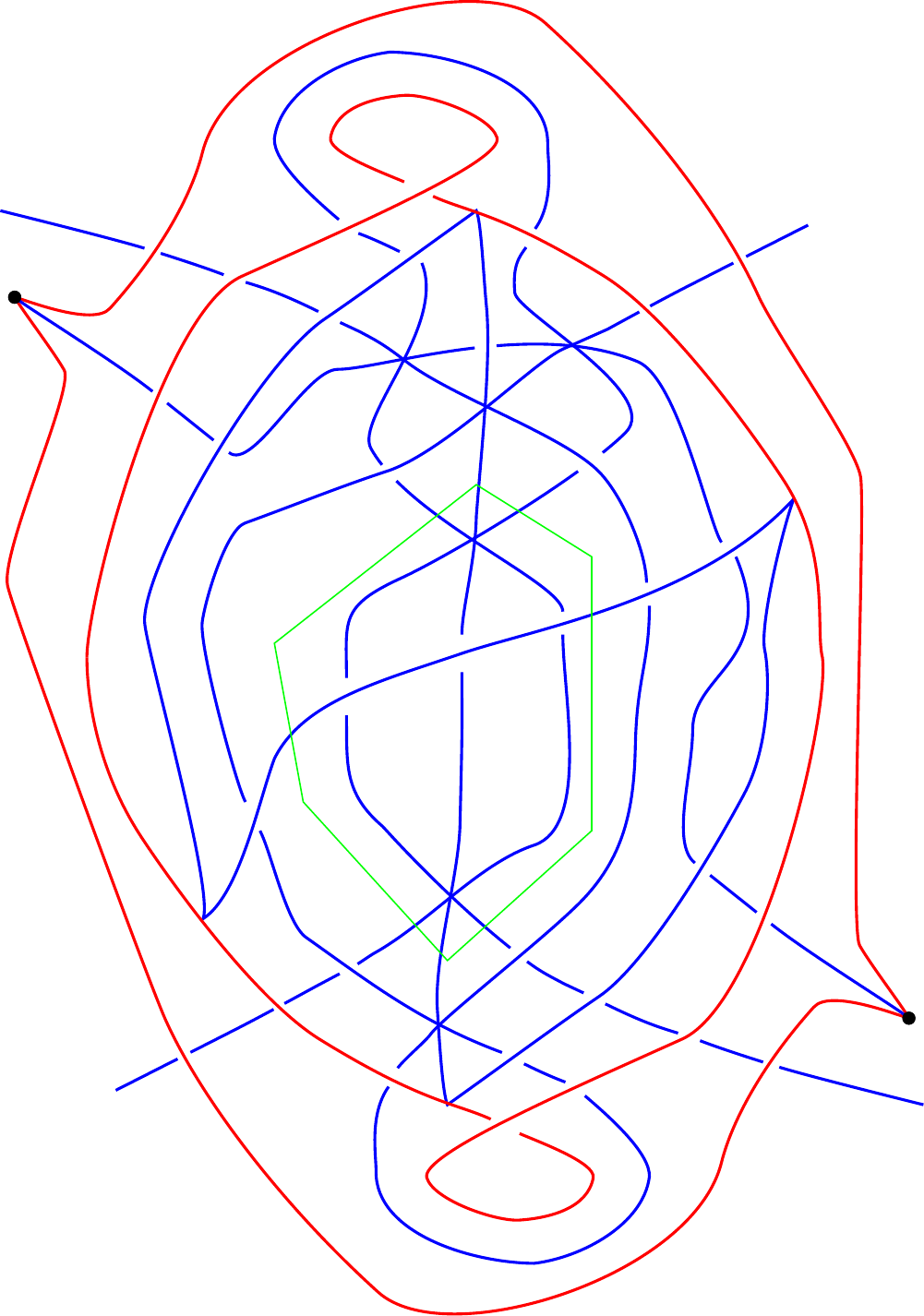}}  \quad \cong
	\]
	\[
	\raisebox{-4cm}{\includegraphics[scale=.4]{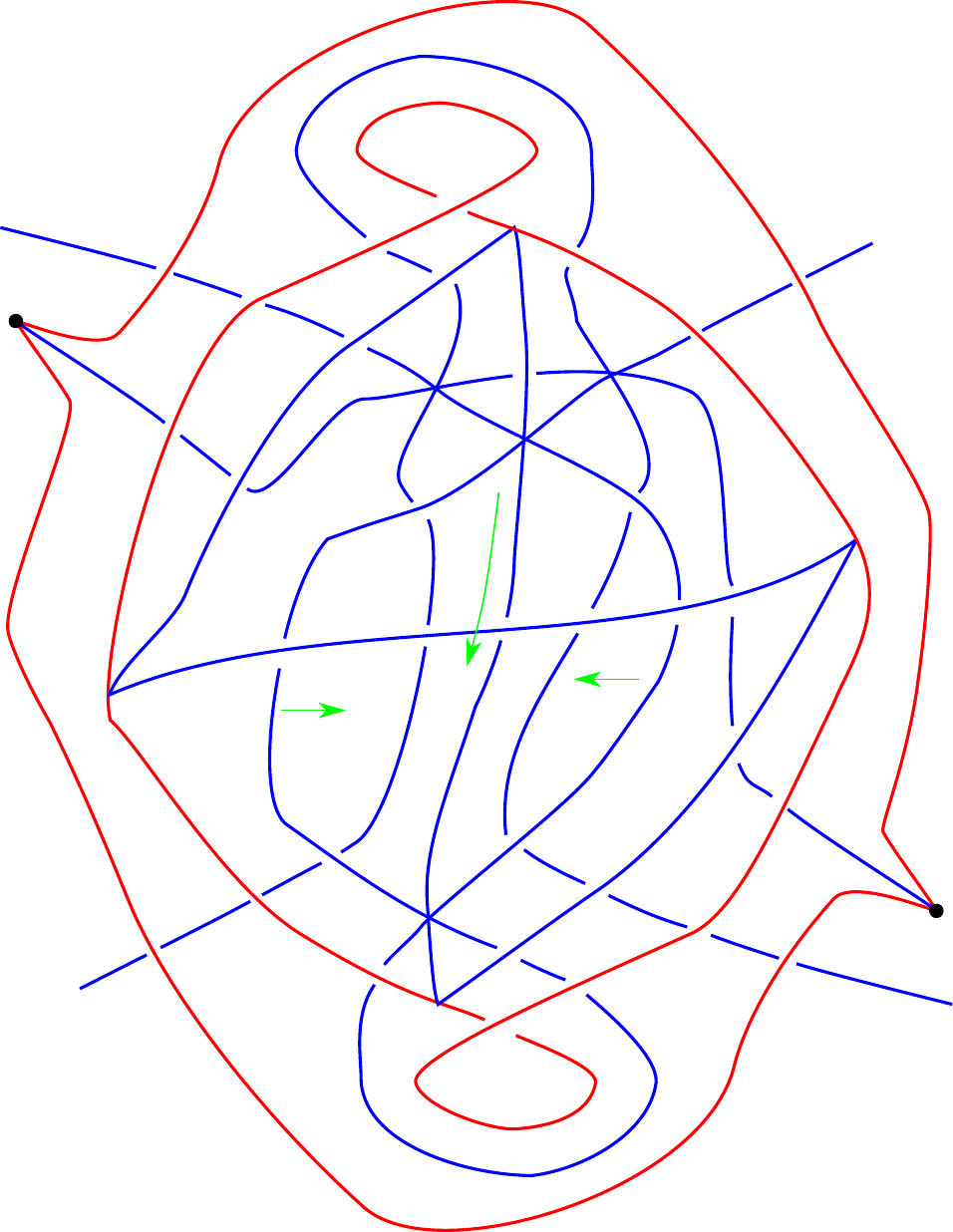}}  \quad  \cong \quad \raisebox{-4cm}{\includegraphics[scale=.4]{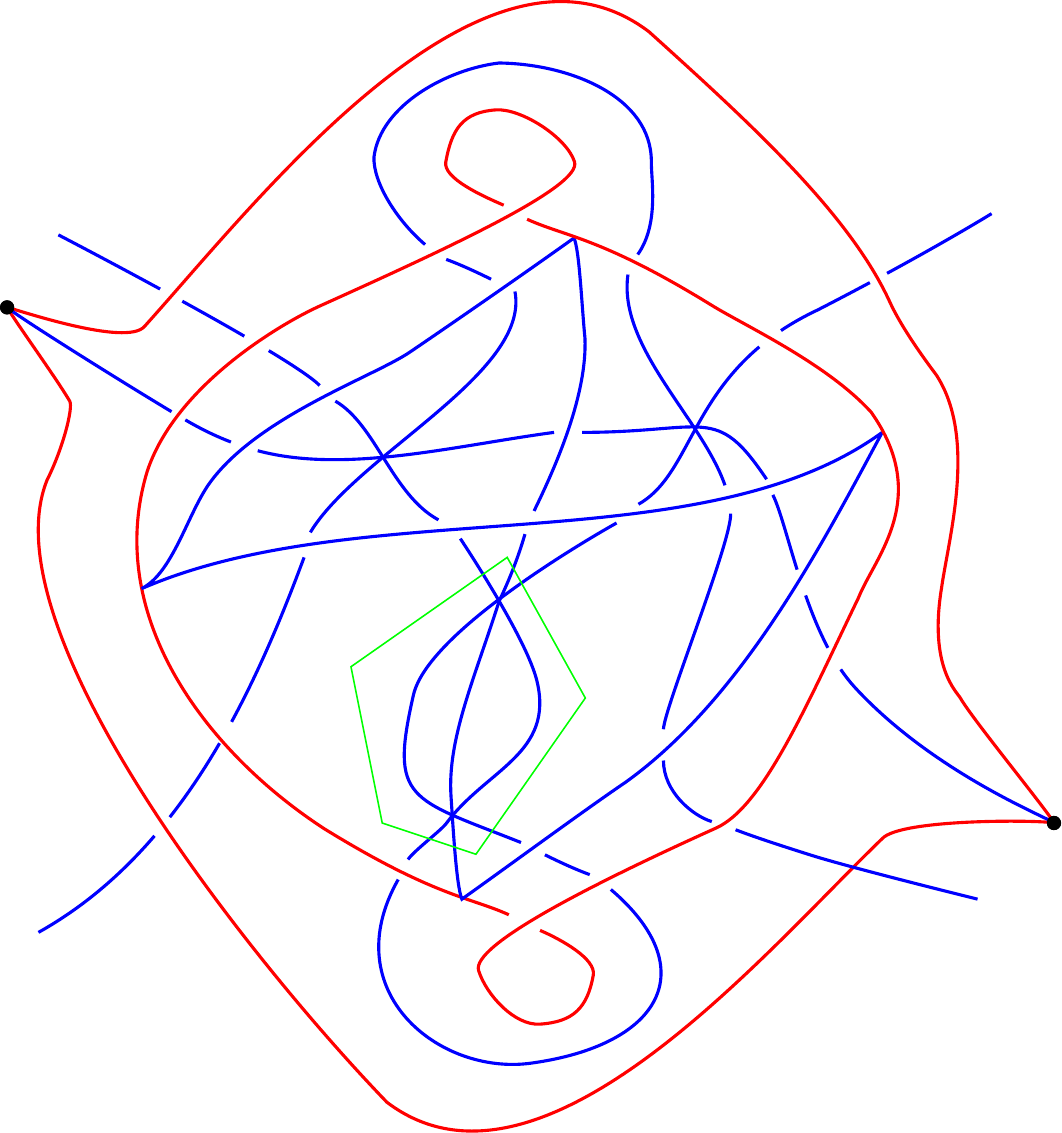}}  \quad \cong
	\]
	\[
	\raisebox{-4cm}{\includegraphics[scale=.4]{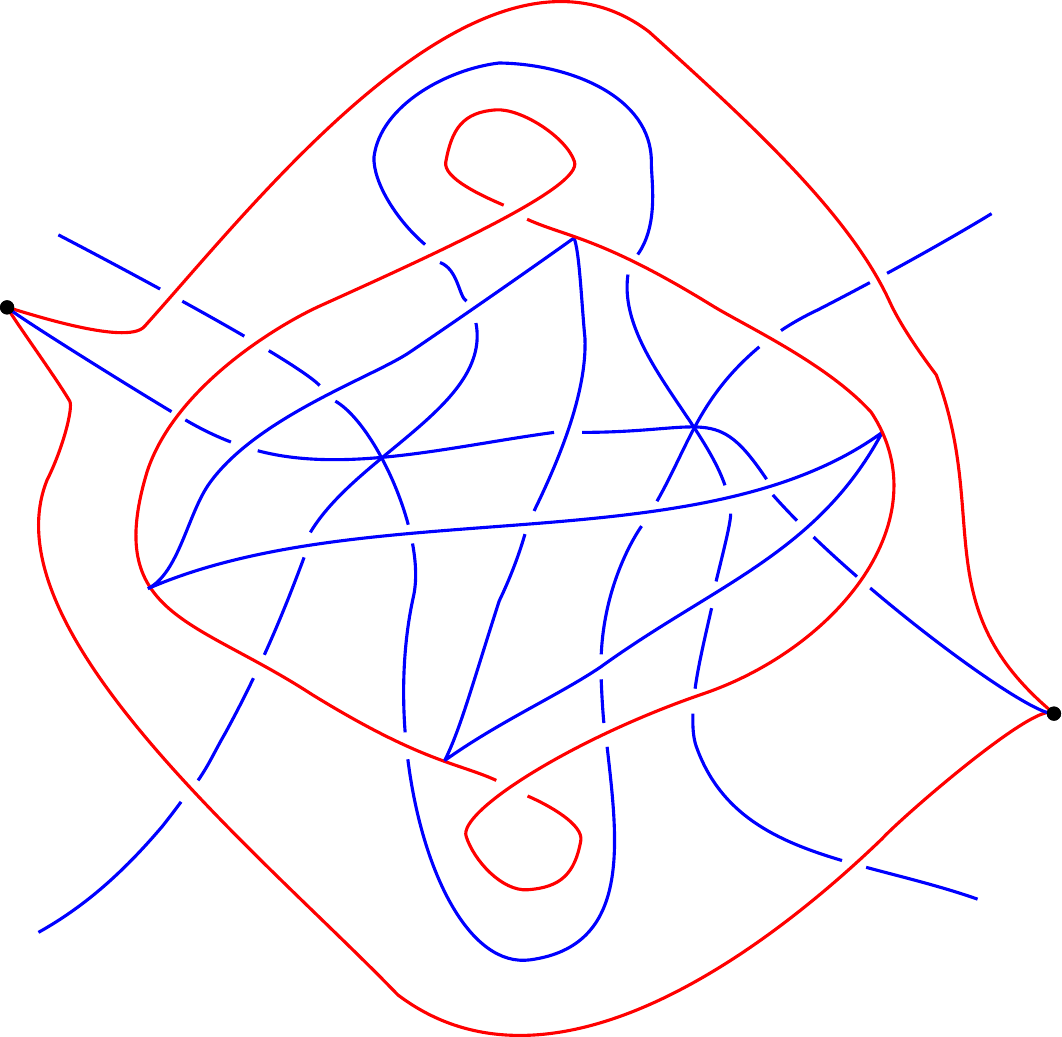}}  \quad  \cong \quad \raisebox{-4cm}{\includegraphics[scale=.4]{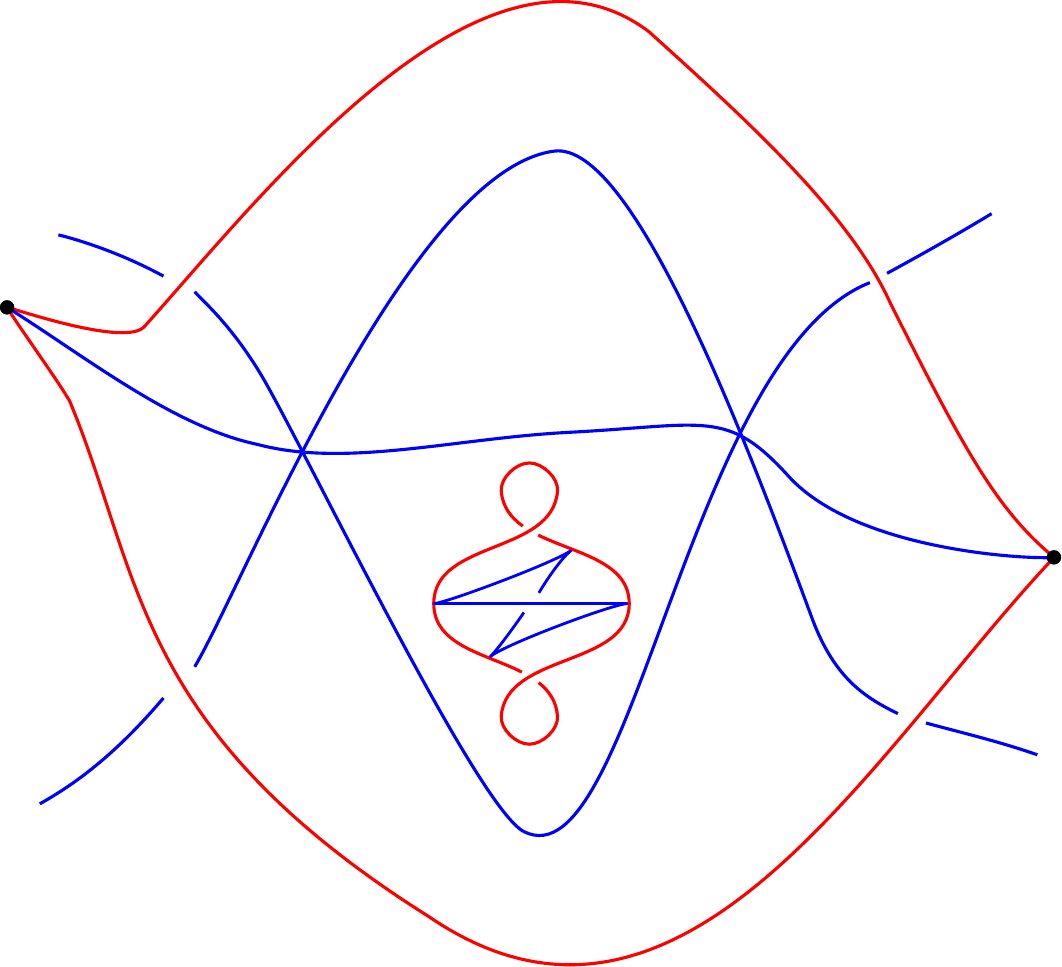}}  
	\]
	The last $\cong$ in the above sequence is via an ambient isotopy of the front projection in $S\times \R$.   Up to diffeomorphism, the singular set in the front projection is unchanged, but its projection  to the $x_1x_2$-plane is altered.  
	
	Finally, we apply the $A_2^2$-move and then the Legendrian isotopy from Lemma \ref{lem:lambda4} to arrive at $\Lambda_3$ via
	\[
	\raisebox{-3cm}{\includegraphics[scale=.4]{images/Isotopy14}}  \quad  \cong \quad \raisebox{-1.5cm}{\includegraphics[scale=.4]{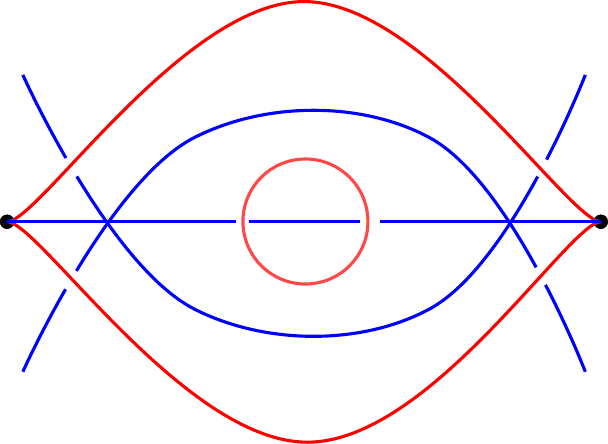}}  \quad \cong \quad \raisebox{-.5cm}{\includegraphics[scale=.4]{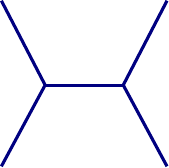}}
	\]
	
\end{proof}

\section{Proof of the remaining skein relations}  \label{sec:7}

With the skein relations (SR1) and (SR2) established in the previous two sections,
to complete the proof of Theorem \ref{thm:main} it  remains to establish (SR3) and (SR4) and the identities (\ref{eq:productrule}) and (\ref{eq:Uvalue}).

\subsection{The $0$-surgery skein relation (SR3)}  Let $\Lambda_1, \Lambda_2 \subset J^1S$ be Legendrians related as in the statement of the (SR3) skein relation from Theorem \ref{thm:main}.  In particular, the fronts of $\Lambda_1, \Lambda_2$ are identical outside of a $3$-ball $B \subset S\times \R$ where they appear as pictured in Figure \ref{fig:SR3cocycle}.  In the terminology of \cite{Rizell, BST}, $\Lambda_2$ is obtained from $\Lambda_1$ via ambient $0$-surgery.  Again, we will need to distinguish between two cases.

\medskip

\noindent {\bf Connected Case.}  The two disks  of $\Lambda_1$ pictured in (SR3) belong to the {\it same} connected component of $\Lambda_1$.

\medskip

\noindent {\bf Disconnected Case.}  The two disks of $\Lambda_1$ pictured in (SR3) belong to {\it different} connected components of $\Lambda_1$.

\medskip

\subsubsection{Combinatorial spin structures and geometric cocycles} 
We choose a combinatorial spin structure $\xi_1$ for $\Lambda_1$ whose base projection is disjoint from $\pi_x(B)$, and equip $\Lambda_2$ with the same structure $\xi_2 := \xi_1$.  Next, we assign collections of geometric cocycles, $\Gamma_1$ and $\Gamma_2$, to $\Lambda_1$ and $\Lambda_2$ in the following manner:  Choose a collection of geometric cocycles $\Gamma$ for $\Lambda_1$ that together span $H^1(\Lambda_1)$ and have their base projections disjoint from $\pi_x(B)$.  An additional cocycle $\mu$ is added to both $\Lambda_1$ and $\Lambda_2$ in the part of $B$ where $\Lambda_1$ and $\Lambda_2$ coincide; see Figure \ref{fig:SR3cocycle}.  We set $\Gamma_1 = \Gamma \cup \{\mu \}$.  In the connected case, we choose a further geometric cocycle $\lambda_2 \subset \Lambda_2$ that intersects $\mu$ exactly once so that $\Gamma_1 \cup \{\lambda\}$ spans $H^1(\Lambda_2)$, and then set
\[
\Gamma_2 = \left\{\begin{array}{cr}
	\Gamma_1 \cup \{ \lambda_2\},  &  \mbox{Connected Case}  \\
	\Gamma_1, &  \mbox{Disconnected Case}.
 	\end{array} \right.
\]

\begin{figure}
	
	\labellist
	\small
	\pinlabel $\Lambda_1$ [t] at 86 0
	\pinlabel $\mu$ [bl] at 166 88
		\endlabellist
	\centerline{\includegraphics[scale=.8]{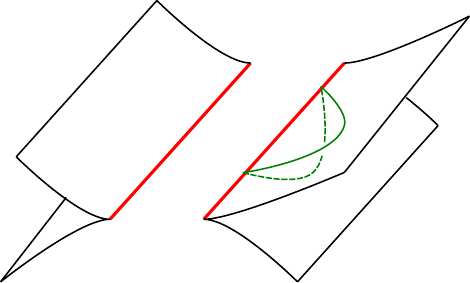} \quad
	\labellist
	\small
	\pinlabel $\Lambda_2$ [t]  at 86 0
	\pinlabel $\mu$ [bl] at 166 88
	\pinlabel $\lambda_2$ [br] at 44 102
	\endlabellist
		\includegraphics[scale=.8]{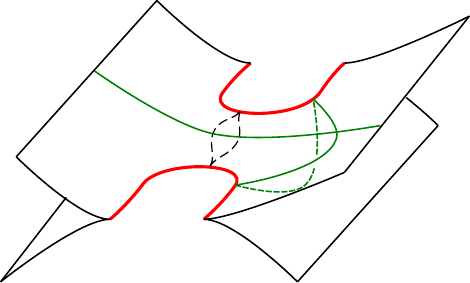}
	 }
	
	\caption{The geometric cocycles $\mu$ and (in the Connected Case) $\lambda_2$.   
	}
	\label{fig:SR3cocycle}
\end{figure}
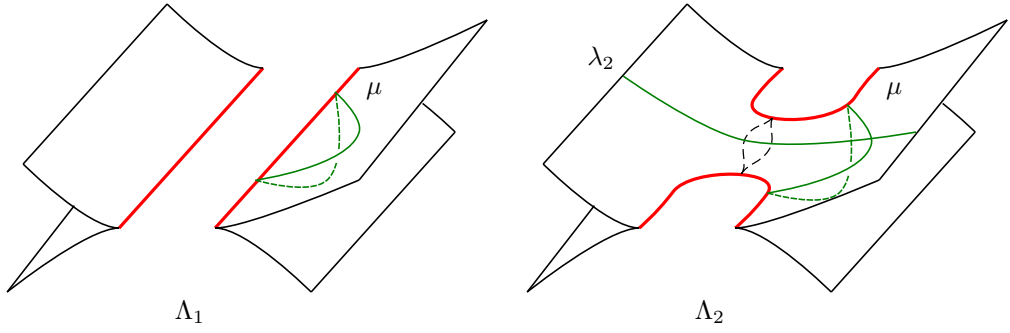

\subsubsection{Cellular computations}
For comparing the augmentation numbers of $\Lambda_1$ and $\Lambda_2$, we make use of CW-decompositions $\mathcal{E}_1$ and $\mathcal{E}_2$ that agree outside of $\pi_x(B)$ and within $\pi_x(B)$ appear as in Figure \ref{fig:SR3base}.  We will pass from the cellular DGAs $\mathcal{A}^{CW}(\Lambda_1, \Gamma_1, \mathcal{E}_1, \xi_1)$ and $\mathcal{A}^{CW}(\Lambda_2, \Gamma_2, \mathcal{E}_2, \xi_2)$ to stable tame isomorphic quotients that we will denote simply as $\alg(\Lambda_1)$ and $\alg(\Lambda_2)$.

\medskip

\noindent {\it Constructing $\alg(\Lambda_1)$}:  The equation
\[
J_{B_0} \cdot \partial \framebox{$B_0$} = \framebox{$A_2$} - A_1 +(A_2B_0-B_0A_1)
\]
allows us to apply Proposition \ref{prop:cancel} to cancel all of the $B_0$ and $A_2$ generators and arrive at the relations
\[
B_0 \doteq 0, \quad \mbox{and} \quad  A_2 \doteq A_1.
\]
We define $\alg(\Lambda_1)$ to be this quotient DGA.

\medskip

\noindent {\it Constructing $\alg(\Lambda_2)$}:  We begin with some notations.  To treat the Disconnected and Connected cases together,  in the Disconnected Case we let $\lambda =1$, while in the Connected Case $\lambda$ denotes the element of the coefficient ring of $\mathcal{A}^{CW}(\Lambda_2, \Gamma_2, \mathcal{E}_2, \xi_2)$ corresponding to the geometric cocycle $\lambda_2$.  Let $n$ denote the number of sheets of $\Lambda_2$ above $B_1$, and let $k$ and $k+1$ be the numbering of the sheets that meet at the cusp edges.  Keeping in mind that cusp edges are not considered sheets, there are then $n-2$ sheets above $A_3$ and $A_4$.  We will use $B_1$ (resp. $A_3$ and $A_4$) for the corresponding $n\times n$ matrix (resp. $(n-2) \times (n-2)$ matrices) of generators of $\mathcal{A}^{CW}(\Lambda_2, \Gamma_2, \mathcal{E}_2, \xi_2)$.   We reprise  $\Theta^\lambda_k$ for the identity matrix with the $k$-th diagonal entry replaced with $\lambda$.  Using the recipe from (\ref{eq:dB}) we then have
\begin{equation} \label{eq:0-surgery-diff}
J_{B_1}\cdot \partial B_1 = \left( \widehat{A}_3\right)_{k,k+1} (\Theta^\lambda_k + B_1) - (\Theta^\lambda_k + B_1) \left( \widehat{A}_4\right)_{k,k+1}
\end{equation}
where $\left( \widehat{A}_3\right)_{k,k+1}$ denotes $A_3$ enlarged to an $n\times n$ matrix by adding two new columns and rows of $0$'s in positions $k$ and $k+1$ along with a $1$ in the $(k,k+1)$-entry.  The matrix $\left( \widehat{A}_4\right)_{k,k+1}$ is defined similarly.

We pass from $\mathcal{A}^{CW}(\Lambda_2, \Gamma_2, \mathcal{E}_2, \xi_2)$ to a stable tame isomorphic quotient $\mathcal{A}(\Lambda_2)$ via three rounds of cancellations.

\medskip

\noindent {\bf Step 1.}  For all $i<k$, cancel $b^1_{i,k+1}$ and $b^1_{i,k}$.  This is done by repeated applications of Proposition \ref{prop:cancel}, starting first with $i=k-1$ and then inductively decreasing $i$.  From (\ref{eq:0-surgery-diff}), we have
\[
\pm \partial \framebox{$b^1_{i,k+1}$} = -\framebox{$b^1_{i,k}$} + \sum_{i<j<k} a^3_{i,j} b^1_{j,k+1} \doteq  b^1_{i,k},
\]
so that in the quotient $b^1_{i,k+1} \doteq b^1_{i,k} \doteq 0$.

\medskip

\noindent {\bf Step 2.}  For all $k+1<j$, cancel $b^1_{k,j}$ and $b^1_{k+1,j}$.  Starting with $j=k+2$ and inductively increasing $j$, we can cancel
\[
\pm \partial \framebox{$b^1_{k,j}$} = \framebox{$b^1_{k+1,j}$} - \sum_{k+1<i<j} b^1_{k,i} a^4_{i-2,j-2} \doteq b^1_{k+1,j}.
\]
This results in $b^1_{k,j} \doteq b^1_{k+1,j} \doteq 0$.

\medskip

After Steps 1 and 2, all $b^1_{i,j}$ generators with $\{i,j\} \cap \{k,k+1\} \neq \emptyset$ have been cancelled, except for $b^1_{k,k+1}$.  We set
\begin{equation} \label{eq:b1lambda}
b:= b^1_{k,k+1}, \quad \mbox{and note} \quad  \pm \partial b = 1 - \lambda.
\end{equation}

\medskip

\noindent {\bf Step 3.} Cancel the remaining $b^1_{i,j}$, except for $b$, with the $a^3_{i,j}$ generators.  Note that the $b^1_{i,j}$ generators with $\{i,j\} \cap \{k,k+1\} = \emptyset$ are in bijection with the $a^3_{i,j}$ generators via $a^3_{i,j} \leftrightarrow b^1_{\tau(i), \tau(j)}$ where $\tau:\{1, \ldots, n-2\} \rightarrow \{1, \ldots, n\}$ has $\tau(i) = i $, if $i<k$, and $\tau(i) = i+2$, if $i \geq k$.  We then apply Proposition \ref{prop:cancel} (with the $A_3$ generators ordered to be larger than the $A_4$ generators and $b$ taken to be smallest in the ordering) to cancel the $(i,j)$-entries with $\{i,j\} \cap \{k,k+1\} =  \emptyset$ in
\[
J_{B_1}\cdot \partial \framebox{$B_1$} = \framebox{$\left( \widehat{A}_3\right)_{k,k+1}$} \Theta^\lambda_k + \left( \widehat{A}_3\right)_{k,k+1}B_1 - (\Theta^\lambda_k + B_1) \left( \widehat{A}_4\right)_{k,k+1}.
\]
In the quotient, we have $B_1 \doteq b E_{k,k+1}$, and (\ref{eq:0-surgery-diff}) becomes 
\[
(1-\lambda)E_{k,k+1} \doteq \left( \widehat{A}_3\right)_{k,k+1} (\Theta^\lambda_k + bE_{k,k+1}) - (\Theta^\lambda_k + bE_{k,k+1}) \left( \widehat{A}_4\right)_{k,k+1}.
\]
Ignoring the $k$ and $k+1$ rows and columns, this identity becomes simply
\[
0 \doteq A_3- A_4 \quad \Rightarrow \quad A_3 \doteq A_4.
\]

\begin{figure}
	
	\labellist
	\small
	\pinlabel $\Lambda_1$ [t] at 48 -2
	\pinlabel $B_0$ [b] at 48 40
	\pinlabel $A_1$ [r] at 30 38
	\pinlabel $A_2$ [l] at 66 38
	\endlabellist
	\centerline{\includegraphics[scale=1]{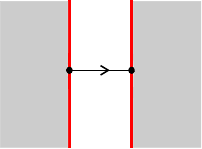} \quad
		\labellist
		\small
		\pinlabel $\Lambda_2$ [t]  at 48 -2
	\pinlabel $B_1$ [r] at 46 34
	\pinlabel $A_3$ [b] at 47 55
	\pinlabel $A_4$ [t] at 47 16
		\endlabellist
		\includegraphics[scale=1]{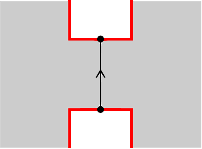}
	}
	
	\caption{Cellular decompositions of the base projections for $\Lambda_1$ and $\Lambda_2$.   
	}
	\label{fig:SR3base}
\end{figure}
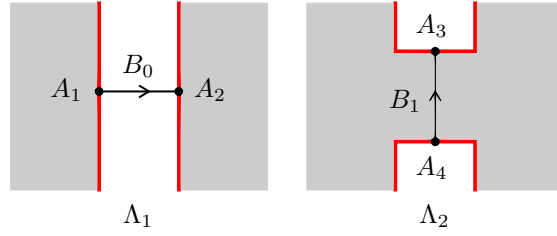

\begin{proof}[Proof of (SR3)] Identifying the $A_i$ generators from $\mathcal{A}(\Lambda_1)$ and $\mathcal{A}(\Lambda_2)$ via
	\[
	A_1 \doteq A_2  \leftrightarrow A_3 \doteq A_4,
	\]
the only consequential difference between $\mathcal{A}(\Lambda_1)$ and $\mathcal{A}(\Lambda_2)$ is from the presence of $b$ and $\lambda$ (in the Connected Case) in $\alg(\Lambda_2)$.  Specializing $\lambda =1$ (if necessary) and $b=0$, we have an isomorphism
\[
\alg(\Lambda_1) \cong \alg(\Lambda_2)\big|_{\lambda =1, b=0}.
\]
Note that every $0$-graded augmentation of $\alg(\Lambda_2)$ must have $\epsilon(\lambda) =1$, because of (\ref{eq:b1lambda}), and $\epsilon(b) =0$ because $b$ has graded degree $|b| =1$.  Therefore, we also have bijections of augmentation varieties
\[
V(\alg(\Lambda_1);\mathbb{F}) \cong V\left(\alg(\Lambda_2)\big|_{\lambda =1, b=0};\mathbb{F}\right) \cong V(\alg(\Lambda_2);\mathbb{F}). 
\]
Moreover, since $|b|=1$, shifted Euler characteristics are related by
\[
\chi^*(\mathcal{A}(\Lambda_2)) = \chi^*(\mathcal{A}(\Lambda_1))-1.
\]

To verify the skein relation, we use Definitions \ref{def:augLambda} and \ref{def:DGalgAugNum} and Proposition \ref{prop:stableinv} together with the above equalities	as follows:  

\medskip

\noindent{ \bf Connected Case.}  Compute

\begin{align}
\aug_{\Lambda_2}(q) &= (q-1)^{\dim H_0(\Lambda_2) - |\Gamma_2|} \aug(\mathcal{A}(\Lambda_2,q))  \notag \\
 &= (q-1)^{\dim H_0(\Lambda_2) - |\Gamma_2|} q^{-\chi^*(\mathcal{A}(\Lambda_2))/2}  |V(\Lambda_2,\mathbb{F}_q)|  \notag \\
 & = (q-1)^{\dim H_0(\Lambda_1) - (|\Gamma_1|+1)} q^{-(\chi^*(\mathcal{A}(\Lambda_1))-1)/2}  |V(\Lambda_1,\mathbb{F}_q)|  \label{eq:line3ofproof} \\
 & = (q-1)^{-1}q^{1/2} (q-1)^{\dim H_0(\Lambda_1)- |\Gamma_1|} q^{-\chi^*(\mathcal{A}(\Lambda_1))/2} |V(\Lambda_1,\mathbb{F}_q)| \notag \\
 & = \frac{q^{1/2}}{q-1} \aug_{\Lambda_1}(q). \notag
\end{align}

\medskip

\noindent{ \bf Disconnected Case.}	The calculation is the same except that to obtain (\ref{eq:line3ofproof}) we use that $\dim H_0(\Lambda_2) = \dim H_0(\Lambda_1)-1$ and $|\Gamma_2| = |\Gamma_1|$.
	\end{proof}

\begin{remark}
	\begin{enumerate}
		\item In the proof, we made use of the grading $|b|=1$ in establishing that the ($0$-graded) augmentation varieties satisfy $V(\alg(\Lambda_1);\mathbb{F}) \cong V(\alg(\Lambda_2);\mathbb{F})$.  This argument goes through for $m$-graded augmentation varieties with $m=0$ or $m \geq 2$, but cannot be used when $m=1$.  At present, we are unsure whether the isomorphism is valid when $m=1$.
		  
		\item  An alternative approach to establishing (SR3) could be based on a suitable extension of Theorem 1.1 and Corollary 1.2 from \cite{Rizell} to allow for $\Z[H_1(\Lambda)]$-coefficients.
	\end{enumerate}
	\end{remark}

\subsection{The zig-zag skein relation (SR4)}  The relation is easily verified from the cellular DGA perspective as follows.  Consider a $1$-cell $B_0$ stretching across the region in between the two cusp edges as in Figure \ref{fig:ZigZagbase}.  Let $k,k+1,k+2$ be the numbering above $B_0$ of the three sheets whose closures contain cusp edge points.  Then, from $J_{B_0} \cdot \partial B_0 = \left(\widehat{A}_1\right)_{k,k+1}(I+B_0) - (I+B_0) \left(\widehat{A}_0\right)_{k+1,k+2}$ with matrix notations as in (\ref{eq:0-surgery-diff}),  we have that
\[
\pm b^0_{k,k+1} = 1 -0   \quad \mbox{and} \quad  \pm b^0_{k+1,k+2} = 0-1.
\]
 Either one of these shows that the DGA has no augmentations.  Thus, the augmentation numbers are $0$.

\begin{figure}
	
	\labellist
	\small
	\endlabellist
	\centerline{\includegraphics[scale=1]{images/Stabilized} 
			\quad
		\labellist
	\small
	\pinlabel $B_0$ [b] at 48 40
	\pinlabel $A_1$ [r] at 30 38
	\pinlabel $A_2$ [l] at 66 38
	\endlabellist
	\includegraphics[scale=1]{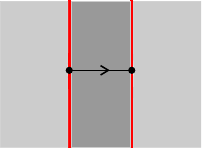} 
		}
	
	\caption{The front and base projections for the (SR4) skein relation.}  
	
	\label{fig:ZigZagbase}
\end{figure}

\subsection{Disjoint unions and the unknot}  To complete the proof of Theorem \ref{thm:main}, we establish the identities (\ref{eq:productrule}), concerning disjoint unions (displaceable in the base projection), and  (\ref{eq:Uvalue}), specifying the value on the unknot.  

\begin{proof}[Proof of (\ref{eq:productrule})]
When $\Lambda_1$ and $\Lambda_2$ have disjoint base projections, we have $\chi^*(\Lambda_1 \sqcup \Lambda_2) = \chi^*(\Lambda_1) + \chi^*(\Lambda_2)$.  Using the cellular DGA (or the original version of the LCH DGA) it is clear that $V(\Lambda_1 \sqcup \Lambda_2;\mathbb{F}) = V(\Lambda_1;\mathbb{F}) \times V(\Lambda_2;\mathbb{F})$.  [The reason is that $\mathcal{A}(\Lambda_1 \sqcup \Lambda_2) = \mathcal{A}(\Lambda_1) \otimes_\Z \mathcal{A}(\Lambda_2)$.  Alternatively, one can reason from the perspective of Remark \ref{rem:Faugs}.]
\end{proof}

\begin{proof}[Proof of (\ref{eq:Uvalue})]
	The DGA of the unknot has a single generator in degree $2$ with vanishing differential, and hence has a unique augmentation.  Using  Definitions \ref{def:augLambda} and \ref{def:DGalgAugNum} directly
	we find $\aug_U(q) = (q-1)q^{-1/2}$.	
\end{proof}

\section{First Examples}  \label{sec:firstex}

In this section, we illustrate the use of the skein relations from Theorem \ref{thm:main} to compute the augmentation numbers for an infinite class of examples introduced and studied by Dmitroglou-Rizell in \cite{Rizell1}.  For comparison, these augmentation numbers can also be computed via the direct study of the Chekanov-Eliashberg DGA and its augmentation variety in \cite[Section 4]{Rizell1}.
As done in the statement of Theorem \ref{thm:main}, when making computations with skein relations we may simply draw the front projection of $\Lambda$ in place of writing $\aug_\Lambda(q)$.

\medskip

\noindent {\bf 1. The standard torus.}  The front projection of a standard ``unknotted'' Legendrian torus, $T_{\mathit{std}}$, is obtained from spinning a standard $1$-dimensional Legendrian unknot $U$ around a vertical axis that is disjoint from $U$.  Using (SR3) and (\ref{eq:Uvalue}), we compute the augmentation numbers, $\aug_{T_\mathit{std}}(q)$, to be
\[
\fig{images/Y6} \,= \, (q^{1/2}-q^{-1/2})^{-1} \fig{images/Y3} = (q^{1/2}-q^{-1/2})^{-1}(q^{1/2}-q^{-1/2})=1.	
\]

\medskip 

\noindent {\bf 2. The cone point torus.}  Next, consider the Legendrian torus, $T_{\mathit{cone}}$, with front projection obtained from the standard $2$-dimensional Legendrian unknot by connecting the upper sheet with the lower sheet via a single cone point.  Alternatively, it is obtained from revolution of a symmetric $1$-dimensional front projection with a single crossing as 
\[
\includegraphics[scale=.8]{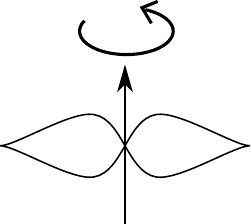}\quad \raisebox{1cm}{=} \quad \raisebox{.5cm}{\includegraphics[scale=.8]{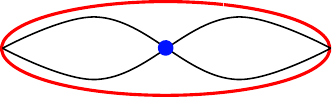}}
\]
where the blue point is a cone point.

\begin{remark} The Legendrian $T_\mathit{cone}$ is sometimes referred to, e.g., in \cite{CZ}, as the Legendrian Clifford torus, though the phrase ``Clifford torus'' is also used to refer to certain (related) non-exact embedded Lagrangian tori.  The DGA of $T_\mathit{cone}$ is first studied in \cite{Rizell1}.  See also \cite{RizellGolovko} for an interpretation of $T_{\mathit{cone}}$ in terms of a $3$-fold cover of a monotone Clifford torus in $\C P^2$ and other interesting results concerning $T_{\mathit{cone}}$.
\end{remark}

Now, the (SR1) relation gives
\begin{equation}  \label{eq:CliffordStep}
q^{1/2}\, \raisebox{-.5cm}{\includegraphics[scale=.5]{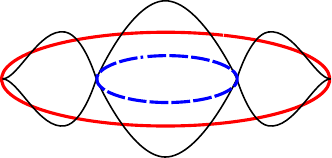}}- q^{1/2} \,\raisebox{-.7cm}{\includegraphics[scale=.5]{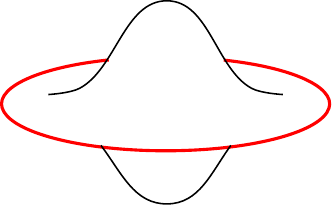}} = (q-1)^2 \left[ \raisebox{-.3cm}{\includegraphics[scale=.5]{images/Y7}}- \raisebox{-.3cm}{\includegraphics[scale=.5]{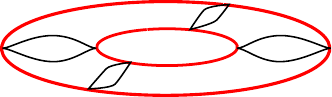}} \right]
\end{equation}
where the front projection in the first term is also a surface of revolution
\[
\includegraphics[scale=.8]{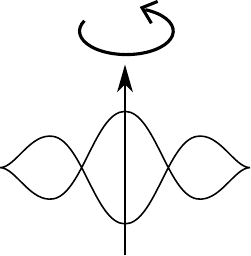}\quad \raisebox{1cm}{=} \quad \raisebox{.5cm}{\includegraphics[scale=.8]{images/Y5}}.
\]
The first term vanishes by (SR4), as the presence of a portion of a front projection of the form
\[
\raisebox{-.5cm}{\includegraphics[scale=.8]{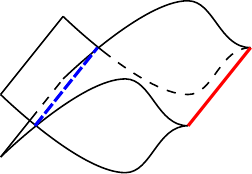}}
\]
implies the existence of a zig-zag after a Legendrian isotopy.  Indeed, a combination of the $A_3$-max II Move and then the $A_2A_1$-max Move result in the following changes to the movie of slices:
\[
\raisebox{-1.2cm}{\includegraphics[scale=.6]{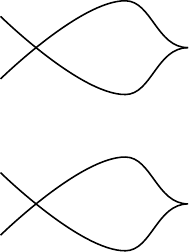}} \quad \quad \cong \quad \quad \raisebox{-1.8cm}{\includegraphics[scale=.6]{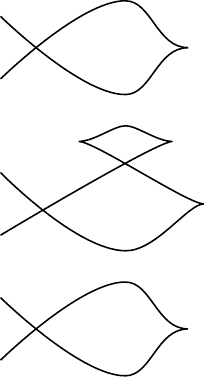}}  \quad \quad \cong \quad \quad
 \raisebox{-3.2cm}{\includegraphics[scale=.6]{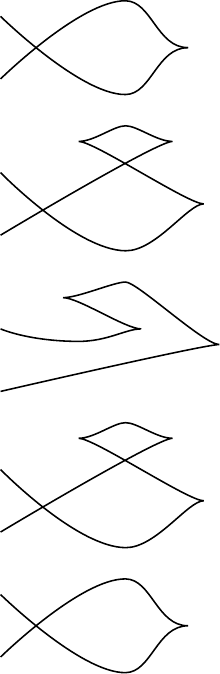}} 
\]

Using the evaluation of the unknot and $T_{\mathit{std}}$ from above, equation (\ref{eq:CliffordStep}) becomes
\[
q^{1/2} \cdot 0- q^{1/2}(q^{1/2}-q^{-1/2}) = (q-1)^2 \left[ \raisebox{-.5cm}{\includegraphics[scale=.8]{images/Y7}}- 1 \right]
\]
which leads to
\[
\raisebox{-.5cm}{\includegraphics[scale=.8]{images/Y7}}= \frac{q-2}{q-1}.
\]

\medskip

\noindent{\bf 3. Connected sums.}  In general, when $\Lambda_1, \Lambda_2 \subset J^1\R^2$ are connected Legendrian surfaces with disjoint base projections, we can form a connected sum of $\Lambda_1$ and $\Lambda_2$ by connecting a cusp edge of $\Lambda_1$ with a cusp edge of $\Lambda_2$ via an ambient $0$-surgery.  This procedure is introduced in \cite{EES05b}.  Moreover, Dimitroglou Rizell showed in \cite[Proposition 4.9 (1)]{Rizell} that, provided orientations are appropriately taken into account, connected sum is well defined on Legendrian isotopy classes.

\begin{proposition}  \label{prop:connectsum}
	For any connected sum $\Lambda_1 \# \Lambda_2$ with $\Lambda_1, \Lambda_2 \subset J^1\R^2$, we have
	\[
	\aug_{\Lambda_1 \# \Lambda_2}(q) = (q^{1/2}-q^{-1/2})^{-1} \aug_{\Lambda_1}(q) \cdot \aug_{\Lambda_2}(q).
	\]
\end{proposition}
\begin{proof}
	This follows from $\aug_{\Lambda_1 \sqcup \Lambda_2}(q) = \aug_{\Lambda_1}(q) \cdot \aug_{\Lambda_2}(q)$ together with (SR3) of Theorem \ref{thm:main}.
\end{proof}

For $a, b \geq 0$, we can form a Legendrian surface of genus $g=a+b$, denoted $T^a_{\mathit{std}} \# T^b_{\mathit{cone}}$, as the connected sum of $a$ copies of $T_{\mathit{std}}$ and $b$ copies of $T^b_{\mathit{cone}}$; see Figure \ref{fig:CliffordConnectSum}.  Using Proposition \ref{prop:connectsum}, we have
\begin{equation} \label{eq:connectsumaug}
\aug_{T^a_{\mathit{std}} \# T^b_{\mathit{cone}}} = (q^{1/2}-q^{-1/2})^{1-a-b} \left(\frac{q-2}{q-1}\right)^{b}.
\end{equation}

\begin{figure}
	
	\centerline{	
		\labellist
		\Large
		\pinlabel $\cdots$ at 158 60 
		\pinlabel $\cdots$  at 338 60
		\pinlabel $a$ [t] at 135 -2
		\pinlabel $b$ [t]  at 368 -2
		\endlabellist
		\includegraphics[scale=.7]{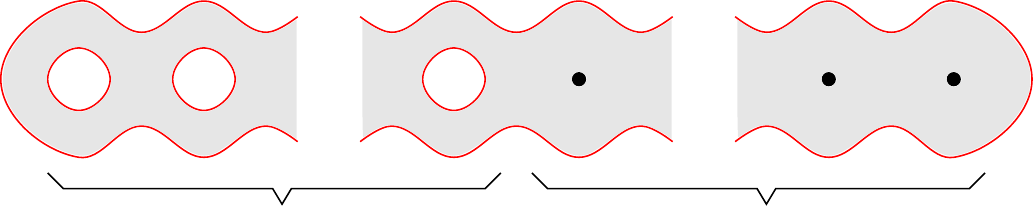}
	}
	
	\quad
	
	\caption{A connected sum of $a$ standard tori, $T_\mathit{std}$, and $b$ knotted tori pictured in the base projection.
	}
	\label{fig:CliffordConnectSum}
\end{figure}
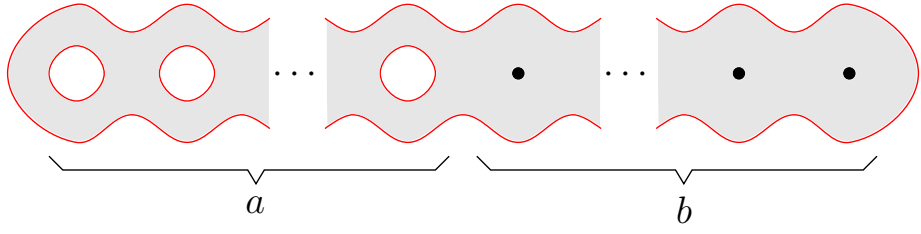

\begin{remark}
	In general, Legendrian surfaces do not have well defined factorizations into connected sums of prime factors.  Indeed, the loose (in the terminology of Murphy \cite{Murphy}) Legendrian sphere $\Lambda_{\mathit{loose}} = \raisebox{-.5cm}{\includegraphics[scale=.5]{images/Y5}}$ has the property that for any arbitrary Legendrian sphere $\Lambda_0 \subset J^1\R^2$ we have $\Lambda_{\mathit{loose}} \# \Lambda_0= \Lambda_{\mathit{loose}}$.  In contrast, the calculation of the augmentation numbers from (\ref{eq:connectsumaug}) shows that $T^a_{\mathit{std}} \# T^b_{\mathit{cone}} \cong T^{a'}_{\mathit{std}} \# T^{b'}_{\mathit{cone}}$ implies that $(a,b) = (a',b')$.
\end{remark}

\begin{question} Are there classes of Legendrian surfaces that {\it do} admit well defined prime factorizations, e.g., all non-loose Legendrian surfaces or all Legendrian surfaces that admit augmentations?
\end{question}

\section{Legendrian $2$-weaves and face colorings via skein relations}  \label{sec:2weaves}

In this section, after a quick review of the elegant construction of Legendrian surfaces in $J^1S^2$ and $J^1\R^2$ (called {\it Legendrian $2$-weaves}) from trivalent graphs in $S^2$, we establish Theorem \ref{thm:2weaves}.  The method of proof is to view the left and right sides of (\ref{eq:2weaves}) as isotopy invariants of trivalent graphs in $S^2$, and show that they satisfy some common properties (the relations from Proposition \ref{prop:Gproperties} below).  The conclusion of the proof applies an inductive argument to show that these properties uniquely characterize the invariants.

\subsection{Legendrian $2$-weaves}
Let $G \subset S^2$ be a smoothly embedded trivalent graph whose vertices have neighborhoods diffeomorphic to the base projection of the $D_4^-$ singularity (the center term in the $D_4^-$ Move above). 
We allow for $G$ to be disconnected and to have components that are simple closed curves without any vertices.  We refer to the components of $S^2\setminus G$ as {\bf faces} of $G$, noting that because we allow $G$ disconnected they are not always disks.

Following Treumann and Zaslow \cite{TZ}, a Legendrian surface $\widetilde{\Lambda}_G \subset J^1S^2$ may be constructed as follows:
\begin{itemize}  
\item above each face of $G$ the front projection, $\pi_{xz}(\widetilde{\Lambda}_G)$, consists of two non-intersecting sheets;
\item a crossing arc of $\pi_{xz}(\widetilde{\Lambda}_G)$ sits above each edge of $G$;
\item above each vertex of $G$ a $D_4^-$-singularity appears in the front projection.
\end{itemize}
We equip $\widetilde{\Lambda}_G$ with the $\Z$-valued Maslov potential that is equal to $0$ everywhere.

In addition, we can take the Legendrian satellite \cite{NgT} of $\widetilde{\Lambda}_G$ with the standard $2$-dimensional Legendrian unknot $U \subset J^1\R^2$ to produce a Legendrian $\Lambda_G \subset J^1\R^2$.  A front projection of $\Lambda_G$ is described as follows: (i)  Take a $2$-copy of $U$, i.e., take a union $U \cup U'$ where $U'$ is a small shift of $U$ in the $z$-direction.  Both components are assigned the same Maslov potential that takes the values $0$ and $1$ on the lower and upper half of the saucer.  (ii) Delete from $\widetilde{\Lambda}_G$ a neighborhood of the north pole in $S^2$ (assumed to be disjoint from $G$) and place the remainder of the front projection of $\widetilde{\Lambda}_G$ as the lower two sheets of the front projection of $U \cup U'$.  See also \cite{TZ, CZ}. Using the terminology from \cite{CZ}, either of  $\widetilde{\Lambda}_G$ or $\Lambda_G$ may be referred to as the {\bf Legendrian $2$-weave} associated to $G$.  See Figure \ref{fig:2Weave}.

\begin{figure}

\centerline{	
\includegraphics[scale=.45]{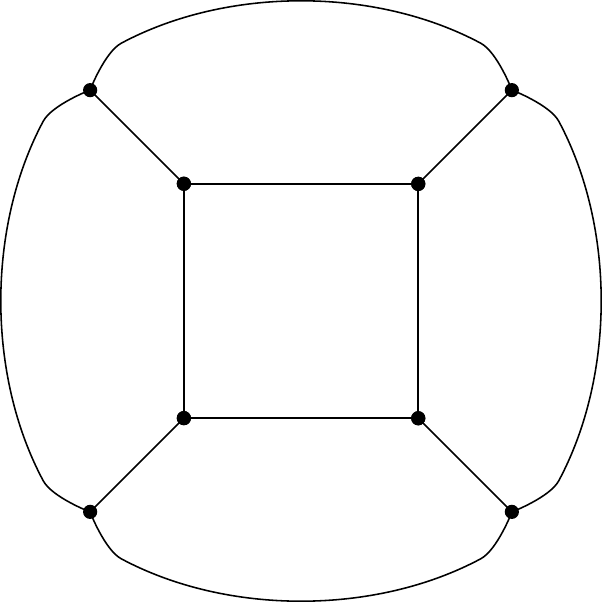}  \quad \quad \quad 
\includegraphics[scale=.45]{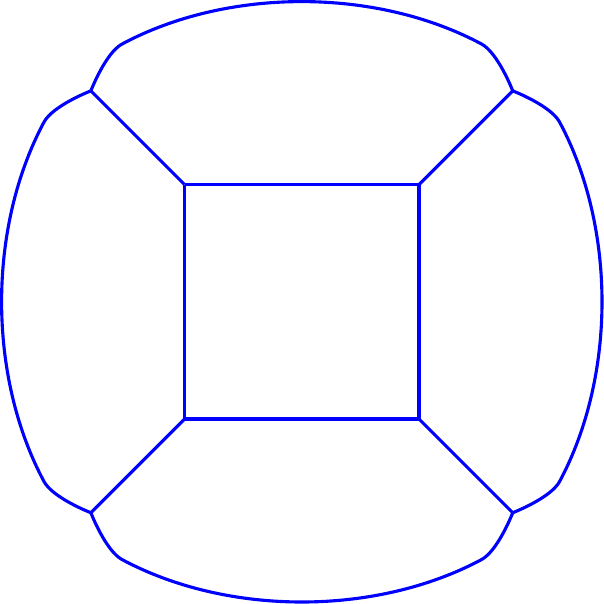}  \quad \quad \quad 
\labellist
\small
\pinlabel $0$ [t] at 32 46 
\pinlabel $0$ [t]  at 32 22
\pinlabel $1$ [b] at 32 92
\pinlabel $1$ [b]  at 32 70
\pinlabel $\Lambda_G$   at 90 20
\endlabellist
\raisebox{.5cm}{\includegraphics[scale=.7]{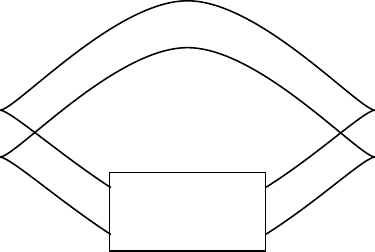}}
}

\caption{(left) A trivalent graph $G \subset S^2$. (middle) The base projection of the Legendrian $2$-weave $\widetilde{\Lambda}_G \subset J^1S^2$.  (right) Schematic of the front projection for the satellited $2$-weave $\Lambda_G \subset J^1\R^2$.
}
\label{fig:2Weave}
\end{figure}

\subsection{The face coloring polynomial}
For $G \subset S^2$ as above, a {\bf proper face coloring} of $G$ by a set $X$ is an assignment of an element of $X$ to each face of $G$ in such a way that along each edge of $G$ the two adjacent faces are assigned different values.  (We include simple closed curve components of $G$ as edges in this definition.)  
We define a {\bf face coloring polynomial}, $P(G^*,x)$, whose value at a positive integer $x$ is the number of proper face colorings of $G$ by elements of a set with $x$ elements.

\begin{remark}  The notation $P(G^*,x)$ reflects that we could have alternatively defined the face coloring polynomial to be the chromatic polynomial of the dual graph, $G^*$, to $G$:  vertices of $G^*$ are faces of $G$, and edges of $G^*$ connect faces that share common bordering edges.  It is well known that $P(G^*,x)$ is indeed a polynomial in $x$.
\end{remark}

\begin{proposition}  The face coloring polynomial, viewed as an assignment $G \mapsto P(G^*,x) \in \Z[x]$, depends only on the ambient isotopy type of $G$ within $S^2$ and satisfies the relations:
\begin{itemize}
	\item[(P1)]
	\[
	\raisebox{-1cm}{\includegraphics[scale=.5]{images/SR4b}} \quad -  
	\quad \raisebox{-1cm}{\includegraphics[scale=.5]{images/SR3b}}  \quad
	=  \quad
	\raisebox{-1cm}{\includegraphics[scale=.5]{images/SR2b}} \quad - 
	\quad  \raisebox{-1cm}{\includegraphics[scale=.5]{images/SR1b}} 
	\] 
	\item[(P2)]
	\[
	\raisebox{-1cm}{\includegraphics[scale=.5]{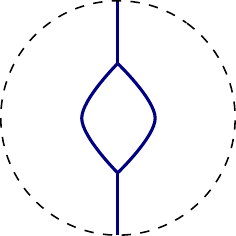}}  \quad
	=  \quad (x-2) \,\,
	\raisebox{-1cm}{\includegraphics[scale=.5]{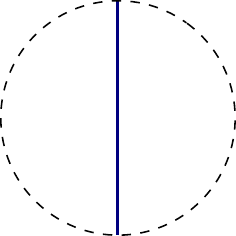}}
	\]
	\item[(P3)]
	\[
	\raisebox{-1cm}{\includegraphics[scale=.5]{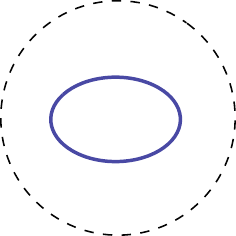}}  \quad
	=  \quad (x-1) \,\,
	\raisebox{-1cm}{\includegraphics[scale=.5]{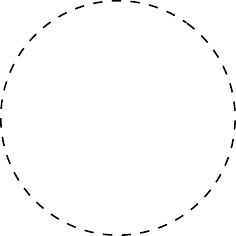}}
	\]
	
\end{itemize}	
\end{proposition}

Note that the bigon from (P2) as well as the disk (bounded by the closed curve) in (P3) are assumed to be innermost, i.e., no additional components of the graph $G$ appear in their interiors.

\begin{proof}  {\bf Proof of (P1):}  Let $G_1, G_2, G_3, G_4$ denote the four graphs appearing in the relation, and label their pictured regions as
\[
\labellist
\small
\pinlabel $A$  at 40 70
\pinlabel $B$  at -2 40
\pinlabel $C$  at 40 10
\pinlabel $D$  at 82 40
\endlabellist 
\raisebox{-1cm}{\includegraphics[scale=.4]{images/SR4bNoBorder}}  \quad \quad \quad 
\labellist
\small
\pinlabel $A$  at 40 82
\pinlabel $B$  at 10 40
\pinlabel $C$  at 40 -2
\pinlabel $D$  at 70 40
\endlabellist 
\raisebox{-1cm}{\includegraphics[scale=.4]{images/SR3bNoBorder}}  \quad \quad \quad 
\labellist
\small
\pinlabel $A$  at 40 82
\pinlabel $E$  at 40 40
\pinlabel $C$  at 40 -2
\endlabellist 
\raisebox{-1cm}{\includegraphics[scale=.4]{images/SR2bNoBorder}}  \quad \quad \quad 
\labellist
\small
\pinlabel $B$  at -2 40 
\pinlabel $F$  at 40 40
\pinlabel $D$  at 82 40 
\endlabellist 
\raisebox{-1cm}{\includegraphics[scale=.4]{images/SR1bNoBorder}}   
\]

\quad

\noindent It is possible and allowable that some of these labeled regions actually belong to the same face.
Write
\[
\begin{array}{lcl}
P(G_1^*,x) = P_1^{B \neq D} + P_1^{B=D}, & \quad & P(G_2^*,x) = P_2^{A \neq C} + P_2^{A=C}, \\
P(G_3^*,x) = P_3^{A \neq C} + P_3^{A =C}, & \quad & P(G_4^*,x) = P_4^{B \neq D} + P_4^{B=D}
\end{array}
\]
where $P_1^{B \neq D}$ (resp. $P_1^{B=D}$) counts only the face colorings of $G_1$ where the regions $B$ and $D$ have different (resp. the same) colors, and the other notations have similar meaning.  Bijections between colorings (defined to preserve the colors outside of the pictured disk) produce identities
\[
P_1^{B \neq D}= P_2^{A \neq C}, \quad P_3^{A \neq C} = P_1^{B=D}, \quad P_4^{B \neq D} = P_2^{A=C}, \quad P_3^{A=C} = P_4^{B=D},
\] 
which allow us to compute
\[
P(G_1^*,x) - P(G_2^*,x) = P_1^{B=D} - P_2^{A=C} = P_3^{A \neq C} - P_4^{B \neq D} = P(G_3^*,x)-P(G_4^*,x).
\]

\medskip

\noindent  {\bf Proof of (P2):}  Face colorings of the graph $G_1$ on the left are in bijection with ordered pairs $(F, C)$ consisting of a face coloring $F$ for the graph $G_2$ on the right together with a color $C$ different from the two colors of $F$ that appear in the pictured regions.  For any choice of $F$ there are $x-2$ choices for $C$, so $P(G_1^*,x) = (x-2) P(G_2^*,x)$.  

\medskip

\noindent  {\bf Proof of (P3):}  Face colorings of the graph on the left are in bijection with pairs $(F,C)$ consisting of a face coloring $F$ of the graph on the right and a color $C$ distinct from the color of the pictured region on the right.  For any $F$ there are $x-1$ choices for $C$.

\end{proof}

\subsection{Two graph invariants}
Given a trivalent graph $G \subset S^2$ and a prime power $q$, we set
\[
X_G(q) = \aug_{\Lambda_G}(q)
\]
and
\[
Y_G(q) = \frac{(q-1)^{6}}{q^{3}(q+1)} \left[ \frac{q^{1/2}}{(q-1)^2} \right]^{Z(G)} 
P(G^*, q+1)
\]
where $Z(G) = \chi(S^2 \setminus G)$ is the Euler characteristic of the complement of $G$. 
To simplify notation, we write $y(q) = \frac{(q-1)^{6}}{q^{3}(q+1)}$ so that 
\[
Y_G(q) = y(q) \left[ \frac{q^{1/2}}{(q-1)^2} \right]^{Z(G)} P(G^*, q+1).
\]

\begin{proposition} \label{prop:Gproperties}
Both of the assignments $G \mapsto X_G(q)$ and $G \mapsto Y_G(q)$  depend only on the ambient isotopy type of $G$ within $S^2$ and satisfy the relations:
\begin{enumerate}
\item[(G1)]  
\[
q^{1/2} \,\raisebox{-1cm}{\includegraphics[scale=.5]{images/SR2b}} \quad - 
\quad q^{1/2} \, \raisebox{-1cm}{\includegraphics[scale=.5]{images/SR1b}} = (q-1)^2 \left[
\raisebox{-1cm}{\includegraphics[scale=.5]{images/SR4b}} \quad -  
\quad \raisebox{-1cm}{\includegraphics[scale=.5]{images/SR3b}} \right]
\]

\item[(G2)]  
\[
\raisebox{-1cm}{\includegraphics[scale=.5]{images/G2A}}  \quad
=  \quad \frac{q^{1/2}}{q-1} \,\,
\raisebox{-1cm}{\includegraphics[scale=.5]{images/G2B}}
\]

\item[(G3)]  
\[
\raisebox{-1cm}{\includegraphics[scale=.5]{images/G3A}}  \quad
=  \quad q \,\,
\raisebox{-1cm}{\includegraphics[scale=.5]{images/G3B}}
\]
\end{enumerate}
\end{proposition}
Again, the bigon and disk that appear in (G2) and (G3) are required to be innermost.

\begin{proof}  {\bf Proof for $X_G(q)$:}  Modifying $G$ by an ambient isotopy of $S^2$ modifies $\widetilde{\Lambda}_G$ by a Legendrian isotopy in $J^1S^2$.  In turn this modifies $\Lambda_G \subset J^1\R^2$ (which is obtained via the Legendrian satellite construction applied to $\widetilde{\Lambda}_G$ and the unknot $U \subset J^1\R^2$) by a Legendrian isotopy within a Weinstein neighborhood of $U$.  Alternatively, one can use the front projection description from Figure \ref{fig:2Weave}, and check via a Reidemeister move argument that the Legendrian isotopy type of $\Lambda_G$ is preserved when an arc of $G$ is isotoped over the north pole.

The relation (G1) follows from the skein relation (SR2) where the Maslov potential values are $r=s=0$.

The relation (G2) follows from (SR3) together with the observation of \cite[Theorem 4.10, Figure 45]{CZ} that adding a bigon to a graph modifies the Legendrian weave by an ambient $0$-surgery. 

The relation (G3) follows from (SR1) where the terms on the RHS vanish since the crossing sheets of $\Lambda_G$ corresponding to the pictured circle have Maslov potentials $r=s=0$.

\medskip

\noindent {\bf Proof for $Y_G(q)$:}  The face coloring polynomial $P(G^*,x)$ and the Euler characteristic $Z(G) = \chi(S^2 \setminus G)$ are invariant under ambient isotopy of $G$, so $Y_G(q)$ is as well.

Labeling the graphs that appear in (G1) from left to right as $G_1, G_2, G_3, G_4$, we have $Z(G_1) = Z(G_2) = Z(G_3)-1= Z(G_4)-1$.  Using this and (P1) we compute
\begin{align*}
q^{1/2} Y_{G_1}(q) - q^{1/2} Y_{G_2}(q) & = q^{1/2} y(q) \left[\frac{q^{1/2}}{(q-1)^2}\right]^{Z(G_1)} \left[P(G_1^*,q+1)-P(G_2^*,q+1) \right] \\
& = q^{1/2} y(q) \left[\frac{q^{1/2}}{(q-1)^2}\right]^{Z(G_3)-1}\left[P(G_3^*,q+1)-P(G_4^*,q+1) \right] \\
& = (q-1)^2 \left[ Y_{G_3}(q) - Y_{G_4}(q) \right]  
\end{align*}
verifying (G1).

With $G_1$ and $G_2$ denoting the graphs on the left and right side of (G2), the identity is  established by using $Z(G_1) = Z(G_2)+1$ and (P2) as
\begin{align*}
Y_{G_1}(q) &= y(q)\left[\frac{q^{1/2}}{(q-1)^2}\right]^{Z(G_1)} P(G_1^*,q+1)  \\
 & =  	y(q)\left[\frac{q^{1/2}}{(q-1)^2}\right]^{Z(G_2)+1} (q-1)P(G_2^*,q+1) = \frac{q^{1/2}}{q-1} Y_{G_2}(q).
\end{align*}  

The relation (G3) is verified by a similar computation using (P3) and that the value of $Z(G)$ is the same for the two graphs that appear in (G3).

\end{proof}

\begin{proposition} \label{prop:lem2}
The assignments $G \mapsto X_G$ and $G \mapsto Y_G$ are uniquely determined by the properties of Proposition \ref{prop:Gproperties} and the values on the empty graph $X_\emptyset(q)$ and $Y_\emptyset(q)$.
\end{proposition}

\begin{proof}  We establish this for $G \mapsto X_G$ via a nested induction argument.  The same proof then applies also to $G \mapsto Y_G$.  

The outer induction is on  $N :=$ the sum of the number of vertices of $G$ and the number of closed curved components of $G$. The value of $X_G$ in the base case $N=0$ is equal to $X_{\emptyset}(q)$.  For the inductive step, for a given $N>0$, assuming the statement for $G$ with lesser values of $N$, we establish the following by induction on $m \geq 0$.

\medskip

\noindent {\bf Inner Inductive Statement:}  The value of $X_G$ is uniquely determined for any graph $G$ that (i) has a total of $N$ vertices and closed curve components and (ii) contains an innermost polygon $P$  with $m$ vertices.

\medskip

In the base case $m=0$ for the inner induction, the innermost polygon with $m=0$ vertices is a disk bounded by a closed curve component.  Thus, (G3) allows us to evaluate $X_G$ in terms of a graph with $N$ smaller by $1$ to which the outer inductive hypothesis (on $N$) applies.  

When $m=1$, the graph $G$ must appear near $P$ as $\raisebox{-.7cm}{\includegraphics[scale=.5]{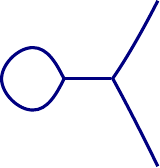}}$, and we can then apply (G1) followed by (G2) and (G3) to evaluate $X_G$ via
\begin{align*}
	\raisebox{-.7cm}{\includegraphics[scale=.5]{images/Mono1}} &=
\raisebox{-.7cm}{\includegraphics[scale=.5]{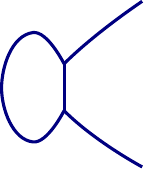}} + \frac{q^{1/2}}{(q-1)^2} \left[
\raisebox{-.7cm}{\includegraphics[scale=.5]{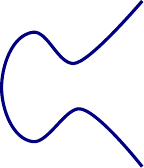}} -
\raisebox{-.7cm}{\includegraphics[scale=.5]{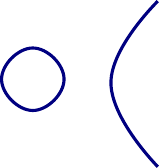}} \right] \\
& = \frac{q^{1/2}}{q-1} \raisebox{-.7cm}{\includegraphics[scale=.5]{images/Mono3}}  + \frac{q^{1/2}}{(q-1)^2} \left[
\raisebox{-.7cm}{\includegraphics[scale=.5]{images/Mono3}} -
q\raisebox{-.7cm}{\includegraphics[scale=.5]{images/Mono3}} \right].
\end{align*}
The value of $X$ on $\raisebox{-.7cm}{\includegraphics[scale=.5]{images/Mono3}}$ is determined by the inductive hypothesis on $N$ as the number of vertices has decreased by $2$ and the number of closed curve components has increased by at most $1$.

When $m=2$, we can use (G2) immediately, and the inductive hypothesis on $N$ then applies.

For $m>2$, we apply (G1) to one of the edges of $P$ to evaluate $X_G$ as
\begin{align*}
	\raisebox{-.7cm}{\includegraphics[scale=.5]{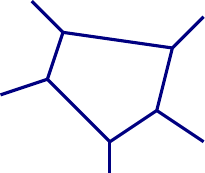}} &=
	\raisebox{-.7cm}{\includegraphics[scale=.5]{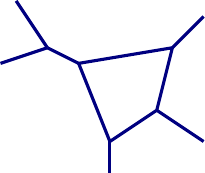}} + \frac{q^{1/2}}{(q-1)^2} \left[
	\raisebox{-.7cm}{\includegraphics[scale=.5]{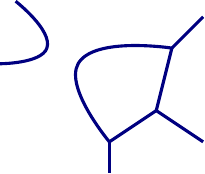}} -
	\raisebox{-.7cm}{\includegraphics[scale=.5]{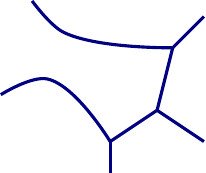}} \right].
\end{align*}
The inductive hypothesis on $m$ applies to the first term on the RHS and the inductive hypothesis on $N$ applies to the remaining terms.

\end{proof}

We are now prepared to prove Theorem \ref{thm:2weaves}.

\begin{proof}[Proof of Theorem \ref{thm:2weaves}]  In view of Propositions \ref{prop:Gproperties} and \ref{prop:lem2}, we need only show that $X_\emptyset(q) = Y_\emptyset(q)$.  To this end, compute
	\begin{align*}
	Y_\emptyset(q) & = y(q) \left[\frac{q^{1/2}}{(q-1)^2} \right]^{Z(\emptyset)} P(\emptyset^*, q+1) \\
	 & = \frac{(q-1)^6}{q^3(q+1)}\left[\frac{q^{1/2}}{(q-1)^2} \right]^{2} (q+1) = \frac{(q-1)^2}{q^2}.
	\end{align*}

For comparison, $X_\emptyset(q) = \aug_{\Lambda_\emptyset}(q)$ where $\Lambda_\emptyset = U \cup U'$ is union of a standard $2$-dimensional Legendrian unknot $U$ with second copy of $U$ that is denoted $U'$ and shifted a small amount in the positive $z$-direction from $U$.   
Using (SR1) with $r=0$ and $s=1$ we get
\[
\aug_{U \cup U'}(q) = q^{-1} \aug_{U \sqcup U}(q)
\]
where $U  \sqcup U$ denotes the $2$-component Legendrian unlink (where the front projections of the two copies of $U$ are completely disjoint).  
Then, applying (\ref{eq:productrule}) and (\ref{eq:Uvalue}) we have
\[
X_\emptyset(q) = q^{-1} \aug_{U \sqcup U}(q) = q^{-1} \aug_{U}(q) \cdot \aug_{U}(q) = q^{-1} (q^{1/2}-q^{-1/2})^2 = \frac{(q-1)^2}{q^2}
\]
which coincides with the value of $Y_\emptyset(q)$.
\end{proof}

\begin{example}  The proof of Proposition \ref{prop:lem2} provides a method to evaluate $X_G(q)=\mathit{Aug}_{\Lambda_{G}}(q)$ via repeated applications of the relations from Proposition \ref{prop:Gproperties} (which in turn derive from the skein relations of Theorem \ref{thm:main}).  Carrying out the process by hand for some graphs with a small number of vertices, and starting from the base case value of $X_\emptyset(q) = \frac{(q-1)^2}{q^2}$ from the above proof we arrive at
\begin{align*}  X_{\fig{images/X2}} &= q X_\emptyset = \frac{(q-1)^2}{q}, \\  & \\
	X_{\fig{images/X3}} &= \frac{q^{1/2}}{q-1} X_{\fig{images/X2}} = q^{1/2}-q^{-1/2}, 	\\  & \\
	X_{\fig{images/X4}} &= q X_{\fig{images/X2}} = (q-1)^2, \\  & \\
	X_{\fig{images/X5}} &= X_{\fig{images/X3}} + \frac{q^{1/2}}{(q-1)^2} \left[ X_{\fig{images/X2}} - X_{\fig{images/X4}}\right] =0, \\  & \\
	X_{\fig{images/X6}} &= \frac{q^{1/2}}{q-1} X_{\fig{images/X3}} = 1, \\  & \\
	X_{\fig{images/X7}} &= X_{\fig{images/X6}} + \frac{q^{1/2}}{(q-1)^2} \left[ X_{\fig{images/X5}}- X_{\fig{images/X3}} \right]	= \frac{q-2}{q-1}.
		\end{align*}
We collect the values along with simplified front projections for each $\Lambda_G$ (cf.,  \cite{TZ,CZ} for the translation between the graphs $G$ and the frontal presentation of $\Lambda_G$) in the following table:

\medskip

\begin{center}
\begin{tabular}{|c|c|c|}
	\hline
	$G$ & $\Lambda_G$ & $\mathit{Aug}_{\Lambda_G}(q)$ \\
	\hline	
	$\raisebox{1.5cm}{\huge $\emptyset$}$ & $\includegraphics[scale=.6]{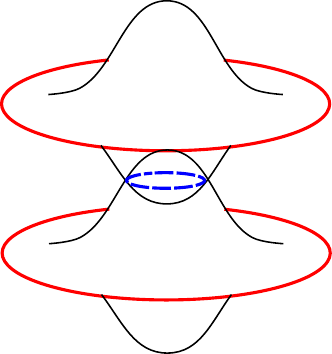}$ & \raisebox{1.5cm}{\huge $\displaystyle \frac{(q-1)^2}{q^2}$}  \\
	\hline	
	$\raisebox{1.5cm}{$\includegraphics{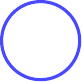}$}$ & $\includegraphics[scale=.6]{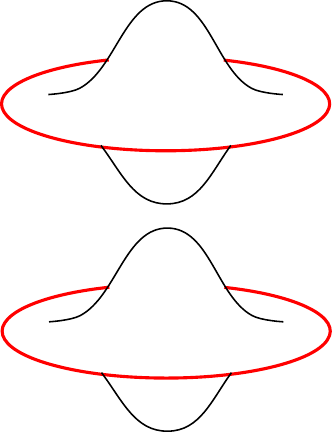}$ &  \raisebox{2cm}{\huge $\displaystyle  \frac{(q-1)^2}{q}$}  \\
	\hline	
	$\raisebox{0cm}{$\includegraphics{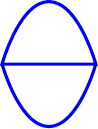}$}$ & $\includegraphics[scale=.6]{images/Y3}$ &  \raisebox{.8cm}{\huge $\displaystyle  \frac{q-1}{q^{1/2}}$}  \\
	\hline	
	$\raisebox{1cm}{$\includegraphics{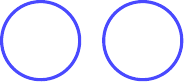}$}$ & $\includegraphics[scale=.6]{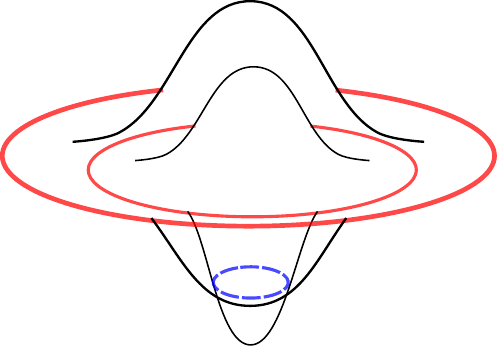}$ & \raisebox{1.5cm}{\huge $(q-1)^2$}  \\
	\hline	
	$\raisebox{.5cm}{\includegraphics{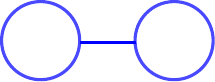}}$ & $\includegraphics{images/Y5}$ &  \raisebox{1cm}{\huge $0$}  \\
	\hline	
	$\includegraphics{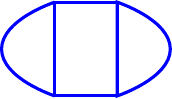}$ & $\includegraphics{images/Y6}$ &  \raisebox{.5cm}{\huge $1$}  \\
	\hline	
	$\includegraphics{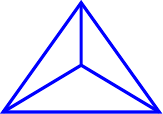}$ & $\includegraphics{images/Y7}$ &  \raisebox{.5cm}{\huge $\displaystyle \frac{q-2}{q-1}$}  \\
	\hline	
	\end{tabular}
\end{center}
\noindent In all pictured front projections for the $\Lambda_G$, the Maslov potential takes only the values $0$ and $1$.

\end{example}

Of course, one can also combine Theorem \ref{thm:2weaves} with existing computations of chromatic polynomials of planar graphs, cf. \cite{Read}, to obtain calculations of $\mathit{Aug}_{\Lambda_G}(q)$.

\section{Appendix: Cellular DGA for Legendrians with cone points}  \label{sec:Cone}

This section establishes an extended version of the cellular DGA that allows for Legendrians with cone point singularities in the front projection; see Theorem \ref{thm:ConeDGA}.  The case of $\Z/2$-coefficients was earlier considered in \cite{RuSu1}.  We then apply this extension to provide the proof of Proposition \ref{prop:ALambda3} that was used earlier
 to establish the skein relation (SR1).  As a preliminary, we review the formulation of the DGA in the presence of swallowtail points that was earlier omitted in the exposition of Section \ref{sec:CWDGA1}.

\subsection{Cellular DGA near swallowtail points}  \label{sec:swDGA} Let $(\Lambda, \Gamma, \mathcal{E}, \xi)$ be a Legendrian surface in $J^1S$ equipped with a collection of geometric cocycles,
$\Gamma = \{\Gamma_i\}$, a compatible polygonal decomposition of $\pi_x(\Lambda)$, $\mathcal{E}$, and a combinatorial spin structure, $\xi$.  Near a swallowtail point, $w$, the base projections of the cusp edges and $w$ separate a small disk $U \subset S$ centered at $\pi_x(w)$ into two regions; we call the region between the two cusp edges where the sheets that border the cusps exist the {\bf swallowtail region}.    Recall from Section \ref{sec:CWDGA1} that the compatible polygonal decomposition $\mathcal{E}$ includes the following additional information:

\begin{itemize}
	\item  At each swallowtail point, $w$, a choice of labeling by $S$ and $T$ is specified for the two $2$-cell corners that border the crossing arc at $w$.
	\end{itemize}
See Figure \ref{fig:STDecorate}.

 \begin{figure}
 	
 	\labellist
 	\small
 	\pinlabel $S$ [b] at 42 80
 	\pinlabel $T$ [t]  at 42 72
 	\pinlabel $w$ [r]  at -2 74
 	 	\endlabellist
 	\centerline{\includegraphics[scale=.7]{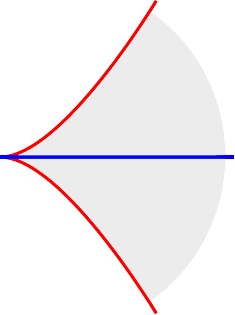} \quad \quad \quad 
 		\labellist
 		\small
 		\pinlabel $T$ [b] at 42 80
 		\pinlabel $S$ [t]  at 42 72
 		\pinlabel $w$ [r]  at -2 74
 		\endlabellist
 		\includegraphics[scale=.7]{images/STDecorate}
 		}
 	
 	\caption{The base projection near a swallowtail point $w$ with the two possible labelings by $S$ and $T$.  The swallowtail region is shaded.
 	}
 	\label{fig:STDecorate}
 \end{figure}

At the $0$-cell, $e^0_\alpha$, that is the projection of $w$, the sheets of $\Lambda$ above $e^0_{\alpha}$, $S_1^\alpha, \ldots, S^\alpha_n$, are totally ordered by descending $z$-coordinate.  Note that the swallowtail point $w$ is itself considered a sheet at $e^0_\alpha$.  This is in contrast to cusp edge points which are not considered sheets; see Definition \ref{def:Sheets} above.  The Maslov potential is extended to the swallowtail point $w$ to agree with the value on the subset of $\Lambda$ that borders $w$ from outside the swallowtail region.  Then, the generators of $\alg^{CW}(\Lambda)$ associated to $e^0_\alpha$ are $a^\alpha_{i,j}$, for all $1 \leq i < j \leq n$, with the usual grading from (\ref{eq:CWgrading}) and differential from (\ref{eq:dA}).   

In addition to the $n\times n$ upper triangular matrix $A= (a_{i,j}^\alpha)$ we need to introduce several  $(n+2)\times(n+2)$ matrices.  We make use of the notations:
\begin{itemize}
\item	$E_{i,j}$ denotes a matrix with $1$ in position $(i,j)$ and all other entries $0$.
\item $\Xi_m$ is  a diagonal matrix with $-1$ in position $(m,m)$ and other diagonal entries $1$.  We extend this notation by $\Xi_{m_1,m_2} :=  \Xi_{m_1}\Xi_{m_2}$ to allow for $-1$'s in more than one diagonal position. 
\item Given an $n\times n$ matrix, $M$, we form an $(n+2)\times(n+2)$ matrix, $\widetilde{M}_{m_1,m_2}$, by inserting two new rows and two new columns in positions $m_1$ and $m_2$ with all entries $0$.
\item $\widehat{M}_{m_1,m_2}$ is $\widetilde{M}_{m_1,m_2}$
with the $0$ entry in position $(m_1,m_2)$ changed to a $1$, i.e., $\widehat{M}_{m_1,m_2}=\widetilde{M}_{m_1,m_2}+E_{m_1,m_2}$.
\item $Q_{m_1, m_2}$ denotes the permutation matrix of a transposition $(m_1 \,m_2)$.  

\end{itemize}

For  an {\it upward} swallowtail point involving sheets in positions $k,k+1,k+2$ within the swallowtail region, we set
\begin{align*}
S &=  \Xi_{k+2} + E_{k+1,k+2} - \sum_{i<k} a^\alpha_{i,k}E_{i,k},  \\
T &=  I+E_{k+1,k+2}, \\
A_T& =  (I-E_{k+1,k+2}) \widehat{A}_{k,k+1} (I + E_{k+1,k+2}), \\
A_S& = Q_{k+1,k+2}A_T Q_{k+1,k+2}.
\end{align*}
For  a {\it downward} swallowtail point involving sheets in positions $l-2,l-1,l$ within the swallowtail region, we set
\begin{align}
	S &=  \Xi_{l-2} + E_{l-2,l-1} + \sum_{l-2<j} a^\alpha_{l-2,j}E_{l,j+2}, \notag  \\
	T &=  I-E_{l-2,l-1}, \notag \\
	A_T &=  (I+E_{l-2,l-1}) \widehat{A}_{l-1,l} (I - E_{l-2,l-1}) ,  \label{eq:ATequation} \\
A_S &= Q_{l-2,l-1}A_T Q_{l-2,l-1}.	\notag
	\end{align}

For a 1-cell, $e^1_\beta$, the boundary matrix, $A_\pm$, associated to a vertex at the swallowtail point is:
\begin{itemize}
	\item $A_T$ assuming (i) $e^1_\beta$ contains the crossing locus adjacent to the swallowtail point and (ii) the choice of ordering of sheets above $e^1_\beta$ matches the ordering above the $T$ side of the crossing locus, cf. Remark \ref{rem:sheetnum}.  If instead, the sheet ordering agrees with the $S$ side, then use $A_S$ 
	\item $\widehat{A}_{k,k+1}$ (resp. $\widehat{A}_{l-1,l}$) in the upward (resp. downward) case when $e^1_\beta$ is any other $1$-cell that lies in the swallowtail region near $w$.
	\item $A$ when $e^1_\beta$ is outside of the swallowtail region near $w$.  This includes the $1$-cells that contain the projection of the cusp arcs adjacent to $w$.
\end{itemize}

For a 2-cell, $e^2_\gamma$, adjacent to the swallowtail point we make the following adjustments:
\begin{itemize}
	\item Let the characteristic map for $e^2_\gamma$ have the form $P_\gamma \rightarrow \overline{e^2_\gamma}$ where $P_\gamma$ is a polygon whose edges parametrize the edges of $\overline{e^2_\gamma}$.  When $e^2_\gamma$ contains an $S$ or $T$ corner at the swallowtail point, we expand the vertex of the corner to a new edge that is labeled $S$ or $T$ and oriented away from the edge of $P_\gamma$ that maps to the crossing arc near $w$.  This produces an {\bf extended polygon} $\widetilde{P}_\gamma$ that we choose an initial and terminal vertex $v_0$ and $v_1$ from.  The differential for the $C$ matrix of $e^2_\gamma$ is then characterized by the formula from (\ref{eq:dC}) with additional terms of $S^{\pm1}$ or $T^{\pm1}$ inserted in to the products of $(\Delta_i+B_i)^{\eta_i}$ at locations where the paths from $v_0$ to $v_1$ pass the new $S$ or $T$ edge.  (The exponent is $+1$ or $-1$ depending whether the orientation of the edge matches the orientation from $v_0$ to $v_1$ or not.)  If $v_0$ is chosen to be the initial endpoint of an $S$ (resp. $T$) edge of $\widetilde{P}_\gamma$, then we use $A_{v_0} = A_S$ (resp. $A_{v_0} = A_T$).  If $v_0$ is the terminal endpoint of either an $S$ or $T$ edge, then we use either $\widehat{A}_{k,k+1}$ or $\widehat{A}_{l-1,l}$ depending on whether $w$ is upward or downward.  The $A_{v_1}$ matrices meet these same specifications.
	\item  When the swallowtail point appears as a vertex of a $2$-cell at a $v_0$ or $v_1$ corner that is inside the swallowtail region but not labelled with $S$ or $T$, then  either $\widehat{A}_{k,k+1}$ or $\widehat{A}_{l-1,l}$ is used depending on whether $w$ is upward or downward.  If the corner is outside the swallowtail region, then we simply use $A$ for $A_{v_0}$ or $A_{v_1}$.	 
\end{itemize}

\subsection{Cellular DGA near cone points}  We expand a bit on the description of the cone point singularity and its desingularization from Section \ref{sec:desingcone} above.  A cone point singularity  $c$ in the front projection of a Legendrian surface $\Lambda_{\mathit{cone}} \subset J^1S$ is a point  $c\in \pi_{xz}(\Lambda_{\mathit{cone}}) \subset S\times \R$ having a local neighborhood  $W \subset S\times \R$
 such that $(W, W \cap \pi_{xz}(\Lambda_{\mathit{cone}}))$ is diffeomorphic to $(V, V \cap C)$ where $V \subset \R^3$ is a neighborhood of $(0,0,0)$ and $C = \{(x_1,x_2,z) ,|\, x_1^2+x_2^2 = z^2\}$ is the standard cone.  See, e.g., the 3rd term of the (SR1) skein relation.  The pre-image of $c$ in $\Lambda_{\mathit{cone}}$ is an embedded $S^1$.  After perturbing $\Lambda_{\mathit{cone}}$ by a small Legendrian isotopy to  $\Lambda_{\mathit{res}}$ (here $\mathit{res}$ stands for ``resolved'') a cone point will resolve into a configuration of four swallowtail points, two downward and two upward, connected by crossing arcs and cusp edges as pictured via the base projection in Figure \ref{fig:Cone} (left); see \cite[Section 3]{Rizell} and \cite[Section 3]{EENS}.  In the base projection, the cusp edges bound  a closed square region $X$ with vertices at the swallowtail points, and the two crossing arcs appear as the diagonals of the square.  In the front projection, suppose $\Lambda_{\mathit{cone}}$ has $n$ sheets near $c$ with sheets $k$ and $k+1$ (numbered from top to bottom in the $z$-direction) meeting at $c$.  Then, $\Lambda_{\mathit{res}}$ will have $n+2$ sheets within the interior of $X$, and along the cusp edges the two sheets in positions $k+1$ and $k+2$ meet.   One of the crossing arcs is an intersection of sheets in positions $k$ and $k+1$ and has its end points at the two downward swallowtail points; and the other is an intersection of sheets in position $k+2$ and $k+3$ having the two upward swallowtail points as endpoints.  See Figure \ref{fig:Cone} where the crossing arc between sheets $k$ and $k+1$ is drawn vertically and labeled with $(k \, k+1)$;  the horizontal crossing arc labeled $(k+2 \, k+3)$ involves sheets $k+2$ and $k+3$.  See also Section \ref{sec:desingcone} above for the $1$-dimensional slices of $\pi_{xz}(\Lambda_{\mathit{res}})$.

\begin{figure}
	
\labellist
\small
\pinlabel $(k\,k+1)$ [l] at 216 198
\pinlabel $(k+2\,k+3)$ [l]  at 216 70
\endlabellist
\centerline{\includegraphics[scale=.7]{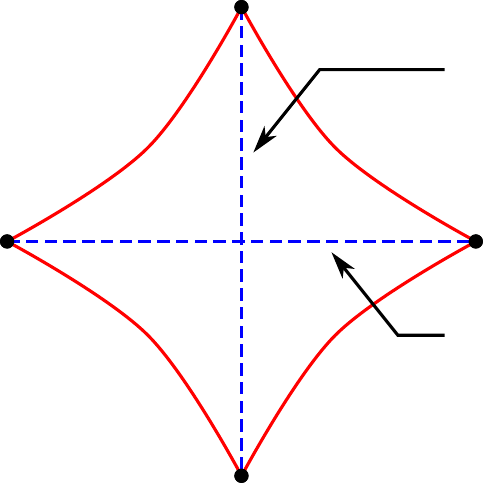} \quad \quad \quad 
\labellist
\small
\pinlabel $U$ [l] at 94 100
\endlabellist
\raisebox{2cm}{$\begin{array}{ccc}
\raisebox{1cm}{\includegraphics[scale=.5]{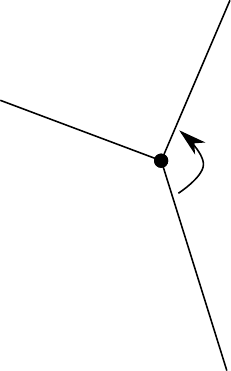}}  & \quad \quad \quad & 
\labellist
\small
\endlabellist
\includegraphics[scale=.5]{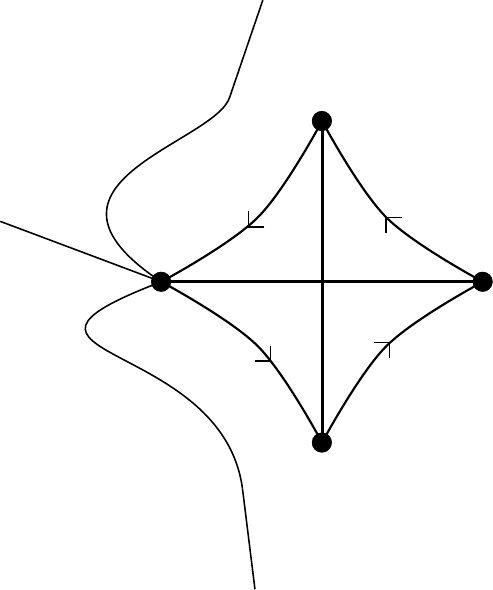}  \\
\mathcal{E}_{\mathit{cone}} &  & \mathcal{E}_{\mathit{res}}
\end{array}$}
}

\caption{(left) The base projection of a desingularized cone point.  (middle and right)  The compatible decompositions $\mathcal{E}_{\mathit{cone}}$ and $\mathcal{E}_{\mathit{res}}$ near the base projection of the cone point and its resolution.  
}
\label{fig:Cone}
\end{figure}

We describe know a version of the cellular DGA that can be applied directly to $\Lambda_{\mathit{cone}}$.
It requires the following choices:

\begin{enumerate}
	\item[(1)]  A compatible polygonal decomposition $\mathcal{E}_{\mathit{cone}}$ of $\pi_{x}(\Lambda_{\mathit{cone}})$ such that the projection of the cone point, $\pi_x(c)$ is contained in the $0$-skeleton.  At $\pi_x(c)$ a corner of a neighboring $2$-cell $e^2_\gamma$ adjacent to $\pi_x(c)$ is selected and labeled with a $U$.  In the extended polygon $\widetilde{P}_\gamma$ for $e^2_\gamma$ (as in Section \ref{sec:swDGA}) we expand the vertex with the $U$ label to an additional edge (that we continue to label with $U$) and choose an orientation for this $U$ edge.
\end{enumerate}
We produce from $\mathcal{E}_{\mathit{cone}}$ a compatible polygonal decomposition for $\Lambda_{\mathit{res}}$,  denoted $\mathcal{E}_{\mathit{res}}$, as follows:  Choose local coordinates on $S$ so that the chosen orientation of the $U$ edge is counter-clockwise around the cone point.
  Expand the $U$ corner at the cone point, and place the square $X$ in the corner with one of its vertices that is an upward swallowtail point located at the previous position of the cone point.  See Figure \ref{fig:Cone} (middle and right) where we have  oriented the $1$-cells appearing along $\partial X$ consistently with the orientation of $U$.

\begin{enumerate}
	\item[(2)] We choose a combinatorial spin structure $\xi$ for $\Lambda_{\mathit{res}}$ that appears near the resolved cone point as pictured in Figure \ref{fig:ConeSpin} (left).
	\item [(3)] We choose a collection of geometric cocycles  for $\Lambda_{\mathit{res}}$ 
	that either has the form $\Gamma = \{\Gamma_1, \ldots, \Gamma_\ell\}$ or $\Gamma =\{\Gamma_1, \ldots, \Gamma_\ell, \lambda\}$  where 
	\begin{itemize}
		\item[(i)] the $\Gamma_i$ curves for $1 \leq i \leq \ell$ all have their base projections disjoint from $X$, and
		\item[(ii)] if it exists, $\lambda$ passes through the $X$ region near the resolved cone point as pictured in Figure \ref{fig:ConeSpin} (right).
	\end{itemize}
	\end{enumerate}

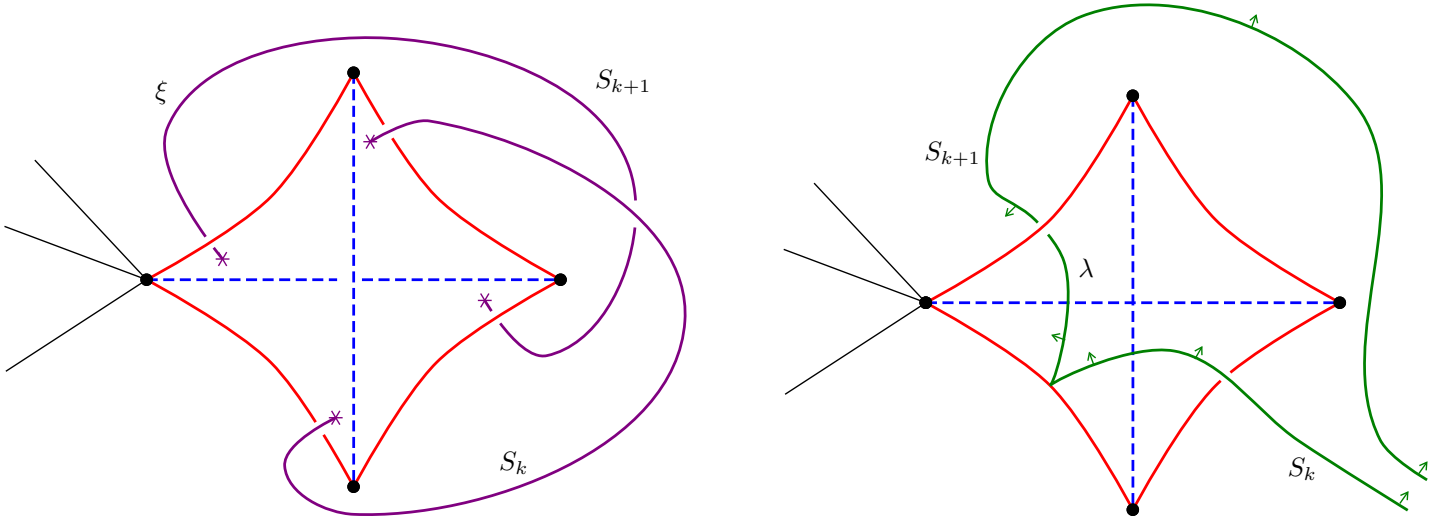
\begin{figure}

\centerline{	
\labellist
\small
\pinlabel $S_{k+1}$ [bl] at 318 226
\pinlabel $S_k$ [br]  at 288 20
\pinlabel $\xi$ [br]  at 92 221
\endlabellist
\includegraphics[scale=.7]{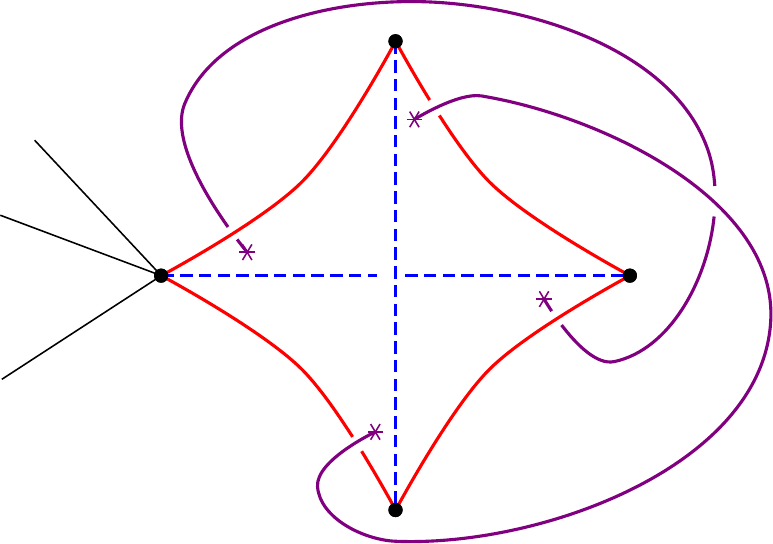}  \quad \quad \quad 
\labellist
\small
\pinlabel $S_{k+1}$ [b] at 92 186
\pinlabel $S_k$ [t]  at 282 36
\pinlabel $\lambda$ [l]  at 156 135
\endlabellist
\includegraphics[scale=.7]{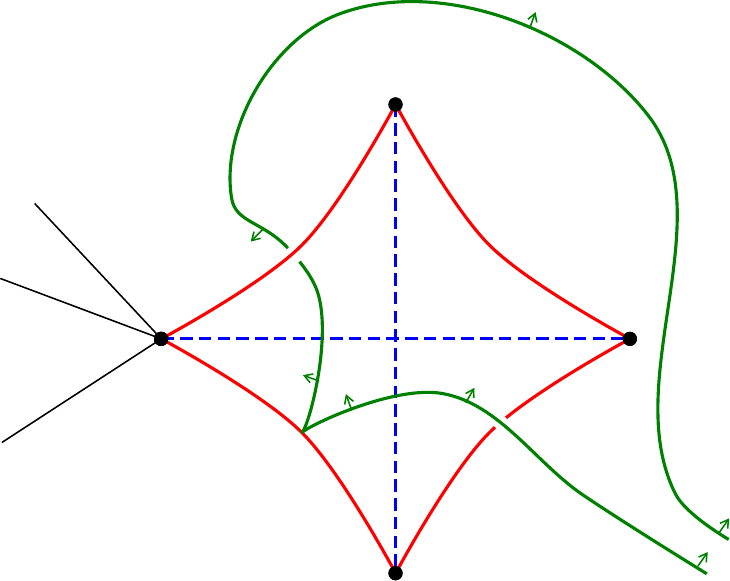}
}

\caption{(left) A combinatorial spin structure.  (right) A geometric cocycle  $\lambda$ passing through the cone point.  Traveling through $X$ beginning at the upper left, $\lambda$ passes sequentially from $S_{k+3}$ to $S_{k+2}$ to $S_{k+1}$ to $S_k$.
}
\label{fig:ConeSpin}
\end{figure}

\begin{remark}
The curve $\lambda$ on $\Lambda_{\mathit{res}}$ is isotopic to the one pictured on  $\Lambda_{\mathit{cone}}$ in Figure \ref{fig:ConeBase} (left).  For simplicity, we may picture this curve in the base projection of $\Lambda_{\mathit{cone}}$ simply as in Figure \ref{fig:ConeBase} (right).  Note that the $180^\circ$ rotation that occurs when passing through the cone point is in the direction specified by the choice of orientation of $U$.  Reversing the choice of orientation of $U$ alters $\lambda$ by summing with a geometric cocycle that is the pre-image of the cone point, cf. the curve $\mu_3$ from Figure \ref{fig:CharMap2}. 
\end{remark}

\begin{figure}
	
	\centerline{	
		\labellist
		\small
		\pinlabel $\lambda$ [l] at 52 22
		\endlabellist
		\raisebox{0cm}{\includegraphics[scale=.8]{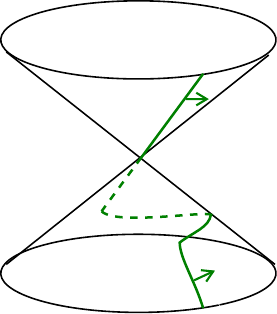}} \quad \quad \quad \quad \quad
		\labellist
		\small
		\pinlabel $\lambda$ [bl] at 150 68
		\pinlabel $U$ [l] at 100 104
		\endlabellist
				\includegraphics[scale=.7]{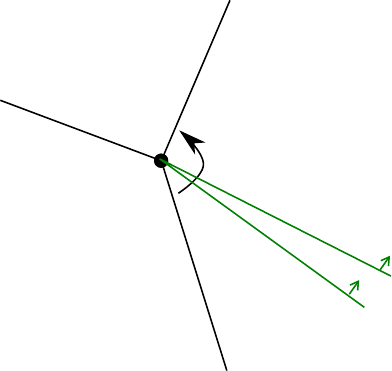} 
		}
	
	\caption{(left) The curve $\lambda$ pictured on $\Lambda_{\mathit{cone}}$.  (right) A schematic presentation of $\lambda$ in the base projection of $\Lambda_{\mathit{cone}}$.  The orientation of $U$ specifies the direction of the  $180^\circ$ rotation that $\lambda$ makes in the lower half of the cone before passing through the cone point.
	}
	\label{fig:ConeBase}
\end{figure}

We can then associate a cellular DGA $\mathcal{A}^{CW}(\Lambda_{\mathit{cone}})$ to $(\Lambda_{\mathit{cone}}, \mathcal{E}_{\mathit{cone}}, \xi, \Gamma)$ as in Section \ref{sec:CWDGA1} above, but with the following modifications made at the cone point.  

\begin{itemize} 
	\item We denote by $\Z[t_1^{\pm1}, \ldots, t_{\ell}^{\pm1}, \lambda^{\pm1}]$ the coefficient ring, where 
\begin{itemize}
	\item[(i)] if $\Gamma = \{\Gamma_1, \ldots, \Gamma_{\ell}, \lambda \}$, we simply use $\lambda$ for the  variable in the coefficient ring corresponding to the cocycle $\lambda$, and
	\item[(ii)] if $\Gamma = \{\Gamma_1, \ldots, \Gamma_{\ell}\}$ we set $\lambda =1$.
\end{itemize}

\item The cone point $0$-cell gives rise to generators $a^0_{i,j}$ for all $1 \leq i < j \leq n$ with $(i,j) \neq (k,k+1)$ where $n$ is the number of sheets of $\Lambda_{\mathit{cone}}$ near the cone point with sheets $k$ and $k+1$ adjacent to the cone point.  Their differentials are characterized by
\[
\partial A_0 = J \cdot (A_0 + (1-\lambda^{-1}) E_{k,k+1})^2.
\] 

\item  The matrix $A_0$ is replaced with $A_0+ (1-\lambda^{-1}) E_{k,k+1}$ when it appears in the formulas for the differential of adjacent $1$-cells or $2$-cells.  

\item  In the differential for the $2$-cell that contains the $U$ corner at the cone point, the matrix
\begin{equation} \label{eq:UmatrixConeDGA}
U= \Theta^\lambda_{k,k+1} + \sum_{i<k} (a_{i,k+1}^0\lambda) E_{i,k} + \sum_{k+1<j} (\lambda a^0_{k,j}) E_{k+1,j}
\end{equation}
is associated to the (oriented) edge $U$.  Here, the notation $\Theta^\lambda_{k,k+1}$ is as in the statement of Proposition \ref{prop:ALambda3}; it denotes a diagonal matrix with $\lambda$ in positions $k$ and $k+1$ and other diagonal entries equal to $1$.

\end{itemize}

\begin{theorem}  \label{thm:ConeDGA}
There is a stable tame isomorphism $\mathcal{A}^{CW}(\Lambda_{\mathit{cone}}) \cong \mathcal{A}^{CW}(\Lambda_{\mathit{res}})$.
\end{theorem}

\begin{remark}
Theorem \ref{thm:ConeDGA} and its proof extend in the natural way to surfaces with more than one cone point.  In another variant, one can allow for multiple curves $\lambda_1, \lambda_2, \ldots, \lambda_m$ to pass through the cone point appearing as parallel shifts of $\lambda$ from Figure \ref{fig:ConeSpin} when within $X$.  In this case, appearances of $\lambda$ in the above description of $\mathcal{A}(\Lambda_{\mathit{cone}})$ should be replaced with the product $\lambda_1\cdots \lambda_m$.  As indicated above, the case without any $\lambda$ curve is allowed as well and results from putting $\lambda=1$ in the above discussion.
\end{remark}

\begin{example}  We illustrate the use of $\mathcal{A}^{CW}(\Lambda_{\mathit{cone}})$ by giving an alternate computation of the DGA of the cone point torus, $T_{\mathit{cone}}$, from Section \ref{sec:firstex}.  The reader may compare the result with the original computation of the Legendrian DGA of $T_{\mathit{cone}}$ from \cite{Rizell1}.  Using the polygonal decomposition $\mathcal{E}_{\mathit{cone}}$ and geometric cocycles $\mu$ and $\lambda$ pictured as

\labellist
\small
\pinlabel $C_1$ at 40 134
\pinlabel $\lambda$ [b] at 90 132
\pinlabel $B_1$ [b] at 22 88
\pinlabel $B_2$ [l] at 180 82
\pinlabel $\mu$ [b] at 148 50
\pinlabel $A_1$ [l] at 104 92
\pinlabel $A_0$ [r] at -2 82
\endlabellist
\[
\includegraphics[scale=.7]{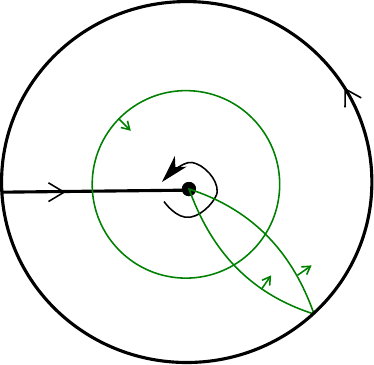} 
\]

\noindent the cellular DGA $\mathcal{A}^{CW}(\Lambda_{\mathit{cone}})$ has coefficient ring $\Z[\lambda^{\pm1}, \mu^{\pm1}]$ and generators:
\begin{center}
	\begin{tabular} {c|c}
		Degree & Generators \\
		\hline & \\
		$1$ & $b:=b^1_{1,2}$ \\
		$2$ & $c:= c^1_{1,2}$ 
		\end{tabular}
\end{center}
Using the Maslov potential with value $0$ (resp. $1$) on the lower (resp. upper) sheet of $T_{\mathit{cone}}$, the sign matrix appearing in (\ref{eq:dB}) and (\ref{eq:dC}) is
\[
J = \left[ \begin{array}{cc} -1 & 0 \\ 0 & 1 \end{array} \right],
\]
and we have the matrix formulas
\begin{align*}
J\cdot \partial B_1 &= (A_1 + (1-\mu^{-1})E_{1,2})( \Delta_1 + B_1) - (\Delta_1 + B_1) A_0  \\
&= \left[ \begin{array}{cc} 0 & 1-\mu^{-1} \\ 0 & 0 \end{array} \right] \left( \left[ \begin{array}{cc} \lambda & 0 \\ 0 & 1 \end{array} \right] + \left[ \begin{array}{cc} 0 & b \\ 0 & 0 \end{array} \right]\right) - \left( \left[ \begin{array}{cc} \lambda & 0 \\ 0 & 1 \end{array} \right] + \left[ \begin{array}{cc} 0 & b \\ 0 & 0 \end{array} \right]\right) \left[ \begin{array}{cc} 0 & 1 \\ 0 & 0 \end{array} \right]  \\
&= \left[ \begin{array}{cc} 0 & 1-\mu^{-1}-\lambda \\ 0 & 0 \end{array} \right]
\end{align*}
and
\begin{align*}
	J \cdot \partial C_1 & = (A_1 + (1-\mu^{-1})E_{1,2}) C_1 + C A_0 + (\Delta_1+B_1)(\Delta_2+B_2) - U (\Delta_1+B_1)  \\
	& = 0 + 0 + \left[ \begin{array}{cc} \lambda & b \\ 0 & 1 \end{array} \right] \left[ \begin{array}{cc} \mu & 0 \\ 0 & \mu \end{array} \right] - \left[ \begin{array}{cc} \mu & 0 \\ 0 & \mu \end{array} \right] \left[ \begin{array}{cc} \lambda & b \\ 0 & 1 \end{array} \right] \\
	& = \left[ \begin{array}{cc} 0 & b \mu - \mu b \\ 0 & 0 \end{array} \right].
\end{align*}
Thus, the differential is given by
\[
\partial b = \lambda + \mu^{-1} -1, \quad \partial c = \mu b - b \mu.
\]
\end{example}

Before proving Theorem \ref{thm:ConeDGA} we deduce  Proposition \ref{prop:ALambda3}  as a consequence. 

\begin{proof}[Proof of Proposition \ref{prop:ALambda3}]
The combinatorial spin structure, $\xi_3$, for $\Lambda_3$ is formed from $\xi_1$ by adding two arcs near the resolved cone point as in Figure \ref{fig:ConeSpin}.  The polygonal decomposition from Figure \ref{fig:CharMap2} can then be used as $\mathcal{E}_{\mathit{cone}}$ with the $U$ edge assigned the counterclockwise orientation around $A_1$ in the coordinates from Figure \ref{fig:CharMap2}.  In the Connected Case, the curve $\lambda_3$ from Figure \ref{fig:CharMap2} is taken to pass through the resolution of the cone point as in Figure \ref{fig:ConeSpin}.  With these choices, the DGA $\mathcal{A}^\mathit{CW}(\Lambda_{\mathit{cone}})$ for $\Lambda_3$ matches the description from Proposition \ref{prop:ALambda3}.
\end{proof} 

\begin{proof}[Proof of Theorem \ref{thm:ConeDGA}]  We orient and label cells within $X$ as in Figure \ref{fig:ConeLabels1} where the choice of $S$ and $T$ corners at swallowtail points is indicated by labeling of edges of the extended polygons for the $2$-cells of $X$.  To prove the theorem we will exhibit a series of applications of Proposition \ref{prop:cancel} that produces a stable tame isomorphic quotient of $\mathcal{A}^{CW}(\Lambda_{\mathit{res}})$ in which all generators associated to the cells of the (closed) square $X$ have been removed excepting those $a^0_{i,j}$ with $(i,j) \neq (k,k+1)$.  Moreover, in this quotient we will see that $a^0_{k,k+1}$ becomes equal to $1-\lambda^{-1}$ and  the product $(\Delta_7 + B_7)(\Delta_6+B_6)(\Delta_5+B_5)(\Delta_0+B_0)$ becomes equal to the matrix $U$ from (\ref{eq:UmatrixConeDGA}).  Thus, the quotient so constructed has both its generators and differentials in agreement with those of $\mathcal{A}^{CW}(\Lambda_{\mathit{cone}})$.

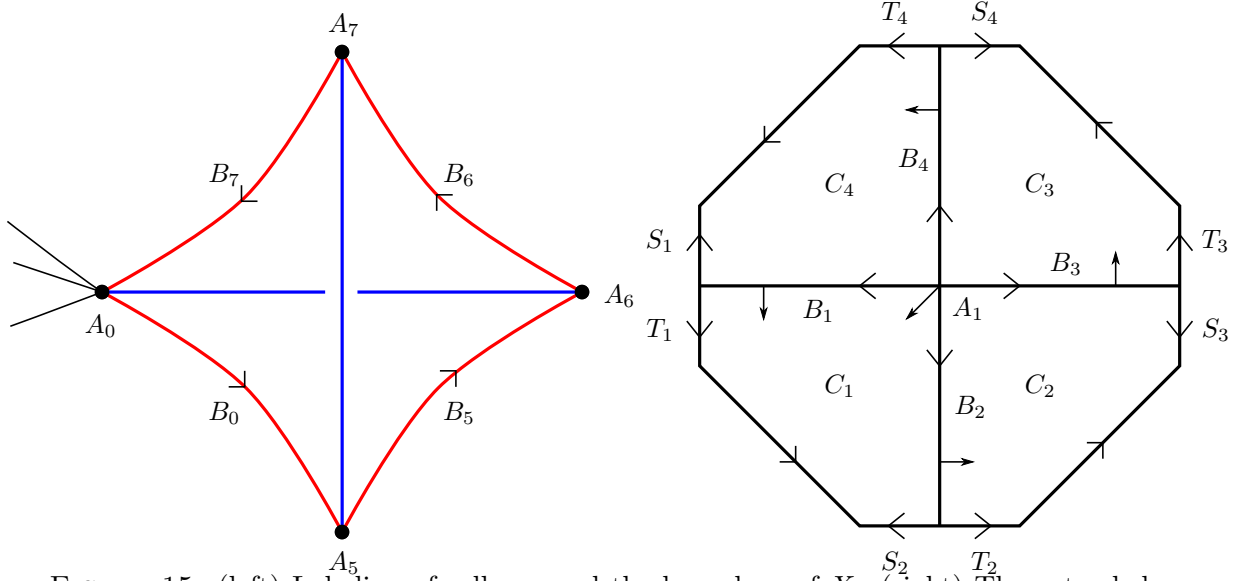
\begin{figure}

\centerline{	
\labellist
\small
\pinlabel $A_0$ [t] at 44 109
\pinlabel $A_5$ [t] at 158 -2
\pinlabel $A_6$ [l] at 276 114
\pinlabel $A_7$ [b] at 158 232
\pinlabel $B_0$ [t] at 102 68
\pinlabel $B_5$ [t] at 212 68
\pinlabel $B_6$ [b] at 212 162
\pinlabel $B_7$ [b] at 102 162
\endlabellist
\includegraphics[scale=.8]{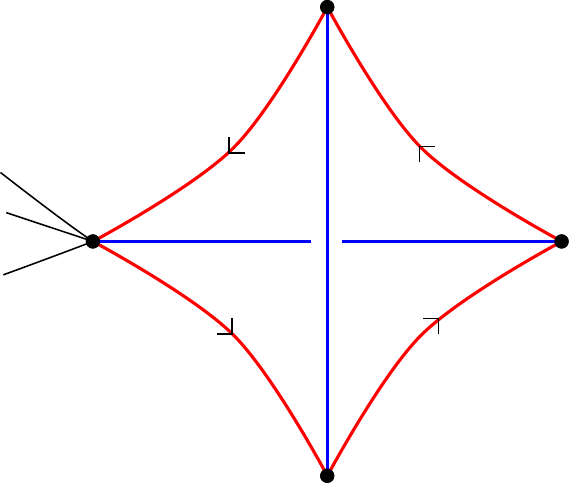}  \quad \quad \quad 
\labellist
\small
\pinlabel $T_4$ [b] at 98 238
\pinlabel $S_4$ [b] at 140 238
\pinlabel $T_1$ [r] at -2 98
\pinlabel $S_1$ [r] at -2 140
\pinlabel $T_3$ [l] at 238 140
\pinlabel $S_3$ [l] at 238 98
\pinlabel $S_2$ [t] at 98 -2
\pinlabel $T_2$ [t] at 140 -2
\pinlabel $C_1$  at 72 72
\pinlabel $C_2$  at 166 72
\pinlabel $C_3$  at 166 166
\pinlabel $C_4$  at 72 166
\pinlabel $B_1$ [t] at 62 116
\pinlabel $B_2$ [l] at 122 62
\pinlabel $B_3$ [b] at 178 120
\pinlabel $B_4$ [r] at 118 178
\pinlabel $A_1$ [tl] at 122 116
\endlabellist
\includegraphics[scale=.8]{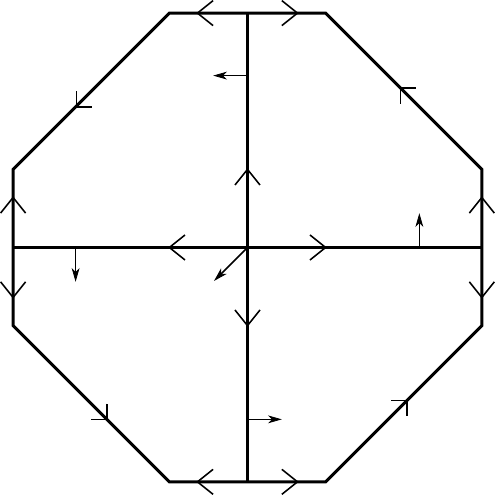}
}

\caption{(left) Labeling of cells around the boundary of $X$.  (right) The extended polygons for the $2$-cells interior to $X$.  The arrows pointing from $B_1, B_2, B_3, B_4, A_1$ to adjacent $2$-cells indicate the ordering of sheets used above these cells.
}
\label{fig:ConeLabels1}
\end{figure}

Recall that applications of Proposition \ref{prop:cancel} rely on an ordering of generators making  $\mathcal{A}^{CW}(\Lambda_{\mathit{res}})$ into a triangular dg-algebra.  Following Construction \ref{const:order}, such orderings arise from a choice of orderings of the cells of $\mathcal{E}_{\mathit{res}}$ in each dimension.  We use an ordering of cells (as usual we use the same notation for a cell and its corresponding matrix of generators) satisfying 
\begin{align}
 & A_5 \succ A_6 \succ A_7 \succ A_1 \succ A_0  \mbox{ for $0$-cells,   and} \notag \\
 & B_7 \succ B_6 \succ B_5 \succ B_0 \succ B_4 \succ B_3 \succ B_2 \succ B_1 \mbox{ for $1$-cells;} \label{eq:ConeOrdering}
\end{align}
 the choice of ordering of $2$-cells is irrelevant.  At a few points in the process we will adjust this ordering by reordering a generator $y$ with $\partial y=0$ to be the smallest element.  (This always results in another valid triangular ordering and allows terms involving $y$ to be absorbed into the $w$ term from the statement of Proposition \ref{prop:cancel}.)  When applying Proposition \ref{prop:cancel}, as above, we often indicate generators that are being cancelled (in the role of  $y_1,\ldots, y_r$ and $z_1, \ldots, z_r$ from Proposition \ref{prop:cancel}) by placing a $\framebox{\color{white}{0}}$ around them or the matrices that contain the canceled generators.  Repeated applications of Proposition \ref{prop:cancel} will produce a sequence of quotients  with generating sets of decreasing size.  We use the notation $\doteq$ to indicate equalities that hold (between equivalence classes) in the currently constructed quotient dg-algebra.

Throughout the proof we will make use of the matrix notations from Section \ref{sec:swDGA}.  In addition, we will use
\begin{itemize}
	\item $\Theta_m = \Theta^\lambda_m$ for a diagonal matrix with $\lambda$ in position $(m,m)$ and other diagonal entries $1$.  This is extended to $\Theta_{m_1,m_2} = \Theta_{m_1} \Theta_{m_2}$ for $m_1<m_2$.
\end{itemize}

\medskip

\noindent {\bf Step 1.}  Cancel the generators around $\partial X$ {\it except for} $A_0$ and $B_0$.

\medskip

We have
\[
J \cdot \partial B_7 = A_0( \Xi_{k+1}\Theta_{k+1} +B_7) - (\Xi_{k+1}\Theta_{k+1} +B_7)A_7
\]
where the left and right pictures of Figure 13 determine the appearance of $\Xi_{k+1}$ and $\Theta_{k+1},$ respectively.
Rewriting as
\[
J \cdot \partial \framebox{$B_7$} = -\Xi_{k+1}\Theta_{k+1} \framebox{$A_7$} + (A_0B_7-B_7A_7) + A_0 \Xi_{k+1}\Theta_{k+1}
\]
allows us to apply Proposition \ref{prop:cancel} to cancel the $b^7_{i,j}$ and $a^7_{i,j}$ generators in pairs.  [Here, comparing with the notations in the statement of Proposition \ref{prop:cancel}, the $b^7_{i,j}$ ordered with $|i-j|$ non-decreasing play the role of $y_1, \ldots, y_r$; the $a^7_{i,j}$ are the $z_1, \ldots, z_r$; the $v$ terms arise from $A_0B_7-B_7A_7$; the $w$ terms arise from $A_0 \Xi_{k+1}\Theta_{k+1}$; and the $\alpha$ coefficients are coming from $-J\Xi_{k+1}\Theta_{k+1}$.]  In the resulting quotient, we have
\begin{equation}
	\label{eq:ConeProof1}
B_7 \doteq 0, \quad \mbox{and} \quad \partial B_7 \doteq 0,
\end{equation}
and the latter identity leads to
\begin{equation}
	\label{eq:ConeProof2}
A_7 \doteq \Xi_{k+1}\Theta_{k+1}^{-1} A_0 \Theta_{k+1} \Xi_{k+1}.
\end{equation}
Similarly, we use
\[
J \cdot \partial \framebox{$B_6$} = A_7(\Xi_{k} + B_6) - (\Xi_{k} + B_6) \framebox{$A_6$}
\]
to cancel the $b^6_{i,j}$ with the $a^6_{i,j}$.  In the resulting quotient we have
\begin{equation}
	\label{eq:ConeProof3}
B_6 \doteq 0 \quad \mbox{and} \quad A_6 \doteq \Xi_kA_7\Xi_{k} \doteq \Xi_{k,k+1} \Theta^{-1}_{k+1}A_0\Theta_{k+1}\Xi_{k,k+1}. 
\end{equation}
Continuing, we apply Proposition \ref{prop:cancel} to cancel
\[
J \cdot \partial \framebox{$B_5$} = A_6(\Xi_{k+1} \Theta_{k} + B_5)- (\Xi_{k+1} \Theta_{k} + B_5)\framebox{$A_5$}
\]
resulting in
\begin{equation} \label{eq:ConeProof4}
B_5 \doteq 0 \quad \mbox{and} \quad A_5 \doteq \Xi_{k} \Theta^{-1}_{k,k+1} A_0 \Theta_{k,k+1} \Xi_k.	
\end{equation}

\medskip

\noindent {\bf Step 2.}  Cancel the following generators within the interior of $X$:
\begin{itemize}
	\item $c^4_{k,k+1}$  with $b^1_{k,k+1}$.
	\item All other $C_4$ generators  with $B_4$.
	\item $c^3_{k+2,k+3}$  with $a^0_{k,k+1}$.
	\item All other $C_3$ generators  with $B_3$.
	\item All $C_2$ generators except for $c^2_{k,k+1}$ with $B_2$.
	\item $B_1$ generators other than $b^1_{k,k+1}$ are cancelled with $A_1$.
\end{itemize} 
This will leave only $c^{2}_{k,k+1}$, $C_1$, $B_0$, and $A_0$ except for $a^0_{k,k+1}$ among generators from $X$.

\medskip
In preparation for the details, we catalog the $S$ and $T$ matrices, enumerated as in Figure \ref{fig:ConeLabels1},  as
\begin{align}
\label{eq:ConeProof5}	T_1 & = T_3 = I + E_{k+2,k+3}, \quad T_2=T_4 = I-E_{k,k+1},  \\
\label{eq:ConeProof6}	S_1 &= \Xi_{k+3} + E_{k+2,k+3} - \sum_{i<k+1} a^0_{i,k+1} E_{i,k+1}, \\
\label{eq:ConeProof7}	S_2 &= \Xi_k + E_{k,k+1} + \sum_{k<j} a^5_{k,j} E_{k+2,j+2} \\
	&  \notag \doteq \Xi_k + E_{k,k+1} + (-\lambda^{-1}a^0_{k,k+1}\lambda)E_{k+2,k+3} + \sum_{k+1<j} (-\lambda^{-1} a^0_{k,j}) E_{k+2,j+2}, \\
\label{eq:ConeProof8}	S_3 &=\Xi_{k+3} + E_{k+2,k+3} - \sum_{i<k+1} a^6_{i,k+1} E_{i,k+1} \\
	 & \notag \doteq \Xi_{k+3} + E_{k+2,k+3}- (a^0_{k,k+1} \lambda) E_{k,k+1} + \sum_{i<k}  (a^0_{i,k+1}\lambda)E_{i,k+1}, \\
\label{eq:ConeProof9}	 S_4&= \Xi_k + E_{k,k+1} + \sum_{k<j} a^7_{k,j} E_{k+2,j+2} \\
	  & \doteq \Xi_k + E_{k,k+1} + (-a^0_{k,k+1} \lambda) E_{k+2,k+3} + \sum_{k+1 < j} a^0_{k,j} E_{k+2,j+2}.
\end{align}

Computing $\partial C_4$ using the endpoints of $B_4$ for $v_0$ and $v_1$, we have
\begin{align}  \label{eq:ConeProof10}
J \cdot \partial C_4 &\doteq (A_7)_{T_4} C_4 + C_4 (Q_{k+2, k+3} A_1 Q_{k+2,k+3}) \\
 & \quad 
+ (I+B_4)-T_4^{-1}(\Theta_{k+3} \Xi_{k+3})^{-1}S_1(Q_{k+2,k+3}(\Theta_{k+2} +B_1)Q_{k+2,k+3}) \notag
\end{align}
where we made use of $(\Theta_{k+3} \Xi_{k+3} + B_7)^{-1} \doteq (\Theta_{k+3} \Xi_{k+3})^{-1}$ from (\ref{eq:ConeProof1}).  [The matrix $(A_7)_{T_4}$ is the $A_T$ matrix from (\ref{eq:ATequation}) for the swallowtail point at $A_7$.  Also, the presence of the permutation matrices $Q_{k+2,k+3}$ is due to the discrepancy in ordering of sheets above $C_4$ and above  $A_1$, $B_1$ (which matches the ordering above $C_1$); see the arrows pointing out of $B_1$ and $A_1$ in Figure \ref{fig:ConeLabels1}.]  
To find $\partial c^4_{k,k+1}$, we will compute the $2 \times 2$ diagonal block that is the intersection of rows and columns $k$ and $k+1$.  In doing so, we use that since all of the multiplicative terms in (\ref{eq:dC}) are upper triangular (including $S$ and $T$ factors) the identity (\ref{eq:dC}) is satisfied by $2 \times 2$ diagonal blocks. [The map from upper triangular matrices to $2\times 2$ upper triangular matrices that extracts the $k,k+1$ diagonal block is a ring homomorphism.]  Keeping in mind that (i) $AC$ and $CA$ terms are $0$ on $2\times 2$ diagonal blocks (being the product of two {\it strictly} upper triangular matrices) and (ii) there is no generator of the form $b^4_{k,k+1}$ due to the crossing arc above $B_4$, we find  
\begin{align*}
\partial \left[\begin{array}{cc} 0 & \pm c^4_{k,k+1} \\ 0 & 0 \end{array} \right] & = \left[\begin{array}{cc} 1 & 0 \\ 0 & 1 \end{array} \right]- \left[\begin{array}{cc} 1 & -1 \\ 0 & 1 \end{array} \right]^{-1}\left[\begin{array}{cc} 1 & 0 \\ 0 & 1 \end{array} \right]\left[\begin{array}{cc} 1 & -a^0_{k,k+1} \\ 0 & 1 \end{array} \right]\left[\begin{array}{cc} 1 & b^{1}_{k,k+1} \\ 0 & 1 \end{array} \right]   \\
 & =\left[\begin{array}{cc} 0 & -1+ a^{0}_{k,k+1}-b^{1}_{k,k+1} \\ 0 & 0 \end{array} \right] 
\end{align*}
which allows us to cancel
\[
\pm\partial \framebox{$c^4_{k,k+1}$} = -\framebox{$b^1_{k,k+1}$} -1 + a^0_{k,k+1}
\]
and obtain
\begin{equation} \label{eq:ConeProof11}
c^4_{k,k+1} \doteq 0 \quad \mbox{and} \quad b^1_{k,k+1} \doteq a^0_{k,k+1} -1.
\end{equation}
At this point, (\ref{eq:ConeProof10}) rewritten in the form $J \cdot \partial \framebox{$C_4$} \doteq \framebox{$B_4$}+ \cdots$   allows us to cancel all remaining $c_{i,j}^4$, $(i,j) \neq (k,k+1)$ with the $b^4_{i,j}$ generators; the $AC$ and $CA$ terms correspond to the $v$ from Proposition \ref{prop:cancel} and the  remaining terms on the RHS (other than $B_4$) correspond to $w$.  In the quotient, we have
\begin{equation} \label{eq:ConeProof12}
C_4 \doteq 0 \quad \mbox{and} \quad I+B_4 \doteq T_4^{-1}(\Theta_{k+3} \Xi_{k+3})^{-1}S_1(Q_{k+2,k+3}(\Theta_{k+2} +B_1)Q_{k+2,k+3})
\end{equation}
where as usual we have deduced the second equality from $\partial C_4 \doteq 0$.
Notice that considering the $2\times2$ diagonal block at position $(k+2,k+3)$, the second identity shows that 
\[
\left[\begin{array}{cc} 1 &  b^4_{k+2,k+3} \\ 0 & 1 \end{array} \right] = \left[\begin{array}{cc} 1 & 0 \\ 0 & 1 \end{array} \right]  \left[\begin{array}{cc} 1 & 0 \\ 0 & -\lambda^{-1} \end{array} \right] \left[\begin{array}{cc} 1 & 1 \\ 0 & -1 \end{array} \right] \left[\begin{array}{cc} 1 & 0 \\ 0 & \lambda \end{array} \right] = \left[\begin{array}{cc} 1 & \lambda \\ 0 & 1 \end{array} \right],
\]
so that
\begin{equation} \label{eq:ConeProof13}
b^4_{k+2,k+3} \doteq \lambda.
\end{equation}

We now turn to $C_3$ where we use the endpoints of $B_3$ as $v_0$ and $v_1$.  The differential, making use of (\ref{eq:ConeProof3}) to replace $(\Xi_k + B_6)^{-1} \doteq (\Xi_k)^{-1}$, satisfies
\begin{align*}
J\cdot \partial C_3 & \doteq (A_6)_{T_3} C_3 + C_3 ( Q_{k,k+1}Q_{k+2,k+3} A_1 Q_{k+2,k+3}Q_{k,k+1}) \\
 & \quad + (1+B_3) - T_3^{-1}(\Xi_k)^{-1}S_4(Q_{k,k+1}(I+B_4)Q_{k,k+1}). 
\end{align*}
Minding that $b^3_{k+2,k+3}$ does not exist, we compute using (\ref{eq:ConeProof13})
\begin{align*}
\partial \left[\begin{array}{cc} 0 & \pm c^3_{k+2,k+3} \\ 0 & 0 \end{array} \right] & \doteq \left[\begin{array}{cc} 1 & 0 \\ 0 & 1 \end{array} \right]- \left[\begin{array}{cc} 1 & -1 \\ 0 & 1 \end{array} \right] \left[\begin{array}{cc} 1 & 0 \\ 0 & 1 \end{array} \right]\left[\begin{array}{cc} 1 & -a^0_{k,k+1}\lambda \\ 0 & 1 \end{array} \right]\left[\begin{array}{cc} 1 & b^4_{k+2,k+3} \\ 0 & 1 \end{array} \right]  \\
 & \doteq \left[\begin{array}{cc} 0 & 1+ a^0_{k,k+1}\lambda-\lambda \\ 0 & 0 \end{array} \right] 
\end{align*}
so that we can cancel
\[
\pm \partial \framebox{$c^3_{k+2,k+3}$} \doteq 1+ \framebox{$a^0_{k,k+1}$}\lambda - \lambda.
\]
We record
\begin{align} \label{eq:ConeProof14}
c^3_{k+2,k+3} \doteq 0 \quad \mbox{and} \quad a^0_{k,k+1} \doteq 1-\lambda^{-1}.
\end{align}
This allows us to eliminate all of the remaining $c^3_{i,j}$ and $b^3_{i,j}$ generators by cancelling $J \cdot \partial \framebox{$C_3$} = \framebox{$B_3$}+ \cdots$.  We then have identities
\begin{equation}  \label{eq:ConeProof14.5}
C_3 \doteq 0 \quad \mbox{and} \quad I+B_3 \doteq T_3^{-1}(\Xi_k)^{-1}S_4(Q_{k,k+1}(I+B_4)Q_{k,k+1}).
\end{equation}
Notice the latter identity (considering the $2\times2$ block at positions $k$ and $k+1$ on the diagonal) provides that
\[
\left[\begin{array}{cc} 1 &  b^3_{k,k+1} \\ 0 & 1 \end{array} \right] \doteq \left[\begin{array}{cc} 1 & 0 \\ 0 & 1 \end{array} \right]^{-1}  \left[\begin{array}{cc} -1 & 0 \\ 0 & 1 \end{array} \right]^{-1}  \left[\begin{array}{cc} -1 & 1 \\ 0 & 1 \end{array} \right] \left[\begin{array}{cc} 1 & 0 \\ 0 & 1 \end{array} \right] = \left[\begin{array}{cc} 1 & -1 \\ 0 & 1 \end{array} \right],
\]
i.e. 
\begin{equation} \label{eq:ConeProof15} 
	b^3_{k,k+1} \doteq -1.
\end{equation}

Next up is $C_2$.  Using the end points of $B_2$ for $v_0$ and $v_1$, we have
\begin{align} \label{eq:ConeProof16}
J\cdot \partial C_2 & \doteq (A_5)_{T_2} C_2 + C_2 ( Q_{k,k+1} A_1 Q_{k,k+1}) \\
 &  \quad + (\Theta_{k}^{-1}+B_2) - T_2^{-1}(\Theta_k\Xi_{k+3})^{-1} S_3 (Q_{k+2,k+3}(I+B_3) Q_{k+2,k+3}).
\end{align}
Computing the $2\times 2$ block at positions $k$ and $k+1$ on the diagonal, and making use of (\ref{eq:ConeProof14}), (\ref{eq:ConeProof15}), and that $b^2_{k,k+1}$ does not exist due to the crossing arc,
shows
\begin{align*}
	\partial \left[\begin{array}{cc} 0 & \pm c^2_{k,k+1} \\ 0 & 0 \end{array} \right] & \doteq \left[\begin{array}{cc} \lambda^{-1} & 0 \\ 0 & 1 \end{array} \right]- \left[\begin{array}{cc} 1 & -1 \\ 0 & 1 \end{array} \right]^{-1} \left[\begin{array}{cc} \lambda & 0 \\ 0 & 1 \end{array} \right]^{-1} \left[\begin{array}{cc} 1 & -a^0_{k,k+1}\lambda \\ 0 & 1 \end{array} \right]\left[\begin{array}{cc} 1 & b^3_{k,k+1} \\ 0 & 1 \end{array} \right] =0.
\end{align*}
As $\partial c^2_{k,k+1} = 0$, we can adjust the triangular ordering of generators (inherited on the current quotient dg-algebra from (\ref{eq:ConeOrdering}) above) by reordering $c:= c^2_{k,k+1}$ to be the smallest element;  note that the triangularity condition from Definition \ref{def:DGAs} is preserved by this reordering.  We can use the form of (\ref{eq:ConeProof16}) as $J\cdot \partial \framebox{$C_2$} = \framebox{$B_2$}+ \cdots$ to cancel all $c^2_{i,j}$ except for $c^2_{k,k+1}$ with $b^2_{i,j}$; here, those terms from $(A_5)_{T_2} C_2 + C_2 ( Q_{k,k+1} A_1 Q_{k,k+1})$ involving $c^2_{k,k+1}$ (resp. $c^2_{i,j}$ with $(i,j) \neq (k,k+1)$) are considered as part of the $w$ term (resp. $v$ term) in the statement of Proposition \ref{prop:cancel}.
In the resulting quotient we have
\begin{equation}  \label{eq:ConeProof17}
C_2 \doteq  c E_{k,k+1}    
\end{equation}
which also implies $\partial C_2 \doteq 0$.  Combined with (\ref{eq:ConeProof16}) this shows
\begin{equation}  \label{eq:ConeProof18}
\Theta^{-1}_k + B_2 \doteq F -G
\end{equation}
with
\begin{equation}  \label{eq:ConeProof19}
F = T_2^{-1}(\Theta_k\Xi_{k+3})^{-1} S_3 (Q_{k+2,k+3}(I+B_3) Q_{k+2,k+3})
\end{equation}
and
\begin{equation}  \label{eq:ConeProof20}
G = (A_5)_{T_2} (cE_{k,k+1}) + (cE_{k,k+1}) ( Q_{k,k+1} A_1 Q_{k,k+1}).
\end{equation}

Next, we turn to $B_1$ and $A_1$.  Note that because of the crossing arcs, $a^1_{k,k+1}$, $a^1_{k+2,k+3}$, and $b^1_{k+2,k+3}$ do not exist.  Moreover,  $b^1_{k,k+1}$ has already been removed from the generating set and, using (\ref{eq:ConeProof11}) and (\ref{eq:ConeProof14}), satisfies $b^1_{k,k+1}\doteq -\lambda^{-1}$.  Thus, the remaining $b^1_{i,j}$ and $a^1_{i,j}$ are in bijection and in the same positions of $B_1$ and $A_1$, so that they can be cancelled via 
\[
J\cdot \partial \framebox{$B_1$} = (A_0)_{T_1}(\Theta_{k+2}+B_1) - (\Theta_{k+2}+B_1)\framebox{$A_1$}.
\]
In the quotient, we have
\begin{equation}  \label{eq:ConeProof21}
B_1 \doteq -\lambda^{-1}E_{k,k+1} \quad \quad \mbox{and} \quad \quad A_1 \doteq (\Theta_{k+2}-\lambda^{-1}E_{k,k+1})^{-1} (A_0)_{T_1} (\Theta_{k+2}-\lambda^{-1}E_{k,k+1}).
\end{equation}

At this point, we take a closer stock of the $F$ and $G$ terms in (\ref{eq:ConeProof18})-(\ref{eq:ConeProof20}).  As a preliminary, using that we now have $B_1 \doteq -\lambda^{-1}E_{k,k+1}$ via (\ref{eq:ConeProof21}) and $a^0_{k,k+1} = 1- \lambda^{-1}$ via (\ref{eq:ConeProof14}) we can update (\ref{eq:ConeProof12}) to 
\begin{align*}
I+B_4 & \doteq T_4^{-1}(\Theta_{k+3} \Xi_{k+3})^{-1}S_1(Q_{k+2,k+3}(\Theta_{k+2}+B_1) Q_{k+2,k+3}) \\
& \doteq (I+E_{k,k+1})\Theta^{-1}_{k+3}\Xi_{k+3}(\Xi_{k+3} +E_{k+2,k+3} - (1-\lambda^{-1})E_{k,k+1} - \sum_{i<k} a^0_{i,k+1} E_{i,k+1}) \\
 & \quad \times(\Theta_{k+3} - \lambda^{-1}E_{k,k+1}) \\
 & = I + \lambda E_{k+2,k+3} - \sum_{i<k} a^0_{i,k+1}E_{i,k+1}
\end{align*}
where the last identity is obtained by carefully multiplying the product out and simplifying.
Using this we can update (\ref{eq:ConeProof14.5}) to become
\begin{align*}
	I+B_3 & \doteq T_3^{-1}\Xi_k^{-1}S_4(Q_{k,k+1}(I+B_4)Q_{k,k+1}) \\
	 &   \doteq (I-E_{k+2,k+3}) \Xi_k(\Xi_k + E_{k,k+1} -(1-\lambda^{-1})\lambda E_{k+2,k+3} + \sum_{k+1<j} a^0_{k,j} E_{k+2,j+2}) \\
	 & \quad \times Q_{k,k+1}(I+\lambda E_{k+2,k+3} - \sum_{i<k} a^0_{i,k+1} E_{i,k+1}) Q_{k,k+1} \\
	 & = I - E_{k,k+1} + \sum_{k+1<j} a^0_{k,j} E_{k+2,j+2} - \sum_{i<k}a^0_{i,k+1} E_{i,k}
\end{align*}
where the last identity is again a direct calculation that requires care.  Now, we are ready to improve (\ref{eq:ConeProof19}) to
\begin{align} \label{eq:ConeProof22}
	F &= T_2^{-1}(\Theta_k\Xi_{k+3})^{-1} S_3 (Q_{k+2,k+3}(I+B_3) Q_{k+2,k+3}) \\  \notag
	& \doteq (I+E_{k,k+1})(\Theta^{-1}_k \Xi_{k+3})\left(\Xi_{k+3}+ E_{k+2,k+3} -(1-\lambda^{-1})\lambda E_{k,k+1} + \sum_{i<k} (a^0_{i,k+1} \lambda) E_{i,k+1}\right) \\ \notag
	& \quad \times Q_{k+2,k+3}(I - E_{k,k+1} + \sum_{k+1<j} a^0_{k,j} E_{k+2,j+2} - \sum_{i<k}a^0_{i,k+1} E_{i,k}) Q_{k+2,k+3} \\ \notag
	& = \Theta^{-1}_k + E_{k+2,k+3} + \sum_{i<k} (a^0_{i,k+1} \lambda) E_{i,k+1} + \sum_{k+1<j} a^0_{k,j} E_{k+3,j+2} \\
 & \quad + \sum_{k+1 <j} a^0_{k,j} E_{k+2,j+2} - \sum_{i<k} a^0_{i,k+1} E_{i,k}. \notag
\end{align}

For the $G$ term we can get by with the following:

\medskip

\noindent {\bf Claim.}  
\begin{itemize}
	\item[(i)]  All entries of the matrix $G$ from (\ref{eq:ConeProof20}) belong to the ideal generated by $c:= c^2_{k,k+1}$. 
	\item[(ii)]  The $(k,k+2)$-entry of $G$ is $c^2_{k,k+1}$.
\end{itemize}

\medskip
Notice that (i) is immediate from (\ref{eq:ConeProof20}).  To verify (ii), for computing the $(k,k+2)$-entry of (\ref{eq:ConeProof20}), the first term, $(A_5)_{T_2}(cE_{k,k+1})$ can be ignored since it can only have non-zero entries in the $(k+1)$ column.  The $(k,k+2)$-entry of the second term, $(cE_{k,k+1})(Q_{k,k+1} A_1 Q_{k,k+1})$ is 
\begin{align*} &  c \cdot \big( \mbox{$(k+1,k+2)$-entry of } Q_{k,k+1} A_1 Q_{k,k+1} \big) \\
=	&   c \cdot \big( \mbox{$(k,k+2)$-entry of } A_1 \big).
\end{align*}
Then, using (\ref{eq:ConeProof21}) we find that
\begin{align*}
A_1 & \doteq (\Theta_{k+2}-\lambda^{-1}E_{k,k+1})^{-1} (A_0)_{T_1} (\Theta_{k+2}-\lambda^{-1}E_{k,k+1}) \\
 & = (\Theta_{k+2}^{-1} + \lambda^{-1}E_{k,k+1})(I-E_{k+2,k+3}) ( (\widetilde{A_0})_{k+1,k+2} + E_{k+1,k+2})(I+E_{k+2,k+3})( \Theta_{k+2} - \lambda^{-1} E_{k,k+1}).
\end{align*}
Upon expanding via the distributive property the only term in position $(k,k+2)$ is from
\[
(\lambda^{-1} E_{k,k+1})(I)(E_{k+1,k+2})(I)(\Theta_{k+2}),
\]
and this term is $\lambda^{-1}\lambda =1$ as required.

\medskip

With the claim established we are now ready for:

\medskip

\noindent {\bf Step 3.}  Cancel 
\begin{itemize}
	\item $c^1_{k+1,k+2}$ with $c^2_{k,k+1}$.
	\item $c^1_{i,k+2}$ with $c^1_{i,k+1}$ for $i<k+1$.
	\item $c^1_{k+1,j}$ with $c^1_{k+2,j}$ for $k+2<j$.
	\item All remaining $C_1$ generators with $B_0$.
\end{itemize}

\medskip

Taking the endpoints of $B_0$ in the extended polygon for $C_1$ as $v_0$ and $v_1$, the differential of $C_1$ satisfies
\begin{align*}
J \cdot \partial C_1 & = (\widehat{A}_5)_{k+1,k+2} C_1 + C_1 (\widehat{A}_0)_{k+1,k+2} + (\Theta^{-1}_{k+1,k+2} \Xi_k + (\widetilde{B}_0)_{k+1,k+2}) \\ & \quad \quad - S_2(Q_{k,k+1}(\Theta^{-1}_k + B_2) Q_{k,k+1})(\Theta_{k+2} + B_1)^{-1} T_1^{-1} \\
& \doteq (\widehat{A}_5)_{k+1,k+2} C_1 + C_1 (\widehat{A}_0)_{k+1,k+2} + (\Theta^{-1}_{k+1,k+2} \Xi_k + (\widetilde{B}_0)_{k+1,k+2}) \\
& \quad \quad - S_2(Q_{k,k+1}(F-G)Q_{k,k+1} )(\Theta_{k+2} -\lambda^{-1} E_{k,k+1})^{-1} T_1^{-1}. 
\end{align*}
First, we see that the $2\times 2$ diagonal block at positions $k+1$ and $k+2$ is
\[
0 + 0 + \left[\begin{array}{cc} \lambda^{-1} & 0 \\ 0 & \lambda^{-1} \end{array} \right] - \left[\begin{array}{cc} 1 & 0 \\ 0 & 1 \end{array} \right]\left[\begin{array}{cc} \lambda^{-1} & -c^2_{k,k+1} \\ 0 & 1 \end{array} \right] \left[\begin{array}{cc} 1 & 0 \\ 0 & \lambda^{-1} \end{array} \right]\left[\begin{array}{cc} 1 & 0 \\ 0 & 1 \end{array} \right]= \left[\begin{array}{cc} 0 & c^2_{k,k+1} \lambda^{-1} \\ 0 & 0 \end{array} \right].
\]
[We used (\ref{eq:ConeProof7}) to evaluate this block of $S_2$; used (\ref{eq:ConeProof22}) and (ii) of the Claim for that of $(Q_{k,k+1}(F-G)Q_{k,k+1})$; computed $(\Theta_{k+2} -\lambda^{-1} E_{k,k+1})^{-1} = \Theta_{k+2}^{-1} + \lambda^{-1} E_{k,k+1}$; and computed $T_1^{-1} =I-E_{k+2,k+3}$ from (\ref{eq:ConeProof5}).]  This allows us to cancel
\[
\pm \partial \framebox{$c^1_{k+1,k+2}$} \doteq \framebox{$c^2_{k,k+1}$}\lambda^{-1}
\]
resulting in
\[
c^1_{k+1,k+2} \doteq c^2_{k,k+1} \doteq 0
\]
and (using (i) of the Claim)
\[
G \doteq 0.
\]
 Combining this with (\ref{eq:ConeProof22}), the last term from $J\cdot \partial C_1$ then becomes
\begin{align*}
 & S_2(Q_{k,k+1}(F-G)Q_{k,k+1} )(\Theta_{k+2} -\lambda^{-1} E_{k,k+1})^{-1} T_1^{-1} \\
\doteq \quad & (\Xi_k + E_{k,k+1} - (\lambda^{-1} a^0_{k,k+1} \lambda) E_{k+2,k+3} + \sum_{k+1<j} (-\lambda^{-1} a^0_{k,j}) E_{k+2,j+2}) \\
& \quad \times \bigg[\Theta^{-1}_{k+1} + E_{k+2,k+3} + \sum_{i<k} (a^0_{i,k+1} \lambda) E_{i,k} + \sum_{k+1<j} a^0_{k,j} E_{k+3,j+2} \\ & \quad \quad + \sum_{k+1<j} a^0_{k,j} E_{k+2,j+2} - \sum_{i<k} a^0_{i,k+1} E_{i,k+1}\bigg] \\
 & \quad \times (\Theta^{-1}_{k+2} + \lambda^{-1}E_{k,k+1})(I - E_{k+2,k+3}) \\
= \quad &\Xi_k \Theta^{-1}_{k+1,k+2} + \sum_{i<k} (a^0_{i,k+1} \lambda) E_{i,k} + \sum_{k+1<j} a^0_{k,j} E_{k+3,j+2}.
\end{align*}
[The last equality follows from carefully expanding via the distributive property; many terms cancel.]
Overall, $J \cdot \partial C_1$  becomes
\begin{align*}
J \cdot \partial C_1 & = (\widehat{A}_5)_{k+1,k+2} C_1 + C_1 (\widehat{A}_0)_{k+1,k+2} + (\Theta^{-1}_{k+1,k+2} \Xi_k + (\widetilde{B}_0)_{k+1,k+2}) \\ & \quad \quad - (\Xi_k \Theta^{-1}_{k+1,k+2} + \sum_{i<k} (a^0_{i,k+1} \lambda) E_{i,k} + \sum_{k+1<j} a^0_{k,j} E_{k+3,j+2}).
\end{align*}
Now,  $(\Theta^{-1}_{k+1,k+2} \Xi_k + (\widetilde{B}_0)_{k+1,k+2}) - (\Xi_k \Theta^{-1}_{k+1,k+2} + \sum_{i<k} (a^0_{i,k+1} \lambda) E_{i,k} + \sum_{k+1<j} a^0_{k,j} E_{k+3,j+2})$ has all entries $0$ in both rows $k+1$ and $k+2$ and also columns $k+1$ and $k+2$.  Thus, for all $i<k+1$ (proceeding in order of decreasing $i$) we can cancel
\[
\pm \partial \framebox{$c^1_{i,k+2}$} \doteq \framebox{$c^1_{i,k+1}$} + \sum_{i<m <k+1} a^5_{i,m} c^1_{m,k+2}
\]
where the $\sum_{i<m <k+1} a^5_{i,m} c^1_{m,k+2}$ terms are taken to be $v$ in Proposition \ref{prop:cancel}.  Similarly, for all $k+2 < j$ we cancel
\[
\pm \partial \framebox{$c^1_{k+1,j}$} \doteq \framebox{$c^1_{k+2,j}$} + \sum_{k+2<m<j} c^1_{k+1,m} a^0_{l-2,j-1}.
\]
In the quotient, we know have that for all $i<k+1$ and $k+2<j$
\[
c^1_{k+1,k+2} \doteq c^1_{i,k+2} \doteq c^1_{i,k+1} \doteq c^1_{k+1,j} \doteq c^1_{k+2,j} \doteq 0,
\]
and at this point the remaining $C_1$ generators are in bijection with the $B_0$ generators so that we can use $J\cdot \framebox{$C_1$} = \framebox{$(\widetilde{B}_0)_{k+1,k+2}$}+ \cdots$ to cancel $c^1_{i,j}$ with $b^0_{\tau(i), \tau(j)}$ for all $i,j$ with $\{i,j\} \cap \{k+1,k+2\} = \emptyset$ where $\tau(m) = \left\{\begin{array}{cr} m, & m < k+1 \\ m-2, & m>k+2 \end{array}\right..$  This results in
$C_1 \doteq 0$ and $\partial C_1 \doteq 0$, allowing us to deduce
\[
(\widetilde{B}_0)_{k+1,k+2} \doteq \sum_{i<k} (a^0_{i,k+1} \lambda) E_{i,k} + \sum_{k+1<j} a^0_{k,j} E_{k+3,j+2}. 
\] 
Removing rows and columns $k+1$ and $k+2$ from this equation shows
\begin{equation}  \label{eq:ConeProof23}
B_0 \doteq \sum_{i<k} (a^0_{i,k+1} \lambda) E_{i,k} + \sum_{k+1<j} a^0_{k,j} E_{k+1,j}.
\end{equation}

At this point, we have successfully cancelled all of the generators from the cells in the closure of the region $X$ (see Figure \ref{fig:ConeLabels1})  except for the $a^0_{i,j}$ with $(i,j) \neq (k,k+1)$.  As explained in the first paragraph of this proof, the argument is now completed by observing that (\ref{eq:ConeProof14}) together with (\ref{eq:ConeProof1}), (\ref{eq:ConeProof3}), (\ref{eq:ConeProof4}), and (\ref{eq:ConeProof23}) provide the required identities
\[
a^0_{k,k+1} \doteq 1-\lambda^{-1}
\]
and
\begin{align*}
& (\Delta_7 + B_7)(\Delta_6+B_6)(\Delta_5+B_5)(\Delta_0+ B_0) \\
\doteq \quad & (\Xi_{k+1}\Theta_{k+1})(\Xi_k)(\Xi_{k+1} \Theta_k)( \Xi_k + B_0)   \\
 \doteq \quad &  \Theta_{k,k+1} + \sum_{i<k} (a^0_{i,k+1} \lambda) E_{i,k} + \sum_{k+1<j} (\lambda a^0_{k,j}) E_{k+1,j}.
\end{align*}

\end{proof}

\end{document}